\documentclass[a4paper,11pt,reqno]{amsart}

\usepackage{amssymb}
\usepackage{latexsym}
\usepackage{amsmath}
\usepackage{mathrsfs}
\usepackage{euscript}
\usepackage{amsthm}
\usepackage{upgreek}
\usepackage{cite}
\usepackage{amscd}
\usepackage{xcolor}
\usepackage{tikz}
\usepackage{calc}

\usepackage{diagbox}

\usepackage{comment}

\usepackage{color}
\definecolor{ceruleanblue}{rgb}
{0.16, 0.32, 0.75}

\newcommand{\scal}[2]{\langle #1,#2\rangle}

\newcommand{\rr}[1]{\mathbf R^{#1}}
\newcommand{\zz}[1]{\mathbf Z^{#1}}

\newcommand{\cc}[1]{\mathbf C^{#1}}
\newcommand{\nn}[1]{\mathbf N^{#1}}

\newcommand{\nm}[2]{\Vert #1\Vert _{#2}}
\newcommand{\NM}[2]{\left \Vert #1\right \Vert _{#2}}
\newcommand{\nmm}[1]{\Vert #1\Vert }

\newcommand{\op}{\operatorname{Op}}

\newcommand{\sets}[2]{\{ \, #1\, ;\, #2\, \} }

\newcommand{\fy}{\varphi}

\newcommand{\cdo}{\, \cdot \, }

\newcommand{\eabs}[1]{\langle #1\rangle}

\newcommand{\tp}{\operatorname{Tp}}
\newcommand{\vrum}{\vspace{0.1cm}}

\newcommand{\wpr}{{\text{\footnotesize $\#$}}}

\newcommand{\GL}{\mathbf{M}}

\newcommand{\sfW}{\mathsf{W}}

\newcommand{\maclL}{\mathcal L}
\newcommand{\maclM}{\mathcal M}

\newcommand{\maclR}{\mathcal R}
\newcommand{\maclS}{\mathcal S}

\newcommand{\mascB}{\mathscr B}

\newcommand{\mascF}{\mathscr F}

\newcommand{\mascP}{\mathscr P}

\newcommand{\mascS }{\mathscr S}

\newcommand{\fka}{\mathfrak a}
\newcommand{\fkb}{\mathfrak b}
\newcommand{\fkc}{\mathfrak c}

\newcommand{\splM}{\EuScript M}

\numberwithin{equation}{section}

\newtheorem{thm}{Theorem}
\numberwithin{thm}{section}

\newcommand{\rubrik}{}
\newtheorem{prop}[thm]{Proposition}
\newtheorem{cor}[thm]{Corollary}
\newtheorem{lemma}[thm]{Lemma}

\theoremstyle{definition}

\newtheorem{defn}[thm]{Definition}
\newtheorem{example}[thm]{Example}

\theoremstyle{remark}

\newtheorem{rem}[thm]{Remark}              

\definecolor{darkred}{rgb}{0.8,0,0}

\author{Joachim Toft}

\address{Department of Mathematics,
Linn{\ae}us University,
V{\"a}xj{\"o}, Sweden}

\email{joachim.toft@lnu.se}

\author{Christine Pfeuffer}

\address{Department of Mathematics, Martin-Luther-Universit\"at Halle-Wittenberg, Halle,  Germany}

\email{christine.pfeuffer@mathematik.uni-halle.de}

\author{Nenad Teofanov}

\address{Department of Mathematics and Informatics,
University of Novi Sad, Novi Sad, Serbia}

\email{nenad.teofanov@dmi.uns.ac.rs}

\title[Modulation 
spaces, and 
Fourier type operators]
{Norm estimates for a broad class of 
modulation spaces, and continuity of 
\\
Fourier type operators}



\begin{document}

\begin{abstract}
We consider a broad class of
modulation spaces $M(\omega ,\mascB )$,
parameterized with weight function
$\omega$ and a normal
quasi-Banach function space $\mascB$
of order $r_0 \in (0,1]$.
Then we prove that $f\in
M(\omega ,\mascB )$,
iff $V_\phi f$ belongs to
the Wiener amalgam space
$\sfW ^r(\omega ,\mascB )$, and
$$
\nm f {M(\omega, \mascB)} 
\asymp
\nm {V _\phi f \cdo \omega}{\mascB}
\asymp
\nm {V _\phi f} 
{\sfW ^r(\omega, \mascB)},
\quad
r\in [r_0,\infty ]
.
$$

\par

We use the results to
extend and improve continuity and lifting
properties
for pseudo-differential and Toeplitz
operators with symbols in weighted
$M^{\infty,r_0}$-spaces, with $r_0\le 1$,
when acting on $M(\omega ,\mascB )$-spaces.
\end{abstract}

\maketitle

\tableofcontents

\par

\section{Introduction}\label{sec0}

\par

The aim of the first part of the paper is
to derive various types of quasi-norm
equivalences for
a broad family of quasi-Banach
modulation spaces. Any such topological
vector space, $M(\omega ,\mascB)$,
is parameterized by a suitable weight function
$\omega$ and a quasi-Banach function space
$\mascB$,
and is adapted and feasible for various
kinds of Fourier techniques.
Our family of equivalent norms
is significantly larger, compared to what
has been deduced so far in the context of
modulation
space and coorbit space theories.
(See e.{\,}g. 
\cite{BenOko,CorRod,Fei3,FeiGro1,FeiGro2,
FG4,GaSa,Gro2,Rau1,Rau2}.)
In contrast
to the usual approaches in these theories,
our methods do not rely on (Gabor) frame theory
or Gabor expansions. In fact,
the achieved norm equivalences should be
considered as a \emph{complement} to existing Gabor
frame approach for modulation space theory,
where we supply the theory with quasi-norm
estimates as alternatives to estimates
obtained from Gabor expansions. Especially,
our methods eliminate the need to find
suitable dual window functions to a given
Gabor frame.

\par

The aim of the second part of the paper
is to find convenient and general
continuity properties for pseudo-differential
operators with symbols in suitable modulation
spaces when acting on broad families of 
modulation
spaces. These results are significantly more
general than the related results
available in the literature.
(See e.{\,}g. \cite{AbCaTo,AbCoTo,BenOko,
CaTo,CorRod,Gro2,PilTeo2,Sjo,Tac,Teofanov3,
Toft04,Toft07,Toft18,ToBo08,TofUst}.)
The norm equivalences from the first
part serve as cornerstones for achieving
our results, and
it seems that earlier methods are
insufficient for reaching these results.

\par

For an overview of some parts of
our progress, see Table 
\ref{Table:ProgressPSDO} below.

\medspace

Our results and ideas might be
applicable in various parts within
analysis, or in a broader perspective,
to various fields in
science and technology, in the
seek of finding suitable 
(topological) spaces. In several
such situations it is important
that these spaces are adapted to
Fourier analysis, and in particular
to the Fourier transform and 
other related transformations. 

\par

The Fourier transform provides a picture
that only 
displays non-localized information on 
momentum or frequency, while in 
different applications, there is a need
for global
representations which simultaneously 
depict the behavior in the phase space and
time-frequency shift space.

\par


A common choice of 
operators fulfilling such requested
properties is the local Fourier transform,
usually called the short-time Fourier 
transform or voice transform in
time-frequency analysis
(see e.{\,}g.
\cite {Fei1,FeiGro1,Gro2}). In
quantum physics, this transform is
often called the
coherent state transform (cf.
\cite {Fei1,FeiGro1,Gro2,LieSol}). 
It possess several convenient properties,
like continuity
on suitable spaces. It can
also be used for characterizing 
topological vector spaces like the 
Schwartz space, Gelfand-Shilov spaces, 
and their distribution spaces
(see e.{\,}g. \cite{ChuChuKim, Eij, Gro2,
GroZim2001, GroZim, Toft17}).

\par

A convenient way to quantify the phase space
or time-frequency content of a
distribution $f$ on $\rr d$
is to estimate its short-time
Fourier transform $V_\phi f$
in suitable ways. Here $\phi$
is a fixed (window) function which
should belong to a convenient
function space, like
$\mascS (\rr d)\setminus 0$.
(See \cite {Ho1}
or Sections \ref{sec1} and
\ref{sec6}--\ref{sec8}
for notations.)
For this reason, Feichtinger
introduced in \cite{Fei1981, Fei1} the (classical)
modulation spaces,
$M^{p,q}_{(\omega )}(\rr d)$,
$p,q\in [1,\infty ]$,
by imposing a mixed $L^{p,q}_{(\omega )}$
norm estimate on the short-time Fourier
transforms of the involved functions
and distributions. Here $\omega$ should belong
to $\mascP (\rr {2d})$, the set of
weight functions on $\rr {2d}$,
which are moderate with polynomially
bounded functions, cf. Subsection
\ref{subsec1.1}.

\par

Thus the Lebesgue parameters
$p$ and $q$ to some extent quantify the degrees
of asymptotic decays and singularities,
respectively, of the distributions in
$M^{p,q}_{(\omega )}(\rr d)$, and similar facts
hold true for the weight function $\omega$.
For example in \cite{Fei1981, Toft04} it is
observed that
\begin{equation}
\label{Eq:InterSecUnionModSp}
\bigcap _{\omega \in \mascP}M^{p,q}_{(\omega )}(\rr d)
=
\mascS (\rr d)
\quad \text{and}\quad
\bigcup _{\omega \in \mascP}
M^{p,q}_{(\omega )}(\rr d)
=
\mascS '(\rr d).
\end{equation}

\par

In \cite{Fei1981, Fei1}, several fundamental
properties for
$M^{p,q}_{(\omega )}$-spaces were established.
In particular, it is here proved that 
$M^{p,q}_{(\omega )}(\rr d)$ increases
with $p$ and $q$, decreases with $\omega$,
and is independent of the choice of
window function
$\phi \in M^{1}_{(v)}(\rr d)\setminus 0$.

\par

After \cite{Fei1981,Fei1}, the theory of 
modulation 
spaces was further developed in different
ways. In \cite{FeiGro0, FeiGro1, FeiGro2},
Feichtinger and Gr{\"o}chenig
developed the coorbit space theory
in the framework of Banach spaces,
and in \cite{FG4},
they showed that these
spaces admit reconstructible sequence 
space representations using Gabor
frames.
In \cite{GaSa}, Galperin and Samarah showed
that this result
still holds
true in the quasi-Banach situation, after
the assumption $p,q\in [1,\infty ]$
is relaxed into
$p,q\in (0,\infty ]$. A similar extension 
of the whole coorbit space theory in
\cite{FeiGro0, FeiGro1, FeiGro2, Gro1991}, to 
allow quasi-Banach spaces
was thereafter performed by Rauhut in
\cite{Rau2}. Some recent contributions
to this topic is performed in \cite{Voi}
by Voigtlaender.

\par

Some extensions to include more general
weight functions, leading to modulation
spaces in the framework of
ultra-distribution, can also be 
found, see e.{\,}g.
\cite{Gro2,PilTeo1,PilTeo2,Teofanov2,
Toft10,Toft2022}.
For example, let
$ \mascP (\rr {2d})$ in
\eqref{Eq:InterSecUnionModSp}
be replaced by
$\mascP _E(\rr {2d})$,
the set of \emph{all} moderate
weights (without any assumptions
on polynomial bounds).
Then it follows that 
\eqref{Eq:InterSecUnionModSp}
holds true with the
(Fourier invariant) Gelfand-Shilov space
$\Sigma _1$ and its distribution space
$\Sigma _1'$, in place of
$\mascS$ and $\mascS '$, respectively
(see e.{\,}g. \cite{Toft10}).
We remark that $\Sigma _1(\rr d)$ is
significantly smaller (but dense) in
$\mascS (\rr d)$, while its distribution
space $\Sigma _1'(\rr d)$ is strictly 
larger than $\mascS '(\rr d)$,
cf. Subsection \ref{subsec1.2}.

\par

The concept of modulation spaces was also 
extended in \cite{Fei6}, by allowing more 
general bound conditions on the
short-time Fourier transform. For example,
here Feichtinger considers modulation
spaces, with our notations denoted
by $M(\omega ,\mascB _0)$, and
which are parameterized by
the solid Banach function space
$\mascB _0$, and the weight function
$\omega$.

\par

In several situations, modulation spaces are 
convenient to use because it is easy to find 
different characterizations for them, especially 
norm characterizations. For example,
for $p,q,r\in (0,\infty ]$,
the mixed (quasi-)norm space
$L^{p,q}_{(\omega )}(\rr {2d})$ and
the Wiener space
$\sfW ^r (L^{p,q}_{(\omega )}(\rr {2d}))$
are the sets of all measurable
functions $F$ on $\rr {2d}$ such that
$\nm F{L^{p,q}_{(\omega )}}<\infty$
and
$\nm F{\sfW ^r (L^{p,q}_{(\omega )})}
<\infty$, respectively. Here
\begin{align*}
\nm F{L^{p,q}_{(\omega )}}
&=
\left (
\int _{\rr d}
\left (
\int _{\rr d}
|F(x,\xi )\omega (x,\xi )|^p\, dx
\right )^{\frac qp}
\, d\xi 
\right )^{\frac 1q}
\intertext{and}
\nm F{\sfW ^r (L^{p,q}_{(\omega )})}
&=
\left (
\sum _{\iota \in  \zz d}
\left (
\sum _{j\in \zz d}
\big (
\nm {F\cdo \omega }{L^r(Q(j,\iota ))}
\big )^p
\right )^{\frac qp}
\, d\xi 
\right )^{\frac 1q},
\\
Q(j,\iota ) &= (j,\iota )+[0,1]^{2d},
\end{align*}
when $p,q<\infty$.

\par

By a straight-forward application of
H{\"o}lder's inequality, it follows
that $\sfW ^r (L^{p,q}_{(\omega )}(\rr 
{2d}))$ decreases with $r$,
its (quasi-)norm increases with $r$,
and
$$
\sfW ^{r_2}
(L^{p,q}_{(\omega )}(\rr {2d}))
\subseteq
L^{p,q}_{(\omega )}(\rr {2d})
\subseteq
\sfW ^{r_1}
(L^{p,q}_{(\omega )}(\rr {2d})),
\quad
r_1\le p,q\le r_2.
$$

\par

Taking into account that the latter
embeddings are strict,
one might be surprised that for any
$\phi \in \Sigma _1(\rr d)\setminus 0$,
indeed
\begin{equation}
\label{Eq:IntroEquiv1}
\begin{aligned}
f\in M^{p,q}_{(\omega )}(\rr d)
\quad &\Leftrightarrow \quad
V_{\phi}f\cdot \omega
\in
L^{p,q}(\rr {2d})
\\[1ex]
&\Leftrightarrow \quad
V_{\phi}f\cdot \omega
\in
\sfW ^\infty
(L^{p,q}(\rr {2d})),
\end{aligned}
\end{equation}
and that the (quasi-)norm equivalence
\begin{equation}
\label{Eq:IntroEquiv2}
\nm f{M^{p,q}_{(\omega )}}
\asymp
\nm {V_\phi f\cdot \omega}
{L^{p,q}}
\asymp
\nm {V_\phi f\cdot \omega}
{\sfW ^\infty (L^{p,q})},
\qquad
f\in M^{p,q}_{(\omega )}(\rr d)
\end{equation}
holds true.
(See e.{\,}g. \cite{GaSa,Toft2022}.)

\par


\par

In the first part of the paper
we extend the norm equivalences
\eqref{Eq:IntroEquiv1} and
\eqref{Eq:IntroEquiv2} to a broad family of
(quasi-Banach) modulation spaces. More precisely,
we replace the mixed norm spaces $L^{p,q}$
above with a solid translation invariant 
quasi-Banach function space (QBF space)
$\mascB$ with respect to $r_0\in (0,1]$
and weight $v_0\in \mascP _E(\rr {2d})$.
%
%
%
Additionally we assume that
$\mascB$ is
\emph{normal}, i.{\,}e.,
for some solid Banach function space
$\mascB _0$, one has
$$
\mascB
=
\sets {F\ \text{measurable}}
{|F|^{r_0}\in \mascB _0}
\quad \text{and}\quad
\nm F{\mascB}
=
\nm {\, |F|^{r_0}\, }{\mascB _0}^{1/r_0}.
$$
Then our extension of \eqref{Eq:IntroEquiv1} and
\eqref{Eq:IntroEquiv2} is that 
for any
$\phi \in \Sigma _1(\rr d)\setminus 0$,
\emph{and any $r\in [r_0,\infty]$}
(and not only $r=\infty$ as in
\eqref{Eq:IntroEquiv1} and
\eqref{Eq:IntroEquiv2}), we have
\begin{equation}
\tag*{(\ref{Eq:IntroEquiv1})$'$}
\begin{aligned}
f\in M(\omega ,\mascB )
\quad &\Leftrightarrow \quad
V_{\phi}f\cdot \omega 
\in
\mascB
\\[1ex]
&\Leftrightarrow \quad
V_{\phi}f\cdot \omega 
\in
\sfW ^r (\mascB ),
\end{aligned}
\end{equation}
and that the
(quasi-)norm equivalence
\begin{equation}
\tag*{(\ref{Eq:IntroEquiv2})$'$}
\nm f{M(\omega ,\mascB )}
\asymp
\nm {V_\phi f\cdot \omega}
{\mascB}
\asymp
\nm {V_\phi f\cdot \omega}
{\sfW ^r (\mascB )}
\qquad
f\in M(\omega ,\mascB )
\end{equation}
holds, see Section \ref{sec3}. In fact, as in \cite{FeiGro1,GaSa,Gro2,Rau2},
we permit $\phi$ in \eqref{Eq:IntroEquiv1}$'$
and \eqref{Eq:IntroEquiv2}$'$, to belong to
suitable weighted $M^{r_0}$-spaces, which
strictly contain $\Sigma _1(\rr d)$.
Even for
classical modulation spaces,
$M^{p,q}_{(\omega )}$, with $p,q\in (0,\infty ]$
the norm equivalences \eqref{Eq:IntroEquiv1}$'$
and \eqref{Eq:IntroEquiv2}$'$ extend the results
known so far.

\par

In the following table we
give an overview of the progress
for norm equivalences for
modulation spaces.
The extensions are obtained by going
downwards and to the right. Here recall
that $\mascP \subsetneq \mascP _E^0
\subsetneq \mascP _E$. Some more details are
also explained in Remark
\ref{Rem:ProgressNormSpaces} in Section
\ref{sec3}.

\par

\begin{table}[h]
\caption{Progress on norm
equivalences for modulation spaces
of the form $M(\omega ,\mascB )$ when $\mascB$
is a quasi-Banach space of order $r_0\in (0,1]$.}
\label{Table:ProgressModNorms}
\begin{tabular}{l | l l}
\diagbox{{\tiny{Spaces}}}
{\hspace{-0.3cm}{\tiny{Norms}}}
& {\hspace{-0.55cm}\footnotesize{$\begin{matrix}
\nm f{M^{p,q}_{(\omega )}} \asymp 
\nm {V_\phi f\cdot \omega}{L^{p,q}}
\\[1ex]
\hspace{1.8cm}
\asymp
\nm {V_\phi f\cdot \omega}{\sfW ^r(L^{p,q})}
\end{matrix}$}}
&
{\hspace{-0.55cm}\footnotesize{$\begin{matrix}
\nm f{M(\omega ,\mascB )}\asymp
\nm {V_\phi f\cdot \omega}{\mascB}
\\[1ex]
\hspace{2.15cm}
\asymp
\nm {V_\phi f\cdot \omega}{\sfW ^r(\mascB )}
\end{matrix}$}}
\\
\hline
&&
\\
{\footnotesize{Banach}}
& {\tiny{Feichtinger} (1983,1989) in
\cite{Fei1989,Fei1}:} &
{\tiny{Feichtinger-Gr{\"o}chenig (1989)
in \cite{FeiGro1}:}} 
\\[-0.5ex]
{\footnotesize{spaces}}
& {\tiny{$p,q\in [1,\infty]$,}}
& {\tiny{$\mascB$ only a Banach space (i.{\,}e. $r_0=1$),}}
\\[-0.5ex]
& {\tiny{$r=\infty$, $\omega$ only in $\mascP$.}}
& {\tiny{$r=\infty$, $\omega$ only in a subset of
$\mascP _E^0$.}}
\\
\\
{\footnotesize{quasi-Banach}}
& {\tiny{Galperin-Samarah} (2004) in
\cite{GaSa}:} &
{\tiny{Rauhut (2007) in \cite{Rau2}:}} 
\\[-0.5ex]
{\footnotesize{spaces}}
& {\tiny{$p,q\in (0,\infty]$, $r_0\le 1,p,q$,}}
& {\tiny{$\mascB$ a quasi-Banach space,}}
\\[-0.5ex]
& {\tiny{$r=\infty$, $\omega$ only in $\mascP$,}}
& {\tiny{$r=\infty$, $\omega$ only in a subset of
$\mascP _E^0$,}}
\\[-0.5ex]
&
{\tiny{restrictions
on $\phi \in M^{r_0}_{(v)}$.}}
&
{\tiny{no restrictions
on $\phi \in M^{r_0}_{(v)}$ (general).}}
\\
\\
{\footnotesize{quasi-Banach}}
& &
{\footnotesize{Theorem \ref{Thm:EquivNorms2}
in Sec. \ref{sec3}:}}
\\[-0.5ex]
{\footnotesize{spaces,}}
& &
{\tiny{$\mascB$ a normal quasi-Banach space}}
\\[-0.5ex]
{\tiny{(new results)}}
&
& {\tiny{$r \in [r_0,\infty ]$ (general), $\omega$ in
$\mascP _E$ (general),}}
\\[-0.5ex]
&
&
{\tiny{no restrictions on
$\phi \in M^{r_0}_{(v)}$ (general).}}
\end{tabular}
\end{table}

\medspace



\par


In the second part of the paper (Sections 
\ref{sec4}--\ref{sec8}),
we show how our
norm estimate results can be applied in 
different contexts.

\par

In Section \ref{sec4} we 
extend certain 
continuity and compactness properties 
obtained in
\cite{BogTof,PfTo19}.
For example, if $\mascB$ is a 
normal QBF space with respect to
$v_0\equiv 1$ 
and $\omega _j\in \mascP _E(\rr {2d})$, 
and $i$ denotes the inclusion map, then 
we prove
\begin{alignat*}{5}
i &: & \, M(\omega _1,\mascB )
&\to & \, 
M(\omega _2,\mascB )
\quad
&\text{continuous} &
\quad
& \Leftrightarrow &
\quad
\frac {\omega _2(X)}{\omega _1(X)}
&\le C
\intertext{and}
i &: & \, M(\omega _1,\mascB )
&\to & \, 
M(\omega _2,\mascB )
\quad
&\text{compact} &
\quad
& \Leftrightarrow &
\quad
\lim _{|X|\to \infty}
\frac {\omega _2(X)}{\omega _1(X)} &=0.
\end{alignat*}

\par

In Section \ref{sec5} we deduce
convolution estimates for
modulation spaces. In particular
we extend some results given in
\cite{Toft04,Toft26} to include
convolutions of modulation spaces
of the form $M(\omega ,\mascB )$.

\par

In Sections \ref{sec6}--\ref{sec8} 
we extend continuity
and lifting properties for pseudo-differential
and Toeplitz operators, available
in the literature (see e.{\,}g.
\cite{AbCaTo,AbCoTo,BenOko,CaTo,CorGro,
CorRod,Gro2,PilTeo2,Sjo,Tac,Teofanov3,
Toft04,Toft07,Toft18,ToBo08,TofUst}).

\par

One type of such results
originates from \cite{Sjo}
by Sj{\"o}strand, where the following
is proved concerning pseudo-differential
operators $\op (\fka )$ (with symbol $\fka$).
Here recall that the modulation space
$M^2(\rr d)$ is the same as $L^2(\rr d)$.

\par

\begin{prop}
\label{Prop:IntroSjostrand}
Let $\fka \in
M^{\infty ,1}(\rr {2d})$. Then
$\op (\fka )$ is continuous on
$M^2(\rr d)$.
\end{prop}

\par

We remark that Proposition
\ref{Prop:IntroSjostrand}
extend
Calderon-Vailancourt's theorem
in \cite{CalVai}, which shows that
Proposition \ref{Prop:IntroSjostrand}
holds true with the H{\"o}rmander class
$S^0_{0,0}(\rr {2d})$ in place of
$M^{\infty ,1}(\rr {2d})$. Note that
$S^0_{0,0}(\rr {2d}) \subseteq
M^{\infty ,1}(\rr {2d})$.

\par

Another type of such results
originates from \cite{Tac} by
Tachizawa, from which we obtain
the following continuity properties
for pseudo-differential operators
with symbols in the weighted
$S^0_{0,0}$ class
$S^{(\omega )}(\rr {2d})$.

\par

\begin{prop}
\label{Prop:IntroTachizawa}
Let $p,q\in [1,\infty ]$,
$\omega _0,\omega$ be
suitable weight functions on
$\rr {2d}$, and $\fka\in
S^{(\omega )}(\rr {2d})$.
Then $\op (\fka )$ is continuous
from $M^{p,q}_{(\omega _0\omega )}(\rr d)$
to $M^{p,q}_{(\omega _0 )}(\rr d)$.
\end{prop}

\par

There are several extensions of
Propositions \ref{Prop:IntroSjostrand}
and \ref{Prop:IntroTachizawa}.
For example, the following proposition
is a special case of
\cite[Theorem 3.1]{Toft18}, which extends both
Proposition \ref{Prop:IntroSjostrand}
and Proposition 
\ref{Prop:IntroTachizawa}. Here the involved
weigh functions are moderate and satisfy
\begin{equation}
\label{Eq:WeightIneqIntro}
\frac {\omega _2(x,\xi )}
{\omega _1(y,\eta )}
\lesssim
\omega _0(x,\eta ,\xi -\eta ,y-x),
\quad 
x,y,\xi ,\eta \in \rr d.
\end{equation}

\par

\begin{prop}
\label{Prop:IntroToft1}
Let $r_0\in (0,1]$, $p,q\in [r_0,\infty ]$,
$\omega _0$, $\omega _1$
and $\omega _2$ be suitable
weight functions
such that \eqref{Eq:WeightIneqIntro}
holds,
and let $\fka\in
M^{\infty ,r_0}_{(\omega _0)}(\rr {2d})$.
Then $\op (\fka )$ is continuous
from $M^{p,q}_{(\omega _1)}(\rr d)$
to $M^{p,q}_{(\omega _2)}(\rr d)$.
\end{prop}

\par

We observe that if $\omega _j=1$, $j=0,1,2$
and $r_0=1$, then Proposition \ref{Prop:IntroToft1}
essentially agrees with \cite[Theorem 14.5.3]{Gro1},
by Gr{\"o}chenig.

\par

In the next example, which
follows from \cite[Theorem 3.2]{Toft07},
we relax the assumptions
of the involved modulation
spaces in Proposition
\ref{Prop:IntroTachizawa}.

\par

\begin{prop}
\label{Prop:IntroToft2}
Let $\mascB _0$ be a solid invariant
BF space on $\rr {2d}$ with respect
to $v_0\equiv 1$,
$\omega _0,\omega$ be
suitable weight functions on
$\rr {2d}$, and $\fka\in
S^{(\omega )}(\rr {2d})$.
Then $\op (\fka )$ is continuous
from $M(\omega _0\omega ,\mascB _0)$
to $M(\omega _0 ,\mascB _0)$.
\end{prop}

\par

We observe that the class of
pseudo-differential operators
is larger in Proposition
\ref{Prop:IntroToft1} compared
to Proposition \ref{Prop:IntroToft2}.
On the other hand, the types of 
modulation spaces in the domains
and images are larger in Proposition
\ref{Prop:IntroToft2} compared
to Proposition \ref{Prop:IntroToft1}.
Consequently, Propositions
\ref{Prop:IntroToft1} and
\ref{Prop:IntroToft2} do not
contain each others. We also remark
that extensions and variations
of Proposition \ref{Prop:IntroToft2}
can be found in \cite{AbCaTo},
where the involved modulation
spaces can be quasi-Banach spaces,
and defined in the framework of
the theory of ultra-distributions.

\par

A natural question
is if it is possible to
deduce similar continuity
properties for more general
types of modulation spaces,
as in Proposition 
\ref{Prop:IntroToft2}, but for
a broader class of operators,
as in Proposition 
\ref{Prop:IntroToft1}.
In Section \ref{sec6} we give
affirmative answers on such
questions. For example, the
following result, covering
both Proposition \ref{Prop:IntroToft1}
and Proposition \ref{Prop:IntroToft2},
is a special case of
Theorem \ref{Thm:PseudoCont2}
in Section \ref{sec6}.

\par

\begin{thm}
\label{Thm:IntroMainPseudo}
Let $\mascB$ be
a normal solid invariant QBF space
on $\rr {2d}$ with respect to
$r_0\in (0,1]$
and $v_0\equiv 1$,
$\omega _0$, $\omega _1$
and $\omega _2$ be
weight functions
such that \eqref{Eq:WeightIneqIntro}
holds,
and let $\fka\in
M^{\infty ,r_0}_{(\omega _0)}(\rr {2d})$.
Then $\op (\fka )$ is continuous
from $M(\omega _1,\mascB )$
to $M(\omega _2,\mascB )$.
\end{thm}

\par

In the following table we
give an overview of the progress
for continuity results for
pseudo-differential operators with
smooth or
modulation space symbols
in weighted $M^{\infty ,q_0}$ spaces,
$q_0\in (0,1]$, when
acting on (other) modulation spaces.
The extensions are obtained by going
downwards and to the right. See also
Remark \ref{Rem:OtherPseudoResults}
in Section \ref{sec6} for other
continuity results for pseudo-differential
operators with symbols in modulation spaces.

\par

\begin{table}[h]
\caption{Progress on continuity
results for pseudo-differential operators
with smooth or modulation space symbols.}
\label{Table:ProgressPSDO}
\begin{tabular}{l | l l l l}
\diagbox{{\tiny{Symb.}}}
{\hspace{-0.3cm}{\tiny{Cont.}}}
& {\footnotesize{$L^2=M^2$}} & {\footnotesize{$M^{p,q}_{(\omega )}$}}
& {\footnotesize{$M(\omega ,\mascB)$ Banach }}
& {\footnotesize{$M(\omega ,\mascB)$ quasi-Banach}}
\\
\hline
&&&&
\\
{\footnotesize{Smooth,}}
& {\tiny{Calderon-}} &
{\tiny{Prop. \ref{Prop:IntroTachizawa}}} &
{\tiny{Prop. \ref{Prop:IntroToft2}}} &
\\[-0.5ex]
{\footnotesize{e.g. $S^{(\omega _0)}$}}
& {\tiny{Valaincourt}} & {\tiny{Tachizawa}} &
{\tiny{1st a.}} &
\\[-0.5ex]
& {\tiny{(1971)}} & {\tiny{(1994)}}
& {\tiny{(2009)}} &
\\
\\
{\footnotesize{Non-smooth,}}
& {\tiny{Prop. \ref{Prop:IntroSjostrand}}}
& {\tiny{Prop. \ref{Prop:IntroToft1}}}& &
{\small{Theorem 0.5}} 
\\[-0.5ex]
{\footnotesize{e.g.
$M^{\infty ,r_0}_{(1/\omega _{0})}$,
}}
& {\tiny{{\tiny{Sj{\"o}strand}}}} &
{\tiny{Gr{\"o}chenig, 1st a.}}& &
{\tiny{(Theorem \ref{Thm:PseudoCont2}
in Sect. \ref{sec6})}}
\\[-0.5ex]
{\footnotesize{$0<r_0\le 1$}}
& {\tiny{(1994)}} &
{\tiny{(2001,2017)}}
&&
\end{tabular}
\end{table}

\par

A fundamental ingredient in the proof of
Theorem \ref{Thm:PseudoCont2} (and 
thereby fundamental for Theorem 
\ref{Thm:IntroMainPseudo}),
concerns the achieved norm equivalences
\eqref{Eq:IntroEquiv1}$'$ and 
\eqref{Eq:IntroEquiv2}$'$,
which give some links on applicability 
of these equivalences.

\par

Theorem \ref{Thm:IntroMainPseudo}
also leads to improvements
of existing continuity results for Toeplitz
operators. For a suitable symbol $\fka$ on the
phase space $\rr {2d}$ and window functions
$\phi _j$, the Toeplitz operator on suitable
functions and distributions on $\rr d$ is
defined by the formula
\begin{equation*}
(\tp _{\phi _1,\phi _2}(\fka )f_1,f_2)
_{L^2(\rr d)}
= 
(\fka \cdot V_{\phi _1}f_1,
V_{\phi _2}f_2)_{L^2(\rr {2d})},
\end{equation*}
see \eqref{Eq:ToepDef}. By interpreting
Toeplitz operators as pseudo-differential
operators (see e.{\,}g. \eqref{Eq:ToeplWeyl}), 
Theorem
\ref{Thm:PseudoCont2} leads to Theorem
\ref{Thm:GenContResForToepl} in Section 
\ref{sec7}. A special case of the latter 
result is the following.

\par

\begin{thm}
\label{Thm:GenContResForToeplIntro}
Let $r_0\in (0,1]$, and suppose
$
\mascB
$
is an invariant QBF space on $\rr {2d}$
with respect to  $r_0$
and 
$v_0\equiv 1$,
$\omega \in \mascP _E (\rr {4d})$,
and
$\omega _1, \omega _2, \vartheta_1, 
\vartheta _2
\in \mascP _E (\rr {2d})$, 
are such that
\begin{equation*}
\frac {\omega _2 (x-z, \xi - \zeta)} 
{\omega _1 (x-y, \xi-\eta)}
\lesssim
\omega _0(x,\xi, \eta-\zeta, z-y)
\vartheta _1 (y,\eta )
\vartheta _2 (z, \zeta)
\end{equation*}
for $x,y, z, \xi, \eta, \zeta \in \rr d$. 
Also suppose
$
\phi _j \in 
M ^{r} _{(\vartheta _j)} (\rr d)
$
and
$
\fka \in 
M ^{\infty} _{(\omega _0)}( \rr {2d})
$.
Then 
$
\operatorname{Tp}_{\phi _1,\phi _2}(\fka )
$
is continuous from 
$M(\omega_1, \mascB)$ 
to 
$M(\omega_2, \mascB)$.
\end{thm}

\par

Theorem \ref{Thm:GenContResForToeplIntro}
itself extends and generalizes several well-known
results in the literature, e.{\,}g. in 
\cite{CorGro,CoPiRoTe2010,Toft04,Toft10}.
Note that the conditions on the weights in
Theorem
\ref{Thm:GenContResForToeplIntro} are somewhat
natural. In fact, using that
\begin{equation*}
\mascS 
=
\bigcap _{\vartheta _j\in \mascP}
M ^{r} _{(\vartheta _j)},
\quad \text{and}\quad
S^{(\omega )}
=
\bigcap _{r>0}M^{\infty}_{(\omega _{0,r})},
\quad
\omega _{0,r}(X,Y)
=
\frac {(1+|Y|)^r}{\omega (X)},
\end{equation*}
$X=(x,\xi )$, $Y=(\eta ,y)$,
it follows from Theorem
\ref{Thm:GenContResForToeplIntro} that if
$\phi _j \in \mascS (\rr d)$,
$\fka
\in S^{(\omega )}(\rr {2d})$, and
$$
\omega (x,\xi )
\lesssim
\frac {\omega _1 (x, \xi )} 
{\omega _2 (x, \xi )},
$$
then
$
\operatorname{Tp}_{\phi _1,\phi _2}(\fka )
$
is continuous from 
$M(\omega_1, \mascB)$ 
to 
$M(\omega_2, \mascB)$.

\par

\par

In Section \ref{sec8} we extend and improve the
lifting results in \cite{AbCoTo,GroTof1}, and
show that these hold for the broad family
of modulation spaces, considered in this paper.
For example, the following special case of Theorem
\ref{Thm:ToeplLift2} is out of reach in
\cite{AbCoTo,GroTof1}.

\par

\begin{thm}
\label{Thm:ToeplLift2Intro}
Let $s> 1$, 
$\omega ,\omega _0,v,v_1\in
\mascP _{s}(\rr {2d})$ be such that
$\omega _0$ is
$v$-moderate and  $\omega $ is
$v_1$-moderate, and let
$\mascB$ be an invariant BF space
with respect to $v_0\equiv 1$. Set
$\vartheta =\omega _0^{1/2}$.
If  $\phi \in M^{1}_{(v_1v)}(\rr d)$,
then $\tp _\phi (\omega _0)$ is
an isomorphism from $M(\vartheta
\omega ,\mascB )$ to
$M(\omega /\vartheta,\mascB )$.
\end{thm}

\par

%
%
%
%
%
%
%
%
%

\par

\section*{Acknowledgements}
J. Toft is supported by Vetenskapsr{\aa}det
(Swedish Science Council), within the project
2019-04890.
C. Pfeuffer is supported by the 
Deutsche Forschungsgemeinschaft 
(DFG, German Research Foundation) 
within the project 
442104339.
N. Teofanov is supported by the Science Fund of the Republic of
Serbia, $\#$GRANT No. 2727, \emph{Global and local analysis of operators and
distributions} - GOALS, and by NITRA (Grants No. 451--03--66/2024--03/200125 $\&$ 451--03--65/2024--03/200125).

\par

\section{Preliminaries}\label{sec1}

\par

In this section we recall some basic facts about
weight functions, the short-time Fourier transform 
and some spaces of  function and distributions which will be used in the sequel. 

\par

\subsection{Weight functions}\label{subsec1.1}

\par

A \emph{weight} or \emph{weight function} on $\rr d$ is a
positive function $\omega
\in  L^\infty _{loc}(\rr d)$ such that $1/\omega \in  L^\infty _{loc}(\rr d)$.
If there is a  weight $v$ on $\rr d$ and a constant $C\ge 1$ such that
\begin{equation}\label{moderate}
\omega (x+y) \le C\omega (x)v(y),\qquad x,y\in \rr d,
\end{equation}
then the weight $\omega$ is called
\emph{moderate}, or \emph{$v$-moderate}.
By \eqref{moderate}
we have
\begin{equation}\label{moderateconseq}
C^{-1}v(-x)^{-1}\le \omega (x)\le C v(x),\quad x\in \rr d.
\end{equation}
We let $\mascP _E(\rr d)$ be the set of all moderate weights on $\rr d$.

\par

We say that a  weight $v$ is \emph{submultiplicative} if
\begin{equation}\label{Eq:Submultiplicative}
v(x+y) \le v(x)v(y)
\quad \text{and}\quad
v(-x)=v(x),\qquad x,y\in \rr d.
\end{equation}

\par

If $v$ is positive and locally bounded and
satisfies the  inequality in 
\eqref{Eq:Submultiplicative},
then $v(x)\le Ce^{r|x|}$ for some positive 
constants $C$ and $r$, cf. \cite{Gro2007}.

\par

We observe that given a $v$-moderate weight 
$\omega$, one can find a continuous $v$-
moderate weight $\omega_0$ such that $\omega 
\asymp \omega_0$ (see e.{\,}g.
\cite{Toft10}). In addition, a moderate 
weight $\omega$
is also moderated by a submultiplicative
weight, cf. e.{\,}g. \cite{FeiGro1,Toft10}.
Therefore, if $\omega \in \mascP _E(\rr d)$, 
then
\begin{equation}\label{Eq:weight0}
\omega (x+y) \lesssim \omega (x) e^{r|y|},
\quad x,y\in \rr d,
\end{equation}
for some $r>0$. 
Here $g_1 \lesssim g_2$ means that
$g_1(x ) \le c \cdot  g_2(x)$
holds uniformly for all $x$
in the intersection of the domains of $g_1$ and $g_2$, and
for some constant $c>0$. We
write $g_1\asymp g_2$
when $g_1\lesssim g_2 \lesssim g_1$.

In particular, \eqref{moderateconseq} shows that
for any $\omega \in \mascP_E(\rr d)$,
there is a constant $r>0$ such that
\begin{equation}\label{Eq:weight1}
e^{-r|x|}\lesssim \omega (x)\lesssim e^{r|x|},\quad x\in \rr d.
\end{equation}

\par

In the sequel, $v$ and $v_0$, always stand for
submultiplicative weights if nothing else is stated.

\par

Some parts of our investigations require
more families of weight functions.

\par

\begin{defn}
\label{Definition:Weightclasses}
Let $s,\sigma >0$.
\begin{enumerate}
\item $\mascP (\rr d)$ consists of
all $\omega \in \mascP _E(\rr d)$ such that
$\omega$ is moderate by $v(x)=\eabs x^s$,
for some $s\ge 0$.

\vrum

\item
The set $\mascP _{s}(\rr d)$ 
($\mascP _{s}^0(\rr d)$) consists of
all $\omega \in \mascP _E(\rr d)$ such that
\begin{equation}\label{Eq:Gsmodw}
\omega(x+y)
\lesssim
\omega(x) e^{r|y|^\frac{1}{s}},
\quad x,y\in\rr{d};
\end{equation}
holds for some (every) $r>0$.

\vrum

\item
The set $\mascP _{s,\sigma }(\rr {2d})$ 
($\mascP _{s,\sigma}^0(\rr {2d})$) consists of
all $\omega \in \mascP _E(\rr {2d})$ such that
\begin{equation}\label{Eq:Gsmodw2}
\omega(x+y,\xi +\eta )
\lesssim
\omega(x,\xi )
e^{r(|y|^\frac{1}{s}+|\eta |^{\frac 1\sigma})},
\quad x,y,\xi ,\eta \in\rr{d};
\end{equation}
holds for some (every) $r>0$.
\end{enumerate}
\end{defn}

\par

\begin{rem}
\label{Rem:WeightClGrpProp}
We observe that each class of weight
functions in Definition
\ref{Definition:Weightclasses} is
a group under multiplication.
\end{rem}

\par

By \eqref{Eq:weight0} it follows that
$\mascP _{s_1}^0=\mascP _{s_2}
=\mascP _E$ when $s_1<1$ and $s_2\le 1$.
For convenience we set
$\mascP^0_E(\rr d)=\mascP^0_{1}(\rr d)$,
and note that
$$
\mascP (\rr d)
\subseteq
\mascP _{s_1}(\rr d)
\subseteq
\mascP _{s_2}^0(\rr d)
\subseteq
\mascP _{s_2}(\rr d),
\quad \text{when}\quad
0<s_1<s_2.
$$

\par

A submultiplicative weight $v$
on $\rr d$ is called GRS weight
(i.{\,}e. Gelfand-Raikov-Shilov weight),
if
$$
\lim _{n\to \infty}v(nx)^{\frac 1n}=1,
\quad x\in \rr d.
$$
A general weight $\omega$ in
$\mascP _E(\rr d)$ is called GRS weight,
if $\omega$ is $v$-moderate for some
submultiplicative GRS weight $v$ on $\rr d$.
We recall that $\mascP^0_E(\rr d)$ agrees
with is the set of all GRS weights
on $\rr d$.
We refer to \cite{Fei0,Gro2,Gro2.5,Toft10}
for these and other facts about weights in 
time-frequency analysis.

\par


\subsection{Gelfand-Shilov spaces}
\label{subsec1.2}

\par

In what follows we let $\mathscr F$ be the
Fourier transform which takes the form
$$
(\mathscr Ff)(\xi )= \widehat f(\xi ) \equiv (2\pi )^{-\frac
d2}\int _{\rr
{d}} f(x)e^{-i\scal  x\xi }\, dx, 
\quad 
\xi \in \rr d,
$$
when $f\in L^1(\rr d)$. Here $\scal \cdo \cdo$ denotes the usual
scalar product
on $\rr d$. 
The map $\mathscr F$ extends
uniquely to a homeomorphism on the space of tempered distributions
$\mascS '(\rr d)$,
to a unitary operator on $L^2(\rr d)$ and restricts
to a homeomorphism on the Schwartz space of smooth rapidly
decreasing functions $\mathscr S(\rr d)$.
Recall that $f\in C^\infty (\rr d)$
belongs to $\mascS (\rr d)$, if
and only if
$$
\nm {x^\alpha \partial ^\beta f}
{L^\infty}<\infty ,
$$
for every $\alpha ,\beta \in \nn d$.
Here $\mathbf N =\{ 0,1,\dots \}$
is the set of natural numbers,
and we use the standard terminology
on multi-indices.

\par

We observe that with our choice of the Fourier
transform, the usual
convolution identity for the Fourier transform
takes the forms
\begin{equation}\label{Eq:FourTransfConv}
\mascF (f\cdot g)
=
(2\pi )^{-\frac d2}\widehat f *\widehat g
\quad \text{and}\quad
\mascF (f* g)
=
(2\pi )^{\frac d2}\widehat f \cdot \widehat g
\end{equation}
when $f,g\in \mascS (\rr d)$. 
As usual,
$
(f* g) (x) = 
\int _{\rr {d}} f(x-y) g (y)\, dy$, 
$ y \in \rr d.$

\par

Since we are interested in general weights $\omega \in \mascP _E(\rr d)$, instead of the framework of
tempered distributions $\mascS '(\rr d)$, which is natural when dealing with weights of polynomial growth,
we consider the Gelfand-Shilov space 
$\Sigma _1(\rr d)$ 
and its dual space of (ultra-)distributions 
$\Sigma _1 '(\rr d)$.

\par

Let $r,s,\sigma \in \mathbf R_+$. Then
the space $\maclS _{s,r}^\sigma (\rr d)$
consists of all
$f\in C^\infty (\rr d)$
such that
\begin{equation}
\label{Eq:GSNorm}
\nm f{\maclS _{s,r}^\sigma}
\equiv
\sup _{\alpha ,\beta \in \nn d}
\left (
\frac {\nm {x^\alpha \partial ^\beta f}
{L^\infty}}
{r^{|\alpha +\beta |}
\alpha !^s\beta !^\sigma}
\right )
<\infty. 
\end{equation}
We let the topology
of $\maclS _{s,r}^\sigma (\rr d)$
be defined by the norm \eqref{Eq:GSNorm},
giving that $\maclS _{s,r}^\sigma (\rr d)$
is a Banach space.

\par

The \emph{Gelfand-Shilov
spaces} $\maclS _s^\sigma (\rr d)$
and $\Sigma _s^\sigma (\rr d)$,
of \emph{Roumieu} and \emph{Beurling types},
respectively, are given by
$$
\maclS _s^\sigma (\rr d)
=
\bigcup _{r>0}
\maclS _{s,r}^\sigma (\rr d)
\quad \text{and}\quad
\Sigma _s^\sigma (\rr d)
=
\bigcap _{r>0}
\maclS _{s,r}^\sigma (\rr d).
$$
We equipp $\maclS _s^\sigma (\rr d)$
and $\Sigma _s^\sigma (\rr d)$
with inductive respectively projective limit
topologies of
$\maclS _{s,r}^\sigma (\rr d)$
with respect to $r$. It follows that
$\Sigma _s^\sigma (\rr d)$
is a Fr{\'e}chet space (under
the semi-norms \eqref{Eq:GSNorm}),
(See e.{\,}g. \cite{GeSh,Pil}.)

\par

There are several characterizations
of $\maclS _s^\sigma (\rr d)$
and $\Sigma _s^\sigma (\rr d)$.
For example,
in several situations it is convenient
to use that if
$f\in \mascS '(\rr d)$,
then 
$f\in \maclS _s^\sigma (\rr d)$
($f\in \Sigma _s^\sigma (\rr d)$),
if and only if
\begin{equation}\label{Eq:GSFtransfChar}
|f(x)|\lesssim e^{-r|x|^{\frac 1s}}
\quad \text{and}\quad
|\widehat f(\xi )|
\lesssim
e^{-r|\xi |^{\frac 1\sigma}},
\quad \text{for some (every)}\quad
r>0,
\end{equation}
cf. \cite{ChuChuKim, Eij}.

\par

As usual, the dual of
$\maclS _{s,r}^\sigma (\rr d)$
is denoted by
$(\maclS _{s,r}^\sigma )'(\rr d)$.
The \emph{Gelfand-Shilov
distribution space}
$(\maclS _s^\sigma )'(\rr d)$
of \emph{Roumieu type}
is the projective limit
of $(\maclS _{s,r}^\sigma )'(\rr d)$
with respect to $r$. 
The \emph{Gelfand-Shilov
distribution space}
$(\Sigma _s^\sigma )'(\rr d)$
of \emph{Beurling type}
is the inductive limit
of $(\maclS _{s,r}^\sigma )'(\rr d)$
with respect to $r$. It follows that
$(\maclS _s^\sigma )'(\rr d)$ is a
Fr{\'e}chet space,
$$
(\maclS _s^\sigma )'(\rr d)
=
\bigcap _{r>0}
(\maclS _{s,r}^\sigma )'(\rr d)
\quad \text{and}\quad
(\Sigma _s^\sigma )'(\rr d)
=
\bigcup _{r>0}
(\maclS _{s,r}^\sigma )'(\rr d).
$$
For convenience we set $\maclS _s=\maclS _s^s$
and $\Sigma _s = \Sigma _s^s$.

\par

If $s+\sigma \ge 1$, then
the $L^2$ scalar product,
$(\cdo ,\cdo )_{L^2}$ on
$\maclS _s^\sigma (\rr d)\times
\maclS _s^\sigma (\rr d)$ is uniquely
extendable to a continuous sesqui-linear
form $(\cdo ,\cdo )$ on $(\maclS _s^\sigma )'(\rr d)\times
\maclS _s^\sigma (\rr d)$,
and it follows that
$(\maclS _s^\sigma )'(\rr d)$
is the (strong) dual of
$\maclS _s^\sigma (\rr d)$
under this form. Similarly when $\maclS _s^\sigma $
is replaced by $\Sigma _s^\sigma $ at each occurance.
We have
\begin{equation}
\label{Eq:GSEmb}
\begin{aligned}
\maclS _s^\sigma (\rr d)
&\hookrightarrow
\Sigma _t^\tau (\rr d)
\hookrightarrow
\maclS _t^\tau (\rr d)
\\[1ex]
&\hookrightarrow
\mascS (\rr d)
\hookrightarrow
L^2(\rr d)
\hookrightarrow
\mascS '(\rr d)
\\[1ex]
&\hookrightarrow
(\maclS _t^\tau )'(\rr d)
\hookrightarrow
(\Sigma _t^\tau )'(\rr d)
\hookrightarrow
(\maclS _s^\sigma )'(\rr d),
\\[1ex]
&\text{when}\quad
s<t,\ \sigma <\tau ,\ s+\sigma \ge 1.
\end{aligned}
\end{equation}
Here $A \hookrightarrow B$
means that $A \subseteq B$ with
continuous inclusion. Furthermore,
the inclusions in \eqref{Eq:GSEmb}
are dense and strict. For other choices of
involved parameters we have
$$
\maclS _s^\sigma (\rr d)
=
\Sigma _t^\tau (\rr d)=\{ 0\}
\quad \text{when}\quad
\sigma + s < 1 ,\ t+\tau \le 1
$$
(see \cite{GeSh,Pet1}).
For example, as a consequence of
\eqref{Eq:GSFtransfChar} we have
\begin{equation}\label{Eq:vrhoEmb}
v_r (x)\equiv e^{r |x|}
\in
\Sigma _1'(\rr d)\setminus
\mascS '(\rr d),\quad x\in \rr d,
\quad \text{for every}\quad
r \in \mathbf R_+.
\end{equation}

\par

By \eqref{Eq:GSFtransfChar} it follows
that the following is true.

\par

\begin{prop}
Let $s,\sigma >0$ be such that $s+\sigma \ge 1$,
$x,\xi \in \rr d$, $A$ be a non-degenerate
real $d\times d$ matrix, and let $T_j$,
$j=1,2,3$,
be the mappings on $\mascS (\rr d)$,
given by
$$
T_1 : f\mapsto e^{i\scal \cdo \xi}f(\cdo -x),
\quad 
T_2 : f\mapsto f(A\cdo )
\quad \text{and}\quad
T_3 : f\mapsto \widehat f.
$$
Then the following is true:
\begin{enumerate}
\item $T_1$ and $T_2$ restrict
to homeomorphisms on $\maclS _s^\sigma (\rr d)$,
and $T_3$ restricts to a homeomorphism
from $\maclS _s^\sigma (\rr d)$ to
$\maclS ^s_\sigma (\rr d)$;

\vrum

\item $T_1$ and $T_2$ extend
uniquely
to homeomorphisms on
$(\maclS _s^\sigma )'(\rr d)$,
and $T_3$ extends uniquely
to a homeomorphism
from $(\maclS _s^\sigma )'(\rr d)$ to
$(\maclS ^s_\sigma )'(\rr d)$.
\end{enumerate}
If, more restrictive, $s+\sigma >1$,
then the same holds true with
$\Sigma _s^\sigma$ and
$\Sigma ^s_\sigma$ in place of
$\maclS _s^\sigma$ and
$\maclS ^s_\sigma$,
respectively, at each occurrence.
\end{prop}

\par

\subsection{Short-time Fourier transform}
In several situations it is convenient 
to use a localized version
of the Fourier transform.
More precisely,
the short-time Fourier transform of
$f\in \maclS _s^\sigma (\rr d)$
($f\in \Sigma _s^\sigma (\rr d)$)
with respect to the fixed \emph{window function}
$\phi \in \maclS _s^\sigma  (\rr d)$
($\phi \in \Sigma _s^\sigma (\rr d)$) is defined by
\begin{equation} \label{Eq:STFTDef}
(V_\phi f)(x,\xi ) \equiv (2\pi )^{-\frac d2}
(f,\phi (\cdo -x)e^{i\scal \cdo \xi})_{L^2}, 
\quad 
x, \xi \in \rr d.
\end{equation}
We observe that using certain properties for 
tensor products of distributions,
\begin{equation}
(V_\phi f)(x,\xi ) = \mascF (f\cdot \overline {\phi (\cdo -x)})(\xi ).
\tag*{(\ref{Eq:STFTDef})\text{$'$}},
\end{equation}
cf. \cite{Ho1,Toft22}. If, in addition
$f\in L^p(\rr d)$ for some $p\in [1,\infty ]$, then
\begin{equation}
(V_\phi f)(x,\xi ) = (2\pi )^{-\frac d2}\int _{\rr d}f(y)
\overline {\phi (y-x)}e^{-i\scal y\xi}\, dy
\quad 
x, \xi \in \rr d.
\tag*{(\ref{Eq:STFTDef})\text{$''$} }
\end{equation}

\par

From now on we usually restrict ourselves
to the case when $s,\sigma \ge 1$.
The following proposition deals with
some basic continuity properties
for the short-time Fourier transform.
We omit the proof since the result is
a special case of results presented
in \cite{Fol,GroZim2001,Teofanov2,Toft10,Toft17}.

\par

\begin{prop}
\label{Prop:ExtSTFTSchwartz}
Let $s\ge 1$. Then the map
\begin{alignat}{4}
(f,\phi ) &\mapsto V_\phi f &  &:\,  &
&\Sigma _{1} (\rr d) \times \Sigma _{1} 
(\rr d) & &\to \Sigma _{1} (\rr {2d})
\label{Eq:STFTSchwartz}
\intertext{is continuous, and is
uniquely extendable to a continuous map}
(f,\phi ) &\mapsto V_\phi f & &:\,  &
&\maclS _s'(\rr d)
\times \maclS _s'(\rr d) & &\to 
\maclS _s'(\rr {2d}),
\label{Eq:STFTTempDist}
\intertext{which in turn restricts to a 
continuous map}
(f,\phi ) &\mapsto V_\phi f & &:\,  &
&\maclS _s '(\rr d)
\times \maclS _s (\rr d) & &
\to \maclS _s '(\rr {2d})
\cap
C^\infty (\rr {2d}),
\label{Eq:STFTTempDistSmooth}
\intertext{and to a continuous map}
(f,\phi ) &\mapsto V_\phi f & &:\,  &
&L^2(\rr d) \times L^2(\rr d) & &\to 
L^2(\rr {2d}),
\label{Eq:STFTL2}
\end{alignat}
with $\nm {V_\phi f}{L^2}
=\nm f{L^2}\nm \phi{L^2}$.

\par

The same holds true with $\Sigma _s$
in place of $\maclS _s$ at each occurrence,
or with $\mascS$
at each occurrence.
\end{prop}

\par

If $s=1$ in
Proposition \ref{Prop:ExtSTFTSchwartz},
then it follows that \eqref{Eq:STFTTempDistSmooth}
can be refined as continuous mappings
\begin{alignat}{4}
\tag*{(\ref{Eq:STFTTempDistSmooth})$'$}
(f,\phi ) &\mapsto V_\phi f & &:\,  &
&\maclS _1 '(\rr d)
\times \maclS _1 (\rr d) & &
\to \maclS _1 '(\rr {2d})
\cap
A_R (\rr {2d})
\intertext{and}
\tag*{(\ref{Eq:STFTTempDistSmooth})$''$}
(f,\phi ) &\mapsto V_\phi f & &:\,  &
&\Sigma _1 '(\rr d)
\times \Sigma _1 (\rr d) & &
\to \Sigma _1 '(\rr {2d})
\cap
A(\rr {2d}).
\end{alignat}
Here $A_R(\rr d)$ is the set of all
real analytic functions on $\rr d$,
and $A(\rr d)$ is the set of all
functions in $A_R(\rr d)$
which are extendable to entire functions
on $\cc d$.

\par

We will often use the formula
\begin{align}
|V_\phi f(X)|
&=
|V_f \phi (-X)|,
\notag
\\[1ex]
X &=
(x,\xi )\in \rr {2d},\ 
f\in \Sigma _1'(\rr d),\ 
\phi \in \Sigma _1(\rr d),
\label{Eq:SwapWindows1}
\intertext{which is a special case
of}
V_\phi f(X)
&=
e^{-i\scal x\xi}
\overline{V_f \phi (-X)},
\notag
\\[1ex]
X&=(x,\xi )\in \rr {2d},\ 
f,\phi \in \Sigma _1'(\rr d),
\label{Eq:SwapWindows2}
\end{align}
which follows by straight-forward
computations.

\par

For the short-time Fourier transform, 
the Parseval identity
is replaced by
\begin{equation}\label{Eq:Moyal}
(V_{\phi _1}f,V_{\phi _2}g)
_{L^2(\rr {2d})}
= (\phi _2 ,\phi _1 )_{L^2(\rr d)}
(f,g)_{L^2(\rr d)},
\end{equation}
$f,g,\phi _1,\phi _2 \in L^2(\rr d)$,
and is often called Moyal's identity.

\par

By \eqref{Eq:Moyal} and the
fact that the embeddings in 
\eqref{Eq:GSEmb} are dense
we get the following.

\par

\begin{prop}\label{Prop:Moyal}
Suppose $j,k\in \{ 1,2\}$ satisfy
$j\neq k$. Then the following is true:
\begin{enumerate}
\item the identity \eqref{Eq:Moyal}
holds true when
$f,g,\phi _1,\phi _2 \in \Sigma _{1}
(\rr d)$;

\vrum

\item the identity in {\rm{(1)}}
extends continuously and uniquely
to any $f\in \Sigma _1'(\rr d)$,
$g\in \Sigma _{1}(\rr d)$ and
$\phi _1,\phi _2 \in L^2(\rr d)$;

\vrum

\item the identity in {\rm{(1)}}
extends continuously and uniquely
to any $f,\phi _j\in \Sigma _1'(\rr d)$
and $g,\phi _k \in \Sigma _{1}(\rr d)$.
\end{enumerate}
\end{prop}

\par

Obviously, the adjoint
$V_\phi ^*$ of $V_\phi$ fulfills
\begin{equation}\label{Eq:STFTAdjoint}
(V_\phi ^*F,g) _{L^2(\rr d)} = (F,V_\phi g) _{L^2(\rr {2d})}
\end{equation}
when $F\in \Sigma _{1} (\rr {2d})$ and $g\in \Sigma _{1}(\rr d)$.
%
%
By straight-forward computations
it follows that
\begin{equation}\label{Eq:STFTAdjointFormula}
(V_\phi ^*F)(x) = (2\pi )^{-\frac d2}
\iint _{\rr {2d}} F(y,\eta )\phi (x-y)e^{i\scal x\eta}\, dy \, d\eta ,
\end{equation}
when $F\in \Sigma _{1} (\rr {2d})$
and $\phi \in \Sigma _1(\rr d)$.
We may now use mapping properties
like \eqref{Eq:STFTTempDist}--\eqref{Eq:STFTL2} to extend
the definition of $V_\phi ^*F$ when $F$ and $\phi$ belong to
various classes of function and distribution spaces. For example,
by \eqref{Eq:STFTTempDist},
\eqref{Eq:STFTSchwartz} and \eqref{Eq:STFTL2},
it follows that the map
$$
(F,g)\mapsto (F,V_\phi g) _{L^2(\rr {2d})}
$$
defines a sesqui-linear form on
$\Sigma _{1} (\rr {2d})\times\Sigma _{1} '(\rr d)$,
$\Sigma _{1} '(\rr {2d})\times \Sigma _{1} (\rr d)$ and on
$L^2(\rr {2d})\times L^2(\rr d)$. This implies that
if $\phi \in \Sigma _{1} (\rr d)$, then
the mappings
\begin{equation}\label{Eq:STFTAdjCont}
\begin{alignedat}{3}
V_\phi ^* :  \Sigma _{1} (\rr {2d}) \to \Sigma _{1} (\rr d),
\qquad
&V_\phi ^* & &: &\Sigma _{1} '(\rr {2d}) &\to \Sigma _{1} '(\rr d)
\\[1ex]
\text{and}\qquad
&V_\phi ^* & &: & L^2(\rr {2d}) &\to L^2(\rr d)
\end{alignedat}
\end{equation}
are continuous.

\par

By Propositions 
\ref{Prop:ExtSTFTSchwartz} and 
\ref{Prop:Moyal},
\eqref{Eq:STFTAdjCont}
and duality it follows that if
$\phi \in \Sigma _{1} (\rr d)\setminus 0$, then the identity operator
on $\Sigma _{1} '(\rr d)$ is given by
\begin{equation}\label{Eq:IdentSTFTAdj}
\operatorname{Id}
=
\left (\nm \phi{L^2}^{-2}\right )
\cdot V_\phi ^*\circ V_\phi .
\end{equation}

\par

We will often use the estimates
\begin{align}
| (V _{\phi _1} \circ V_{\phi _2} ^*)F| &\leq
|F| * | V _{\phi _1}  \phi _2|
\label{Eq:adjoint-convolution}
\intertext{and}
| V _{\phi _1}  f |
&\leq 
\| \phi_2 \|_{L^2} ^{-2} 
(|V _{\phi _2} f| * | V _{\phi _1}
\phi _2|).
\label{Eq:change-of-window}
\end{align}
Clearly, \eqref{Eq:change-of-window} follows from \eqref{Eq:IdentSTFTAdj} 
and \eqref{Eq:adjoint-convolution}. Here above $F$ is a suitable distribution
on $\rr {2d}$, 
$f$ is a suitable distribution on $\rr d$, and 
$\phi_1, \phi_2$ are  suitable functions 
on $\rr d$. We refer to
e.g Proposition 11.3.2 in \cite{Gro2} for the proofs of 
\eqref{Eq:IdentSTFTAdj} 
and \eqref{Eq:adjoint-convolution}.

\par

We also need the following estimate of 
the short-time Fourier transform with a 
Gaussian window. For the proof we refer 
to \cite[Lemma 2]{Toft2022}. Here $B_r(x)$
denotes the open ball in $\rr d$ with center at
$x\in \rr d$ and radius $r>0$.

\par

\begin{lemma}\label{Lemma:EstSTFT}
Let $p \in (0,\infty]$, $R > 0$,
$X_0\in \rr {2d}$ be fixed, 
and let $\phi _0(x)
=
\pi ^{-\frac d4}e^{-\frac 12|x|^2}$, $x\in \rr d$. Then
$$
\left | V_{\phi_0} f (X_0) \right | 
\leq
C \nm{ V_{\phi_0} f }{L^p (B_R (X_0) )}, 
\qquad
f \in \Sigma_1' (\rr {d} ),
$$
where the constant $C$ is independent
of $X_0$ and $f$.
\end{lemma}

\par

\subsection{Classical modulation 
spaces}\label{subsec1.4}

\par

For the introduction of classical
modulation spaces, we need to discuss some
issues on mixed norm spaces of Lebesgue types.
Let $\Omega \subseteq \rr d$ be a
(Lebesgue) measurable set in $\rr d$,
and let $\maclM (\Omega )$ be the set of all 
complex-valued (Lebesgue)
measurable functions on $\Omega$.

\par

For any $p,q\in (0,\infty ]$, the mixed norm
spaces
$L^{p,q}_{(\omega )}(\rr {d_1+d_2})$
and $L^{p,q}_{*,(\omega )}(\rr {d_1+d_2})$
of Lebesgue type, are given by
\begin{alignat}{3}
L^{p,q}_{(\omega )}(\rr {d_1+d_2})
&=
\sets{F\in \maclM (\rr {d})}
{\nm F{L^{p,q}_{(\omega )}}<\infty }
\label{Eq:MixLebSpace1}
\intertext{and}
L^{p,q}_{*,(\omega )}(\rr {d_1+d_2})
&=
\sets{F\in \maclM (\rr {d})}
{\nm F{L^{p,q}_{*,(\omega )}}<\infty },
\label{Eq:MixLebSpace2}
\end{alignat}
for some $p,q\in (0,\infty ]$,
$\omega \in \mascP _E(\rr {d})$,
and $d_1,d_2\ge 0$ are integers which
satisfy
$d=d_1+d_2>0$. Here
\begin{alignat*}{3}
\nm f{L^{p,q}_{(\omega )}}
&\equiv
\nm {G_{f,\omega ,p}}{L^q(\rr {d_2})}, &
\quad
G_{f,\omega ,p}(\xi )
&=
\nm {f(\cdo ,\xi )\omega (\cdo ,\xi )}
{L^p(\rr {d_1})}&,
\quad
f&\in \maclM (\rr {d})
\intertext{and}
\nm f{L^{p,q}_{*,(\omega )}}
&\equiv
\nm {H_{f,\omega ,q}}{L^p(\rr {d_1})}, &
\quad
H_{f,\omega ,q}(x)
&=
\nm {f(x,\cdo )\omega (x,\cdo )}{L^q(\rr {d_2})}&,
\quad
f&\in \maclM (\rr {d}).
\end{alignat*}
We put
$$
L^{p,q}= L^{p,q}_{(\omega )}
\quad \text{and}\quad
L^{p,q}_*= L^{p,q}_{*,(\omega )}
\quad \text{when}\quad
\omega \equiv 1.
$$
We also put $L^p_{(\omega )}=L^{p,p}_{(\omega )}$, and observe that
$L^p=L^{p,p}$.

The (classical) modulation spaces, essentially 
introduced in \cite{Fei1} by Feichtinger are 
given in the following.
(See e.{\,}g. \cite{Fei6} for definition of more general modulation spaces.)

\par

\begin{defn}
\label{Def:ClassModSp}
Let $p,q\in (0,\infty ]$,
$\omega \in \mascP _E(\rr {2d})$ and
$\phi \in \Sigma _1 (\rr d)\setminus 0$.
\begin{enumerate}
\item The (classical) \emph{modulation space} 
$M^{p,q}_{(\omega )}(\rr d)$
consists of all $f\in \Sigma _1 '(\rr d)$
such that
\begin{equation}
\label{Eq:ModNorm2}
\nm f{M^{p,q}_{(\omega )}}\equiv \nm {V_\phi f}
{L^{p,q}_{(\omega )}}
\end{equation}
is finite. The topology of $M^{p,q}_{(\omega )}
(\rr d)$ is defined by
the (quasi-)norm $\nm \cdo{M^{p,q}_{(\omega )}}$;

\vrum

\item The \emph{modulation space (of Wiener amalgam type)}
$W^{p,q}_{(\omega )}(\rr d)$ consists of all $f\in \Sigma _1 '(\rr d)$ such that
$$
\nm f{W^{p,q}_{(\omega )}}\equiv \nm {V_\phi f}{L^{p,q}_{*,(\omega )}}
$$
is finite. The topology of $W^{p,q}_{(\omega )}(\rr d)$ is defined by
the (quasi-)norm $\nm \cdo{W^{p,q}_{(\omega )}}$.
\end{enumerate}
\end{defn}

\par

In the following propositions we
list well-known properties for modulation
spaces. The first proposition deals with
invariant and topological properties
for modulation spaces.

\par

\begin{prop}
\label{Prop:InvModAndBanach}
Let $p,q\in (0,\infty ]$,
$\omega \in \mascP _E(\rr {2d})$ and
$\phi \in \Sigma _1(\rr d)\setminus 0$.
Then the following is true:
\begin{enumerate}
\item the definitions of
$M_{(\omega )}^{p,q}(\rr d)$ and
$W_{(\omega )}^{p,q}(\rr d)$
are independent of the choices of
$\phi \in \Sigma _1 (\rr d)\setminus 0$,
and different choices give rise to
equivalent quasi-norms;

\vrum

\item the spaces $M^{p,q}_{(\omega )}(\rr d)$ and
$W^{p,q}_{(\omega )}(\rr d)$ are quasi-Banach 
spaces
which increase with $p$ and $q$, and
decrease with $\omega$. If in addition
$p,q\ge 1$, then they are Banach spaces.
\end{enumerate}
\end{prop}

\par

For convenience we set
$M^p_{(\omega )}=M^{p,p}_{(\omega )}$.
If in addition $\omega \equiv 1$, then we
set $M^{p,q}=M^{p,q}_{(\omega )}$
and $M^p=M^p_{(\omega )}$.

\par

The next proposition deals with questions on
duality. We need the conjugate exponents of $p$
respectively $q$, denoted by $p'$ respectively $q'$.
They are defined by 
$\frac{1}{p}+\frac{1}{p'} =1$ and
$\frac{1}{q}+\frac{1}{q'} =1$.

\par

\begin{prop}
\label{Prop:ModDuality}
Let $p,q\in [1,\infty ]$,
$\omega \in \mascP _E(\rr {2d})$ and
$\phi \in \Sigma _1(\rr d)\setminus 0$.
Then the following is true:
\begin{enumerate}
\item the $L^2(\rr d)$ scalar product,
$(\cdo ,\cdo )_{L^2(\rr d)}$, on $\Sigma _1 (\rr d)\times \Sigma _1 (\rr d)$ is
uniquely extendable to dualities between 
$M^{p,q}_{(\omega )}(\rr d)$
and $M^{p',q'}_{(1/\omega )}(\rr d)$, and between
$W^{p,q}_{(\omega )}(\rr d)$
and $W^{p',q'}_{(1/\omega )}(\rr d)$;

\vrum

\item if in addition $p,q<\infty$, then
the dual spaces of $M^{p,q}_{(\omega )}(\rr d)$ 
and $W^{p,q}_{(\omega )}(\rr d)$
can be identified with
$M^{p',q'}_{(1/\omega )}(\rr d)$ and
$W^{p',q'}_{(1/\omega )}(\rr d)$, respectively, 
through the form $(\cdo ,\cdo )_{L^2(\rr d)}$.
\end{enumerate}
\end{prop}

\par

The next result concerns embedding
properties.

\par

\begin{prop}
\label{Prop:EmbModSp}
Let $p,p_j,q,q_j\in (0,\infty ]$, $j=1,2$,
be such that $p_1\le p_2$ and $q_1\le q_2$,
$\omega ,\omega _j\in \mascP _E(\rr {2d})$,
$j=0,1,2$, be such that
$\omega _0(x,\xi )=\omega (-\xi ,x)$ and
$\phi \in \Sigma _1(\rr d)\setminus 0$.
Then the following is true:
\begin{enumerate}
\item $\mascF$ on $\Sigma _1 '(\rr d)$
restricts to a homeomorphism from 
$M^{p,q}_{(\omega )}(\rr d)$ to
$W^{q,p}_{(\omega _0)}(\rr d)$;

\vrum

\item the inclusions
\begin{align}
\Sigma _1 (\rr d) &\hookrightarrow
M^{p,q}_{(\omega )}(\rr d),
W^{p,q}_{(\omega )}(\rr d)
\hookrightarrow \Sigma _1 '(\rr d)
\label{Eq:ModGSEmbeddings1}
\intertext{and}
M^{p_1,q_1}_{(\omega _1)}(\rr d)
&\hookrightarrow
M^{p_2,q_2}_{(\omega _2)}(\rr d),
\quad \text{when}\quad
\omega _2\lesssim \omega _1,
\end{align}
are continuous.
If in addition $p,p_2,q,q_2<\infty$,
then these inclusions are dense;

\vrum

\item if
$$
v_{r}(x,\xi )
=
e^{r(|x|+|\xi |)}
\quad \text{and}\quad
v_{0,r}=(1+|x|+|\xi |)^r,
$$
then
\begin{equation}\label{Eq:UnionIntersectMod}
\begin{alignedat}{2}
\Sigma _1(\rr d)
&=
\bigcap _{r>0}M^{p,q}_{(v_{r})}(\rr d),&
\Sigma _1'(\rr d)
&=
\bigcup _{r>0}M^{p,q}_{(1/v_{r})}(\rr d),
\\[1ex]
\mascS (\rr d)
&=
\bigcap _{r>0}M^{p,q}_{(v_{0,r})}(\rr d)&
\quad \text{and}\quad
\mascS '(\rr d)
&=
\bigcup _{r>0}M^{p,q}_{(1/v_{0,r})}(\rr d).
\end{alignedat}
\end{equation}
\end{enumerate}
\end{prop}

\par

The next proposition deals with
complex interpolation properties
for modulation spaces.

\par

\begin{prop}
\label{Prop:ComplIntepolMod}
Let $p_j,q_j\in [1,\infty ]$, $j=0,1,2$,
be such that
$$
\frac 1{p_0}
=
\frac {1-\theta}{p_1}+\frac \theta{p_2}
\quad \text{and}\quad
\frac 1{q_0}
=
\frac {1-\theta}{q_1}+\frac \theta{q_2},
$$
for some $\theta \in [0,1]$,
and let
$\omega \in \mascP _E(\rr {2d})$.
Then
$$
(M^{p_1,q_1}_{(\omega )}(\rr d),
M^{p_2,q_2}_{(\omega )}(\rr d))_{[\theta]}
=
M^{p_0,q_0}_{(\omega )}(\rr d),
$$
with equivalent norms.
\end{prop}

\par

Propositions
\ref{Prop:InvModAndBanach},
\ref{Prop:ModDuality} and
\ref{Prop:ComplIntepolMod},
as well as essential parts of Proposition
\ref{Prop:EmbModSp} follow from
\cite{Fei1,Fei6,FeiGro1,FeiGro2,GaSa,Gro2}
and the analyses therein.
The identities in
\eqref{Eq:UnionIntersectMod}
are essentially special cases of
\cite[Theorem 3.9]{Toft10}, see also
\cite{GroZim, Teofanov2}.
For these reasons we omit the 
proofs of Propositions
\ref{Prop:InvModAndBanach}--\ref{Prop:ComplIntepolMod}.
Here we remark that the topologies of the spaces
on the left-hand sides of
\eqref{Eq:UnionIntersectMod}
are obtained by replacing each intersection
and union,
by projective limit topology and
inductive limit topology, respectively,
with respect to $r>0$.

\par

We remark that for the Banach space
case, $p,q\ge 1$, Propositions
\ref{Prop:InvModAndBanach}--\ref{Prop:EmbModSp},
can be proved
by straight-forward applications
of Fourier analysis in combination of
classical estimates in harmonic analysis
(see e.{\,}g. \cite[Chapter 11]{Gro2}).
In order to reach the general result
(allowing $p$ and $q$ to be
smaller than one), one usually combine
these techniques with Gabor analysis
(see e.{\,}g. \cite{GaSa}). Here we remark
that Proposition \ref{Prop:InvModAndBanach}
and Proposition \ref{Prop:EmbModSp}
will be extended in several ways in the
next two sections, without any use of Gabor
analysis.

\par

\subsection{Invariant quasi-Banach 
function spaces}

\par


\par

The functional
$\nm \cdo{\mascB}$, defined on the
vector space $\mascB$ is called a
\emph{quasi-norm} of order
$r _0 \in (0,1]$, if the conditions
\begin{alignat}{2}
\nm {f+g}{\mascB}
&\le
2^{\frac 1 {r _0} -1}
(\nm {f}{\mascB} + \nm {g}{\mascB}),
& \quad f,g &\in \mascB ,
\label{Eq:WeakTriangle1}
\\[1ex]
\nm {\alpha \cdot f}{\mascB}
&=
|\alpha| \cdot \nm f{\mascB}\ge 0,
& \quad \alpha &\in \mathbf{C},
\quad  f \in \mascB
\notag
\intertext{and}
\nm f{\mascB} &= 0
\quad  \Leftrightarrow \quad
f=0, & &
\notag
\end{alignat}
hold true. A vector space which contains 
a quasi-norm (of order $r_0$)
is called a quasi-norm space (of order 
$r_0$). A complete quasi-norm space
(of order $r_0$) is called a
quasi-Banach space (of order $r_0$).
If $\mascB$ is a quasi-Banach space 
with quasi-norm satisfying 
\eqref{Eq:WeakTriangle1}
then by \cite{Aik,Rol} there is an 
equivalent quasi-norm to $\nm \cdo 
{\mascB}$ of order $r_0$, which 
additionally satisfies
\begin{align}\label{Eq:WeakTriangle2}
\nm {f+g}{\mascB} ^{r _0} \le \nm {f}
{\mascB} ^{r _0} + \nm {g}{\mascB} 
^{r _0}, 
\quad f,g \in \mascB .
\end{align}
From now on we always assume that 
the  quasi-norm of the quasi-Banach 
space  $\mascB$ is chosen in such a 
way that both 
\eqref{Eq:WeakTriangle1} and 
\eqref{Eq:WeakTriangle2} hold.

\par

The quasi-Banach
space $\mascB$ is called a
\emph{quasi-Banach function
space} (or \emph{QBF space}) on
$\rr d$, whenever $\mascB$ is
contained in
$\maclM (\rr d)$.
It is here
understood that the zero element
in $\maclM (\rr d)$ (which is
an equivalent class), is the zero
element in $\mascB$.

\par

We always assume that the QBF spaces
are \emph{solid}, which is described
in the following definition.
(See e.{\,}g. (A.3) in \cite{Fe1992}.)

\par

\begin{defn}
\label{Def:SolidBF}
Let $\mascB$ be a
\emph{QBF space} on
$\rr d$.
Then $\mascB$ is
called \emph{solid}, if the
following conditions hold:
\begin{enumerate}
\item
for any
$f,g\in \maclM (\rr d)$
which satisfy
$g\in \mascB$ and $|f|
\le |g|$, then $f\in \mascB$ and
\begin{equation}\label{Eq:Solid}
\nm f{\mascB}\le C\nm g{\mascB},
\end{equation}
for some constant $C>0$ which is
independent of $f$ and $g$ above;

\vrum

\item
for any $x\in \rr d$,
there is a constant $c>0$ and a function
$f\in \mascB$ such that $|f|>c$ on a convex
set in $\rr d$ which is not a zero set,
and which contains $x$.
\end{enumerate}
\end{defn}

\par

In the following definition
we put needed restrictions
on QBF spaces.

\par

\begin{defn}\label{Def:BFSpaces}
Let $\mascB$ be a solid QBF space of
order $r_0\in (0,1]$ on $\rr d$,
$C>0$ be the same as in 
\eqref{Eq:Solid},
and let $v_0\in \mascP _E(\rr d)$
be submultiplicative.
\begin{enumerate}
\item The space $\mascB$
is called a
\emph{solid translation invariant
quasi-Banach function space on
$\rr d$} 
(with respect to $r_0$ and $v_0$), or 
\emph{invariant QBF space on $\rr d$},
if for any $x\in \rr d$ and
$f\in \mascB$, then $f(\cdo -x)\in
\mascB$, and 
\begin{equation}\label{translmultprop1}
\nm {f(\cdo -x)}{\mascB}
\le
Cv_0(x)\nm {f}{\mascB}\text ;
\end{equation}

\vrum

\item The space $\mascB$
is called a \emph{solid translation 
invariant Banach function space on
$\rr d$} 
(with respect to $v_0$), or 
\emph{invariant BF space on $\rr d$},
if $\mascB$ is a Banach space and
an invariant QBF space on $\rr d$
with respect to $v_0$.

\vrum

\item The space $\mascB$
is called a \emph{solid translation 
invariant normal quasi-Banach function 
space on $\rr d$} (with respect to $r_0$ and $v_0$), or 
\emph{a normal QBF space on $\rr d$},
if there is an invariant BF space
$\mascB _0$ on $\rr d$
with respect to $v_0^{r_0}$ such that
\begin{equation}
\label{Eq:NormQBFNorm}
\nm f{\mascB}
\equiv
\nm {\, |f|^{r_0}\, }{\mascB _0}^{1/r_0},
\quad
f\in \mascB .
\end{equation}
\end{enumerate}
\end{defn}

\par

Evidently, $\mascB$ in
Definition \ref{Def:BFSpaces} (3),
is an invariant QBF space with respect
to $r_0$ and $v_0$.

\par

We observe that it suffices to verify
condition (2) in Definition \ref{Def:SolidBF} at
one point $x$, when $\mascB$ is an invariant
QBF space in $\rr d$. We also observe that if $\mascB$
is an invariant QBF space, then any characteristic
function of a bounded measurable in $\rr d$ belongs
to $\mascB$.

\par

Suppose that
$\omega \in \mascP _E(\rr d)$, and 
$\mascB$ is a quasi-Banach function
space on $\rr d$ with respect
to $r_0\in (0,1]$ and
$v_0\in \mascP _E(\rr d)$. Then we
let
\begin{equation}\label{Eq:BomegaDef}
\mascB _{(\omega )}=\mascB (\omega)
\equiv
\sets {f\in \maclM (\rr d)}
{f\cdot \omega \in \mascB},
\end{equation}
and equip the space in
\eqref{Eq:BomegaDef} with the
quasi-norm
$$
\nm f{\mascB _{(\omega )}}
\equiv
\nm {f\cdot \omega}{\mascB},
\qquad
f\in \mascB _{(\omega )}.
$$
For any $f\in \maclM (\rr d)
\setminus \mascB _{(\omega )}$, we put
$\nm f{\mascB}=\infty$.
%
%
%

\par

\begin{rem}\label{Rem:WeightedQBF}
Suppose that
$\omega ,v\in \mascP _E(\rr d)$
are chosen such that $\omega$
is $v$-moderate.
By straight-forward computations
it follows that the following
is true:
\begin{enumerate}
\item if $\mascB$ is an invariant
QBF space with respect to
$r_0\in (0,1]$ and
$v_0\in \mascP _E(\rr d)$, then
$\mascB _{(\omega )}$ is an invariant
QBF space with respect to
$r_0$ and $v_0v$;

\vrum

\item if $\mascB$ is an invariant
BF space with respect to
$v_0\in \mascP _E(\rr d)$, then
$\mascB _{(\omega )}$ is an invariant
BF space with respect to $v_0v$;

\vrum

\item if $\mascB$ is a normal invariant
QBF space with respect to
$r_0\in (0,1]$ and
$v_0\in \mascP _E(\rr d)$, then
$\mascB _{(\omega )}$ is a normal
invariant QBF space with respect to
$r_0$ and $v_0v$.
\end{enumerate}
\end{rem}

\par

In view of Remark \ref{Rem:WeightedQBF}, 
it follows that the
families of invariant QBF spaces
in Definition \ref{Def:BFSpaces}
are not enlarged by including
weighted versions as in
\eqref{Eq:BomegaDef}. On the other hand,
in some situations, it might be
convenient to include weights in the
definition of such spaces, see e.{\,}g. 
\cite{PfTo19}.

\par


\par

Every invariant BF space $\mascB$ is 
continuously embedded in
$\Sigma _1'(\rr d)$
and the map $(f,\fy )\mapsto f*\fy$
is well-defined and continuous from 
$\mascB \times L^1_{(v_0)}(\rr d)$
to $\mascB$. 
For the invariant BF space
$\mascB \subseteq L^1_{loc}(\rr d)$ with respect to $v_0$ we have
Minkowski's inequality, i.{\,}e.,
\begin{equation}\label{Eq:MinkIneq}
\nm {f*\fy}{\mascB}\le C \nm {f}{\mascB}\nm \fy{L^1_{(v_0)}},
\qquad f\in \mascB ,\ \fy \in L^1_{(v_0)} (\rr d),
\end{equation}
for some $C>0$ which is independent of
$f\in \mascB$ and $\fy \in L^1_{(v_0)} (\rr d)$, 
see e.g. \cite{PfTo19}.

\par

For any $\mascB$ in Definition
\ref{Def:BFSpaces}, the corresponding
sequence space,
$\ell _{\mascB} = \ell _{\mascB}(\zz d)$
is the set of all sequences
$a = \{ a(j) \} _{j\in \zz d}\subseteq
\mathbf C$ such that
\begin{equation} \label{eq:seq-space-norm}
\nm a{\ell _{\mascB}}
\equiv
\NM
{\sum _{j\in \zz d}a(j)\chi _{j+Q}}
{\mascB},
\qquad Q=[0,1]^d.    
\end{equation}
Here $\chi _E$ is the characteristic
function of the set $E\subseteq \rr d$.

\par

\begin{example}
\label{Example:LebCase}
Let $d=d_1+d_2$. A common choice of $\mascB$
in Definition \ref{Def:BFSpaces}
is given by $L ^{p,q}_{(\omega)} (\rr {d})$
or $L ^{p,q}_{*,(\omega)} (\rr {d})$ in
\eqref{Eq:MixLebSpace1} and
\eqref{Eq:MixLebSpace2}.

\par

Let $v\in \mascP _E(\rr {d})$ be
chosen such that $\omega$ is 
$v$-moderate,
$$
r _0= \min(1,p,q),
\quad
p_0=\frac{p}{r_0}
\quad \text{and}\quad
q_0=\frac{q}{r_0}.
$$
Then
$L ^{p,q}_{(\omega)} (\rr {d})$
and 
$L^{p,q}_{*,(\omega )} (\rr {d})$
are
normal QBF spaces on $\rr {d}$ with 
respect to 
$v$, $r _0$ and 
$\mascB _0
=
L^{p_0, q_0}_{(\omega _{0})}
(\rr {d})$ 
respectively
$\mascB _0
=
L^{p_0,q_0}_{\ast ,(\omega _{0})}
(\rr {d})$, $\omega _{0}=\omega ^{r_0}$.

\par

Let $\Omega \subseteq \rr d$ be Lebesgue measurable.
As usual we let $L^{p,q}_{(\omega )}(\Omega )$
($L^{p,q}_{*,(\omega )}(\Omega )$)
be the set of all $f\in \maclM (\Omega )$ such that
$$
f_\Omega (x)
=
\begin{cases}
f(x), & x\in \Omega ,
\\[1ex]
0, & x\in \rr d \setminus \Omega ,
\end{cases}
$$
belongs to $L^{p,q}_{(\omega )}(\rr d)$
($L^{p,q}_{*,(\omega )}(\rr d)$). The corresponding
norm is given by
$$
\nm f{L^{p,q}_{(\omega )}(\Omega )}
\equiv 
\nm {f_\Omega}{L^{p,q}_{(\omega )}(\rr d )}
\qquad
\big (
\nm f{L^{p,q}_{*,(\omega )}(\Omega )}
\equiv 
\nm {f_\Omega}{L^{p,q}_{*,(\omega )}(\rr d )}
\big ).
$$

\par

In similar ways we let
$\ell ^{p,q}_{(\omega )}(\zz {d})$
and
$\ell ^{p,q}_{*,(\omega )}(\zz {d})$, 
$d=d_1 + d_2$ with $d_1, d_2 \in \mathbb N$,
be the set of all
$c\in \ell _0'(\zz {d})$
such that
$$
\nm {c}{\ell _{(\omega )}^{p,q}}<\infty 
\quad \text{and}\quad
\nm {c}{\ell _{*,(\omega )}^{p,q}}
<\infty,
$$
respectively. Here
\begin{alignat*}{3}
\nm {c}
{\ell _{(\omega )}^{p,q}(\zz {d})}
&\equiv
\nm {G_{c,\omega ,p}}{\ell ^q(\zz {d_2})}, &
\quad 
G_{c,\omega, p}(\iota)
&=
\nm {c(\cdo ,\iota )\omega (\cdo ,\iota)}
{\ell ^p(\zz {d_1})}, & \quad
\iota &\in \zz {d_2}, 
\intertext{and}
\nm {c}
{\ell ^{p,q}_{*,(\omega )}(\zz {d})}
&\equiv
\nm {H_{c,\omega ,q}}{\ell ^p(\zz {d_1})}, &
\quad 
H_{c,\omega, q}(j)
&=
\nm {c(j,\cdo )\omega (j,\cdo )}
{\ell ^q(\zz {d_2})}, & \quad
j&\in \zz {d_1},
\end{alignat*}
and $\ell _0'(\zz d)$
denotes the set of all 
formal (complex-valued) sequences 
$c=\{ c(j)\} _{j\in \zz d}$ on $\zz d$.
In similar ways as above, we put 
$$
\ell ^{p,q}= \ell ^{p,q}_{(\omega )}
\quad \text{and}\quad
\ell ^{p,q}_*= \ell ^{p,q}_{*,(\omega )}
\quad \text{when}\quad
\omega \equiv 1,
$$
and also put $\ell ^p_{(\omega )}=\ell ^{p,p}_{(\omega )}$, which lead to
$\ell ^p=\ell ^{p,p}$.

\par

By letting $\mascB$ be equal to
$L^{p}(\rr {d})$, $L^{p,q} (\rr {d})$, 
or $L ^{p,q} _\ast (\rr {d})$
in \eqref{eq:seq-space-norm}, we obtain
$$
\ell _ {L ^p_{(\omega)} (\rr {d})}
=
\ell ^p_{(\omega)}(\zz d),
\quad
\ell _ {L ^{p,q} _{(\omega)} (\rr {d})} 
=
\ell ^{p,q} _{(\omega)}(\zz {d})
\quad
\text{and}
\quad
\ell _ {L ^{p,q} _{\ast, (\omega)}
(\rr {d})}
=
\ell ^{p,q} _{\ast, (\omega)}(\zz {d}),
$$
with equality in norms.
%
%
\end{example}

\par

Let $\mascB$ be an invariant
QBF space on $\rr d$ with
respect to $r_0\in (0,1]$
and $v_0\in \mascP _E(\rr d)$.
Then the (discrete) 
convolution $*$, formally given by
$$ 
a*b = \sum_{j \in \zz d}
a(\cdot -j) b(j), 
\qquad 
a,b \in \ell _0' (\zz d),
$$
is well-defined on
$\ell _{(v_0)}^{r_0}(\zz d)
\times
\ell _{\mascB}(\zz d)
$
and satisfies
\begin{equation}\label{Eq:DiscConvEst}
\nm {a*b}{\ell _{\mascB}}
\le
C
\nm a{\ell _{(v_0)}^{r_0}}
\nm b{\ell _{\mascB}},
\qquad
a\in \ell _{(v_0)}^{r_0}(\zz d),
\,
b\in \ell _{\mascB}(\zz d),
\end{equation}
where the constant $C$ in
\eqref{Eq:DiscConvEst} only depends
on the constant $C$ in
\eqref{translmultprop1} and the
dimension $d$.

\par

\subsection{Wiener amalgam spaces and a broad class of modulation spaces}
We first define Wiener amalgam spaces.

\par

\begin{defn}
\label{Def:WienAm}
Let $\mascB$ be an invariant
QBF space on $\rr d$ with
respect to $r_0\in (0,1]$
and $v_0\in \mascP _E(\rr d)$.
Also let $\omega \in \mascP _E(\rr d)$,
$d,d_1,d_2\in \mathbf N$
be such that $d=d_1+d_2>0$, and let $r_1,r_2\in (0,\infty ]$.
Then the norm
$\nm f{\sfW ^{r_1,r_2}(\omega ,\mascB)}$
of $f\in \maclM (\rr d)$ is given
by
\begin{equation}
\label{Eq:WienerQuasiNormsSimpleExt}
\begin{aligned}
\nm f{\sfW ^{r_1,r_2}_{d_1,d_2}(\omega ,\mascB)}
&=
\nm f{\sfW ^{r_1,r_2}_{d_1,d_2}(\omega ,\ell _{\mascB})}
\equiv
\nm {a}{\ell _{\mascB}},\quad
\text{where}
\\[1ex]
a(j)
&\equiv
\nm {f \, \omega }{L^{r_1,r_2} (j+Q)},
\\
Q &=[0,1] ^d,\ 
j=(j_1,j_2),\ j_k\in \zz {d_k},\ k=1,2.
\end{aligned}
\end{equation}
The Wiener amalgam space
\begin{equation}
\label{Wiener-notation-1}
\sfW ^{r_1,r_2}_{d_1,d_2}(\omega ,\mascB )
=
\sfW ^{r_1,r_2}_{d_1,d_2}(\omega ,\ell _{\mascB}
(\zz {d}))
\end{equation}
consists of all $f\in \maclM (\rr d)$
such that
$\nm f{\sfW ^{r_1,r_2}_{d_1,d_2}(\omega ,\mascB)}$ is 
finite.

\par

The norm $\nm \cdo{\sfW ^{r_1,r_2}_{d_1,d_2,*}
(\omega ,\mascB)}$ and space
$$
\sfW ^{r_1,r_2}_{d_1,d_2,*}(\omega ,\mascB )
=
\sfW ^{r_1,r_2}_{d_1,d_2,*}(\omega ,\ell _{\mascB}
(\zz {d}))
$$
are obtained by replacing $a(j)$ in
\eqref{Eq:WienerQuasiNormsSimpleExt} as
$$
a(j)
\equiv
\nm {f \, \omega }{L^{r_1,r_2}_* (j+Q)}.
$$
\end{defn}

\par

To shorten
the notation we set
\begin{align}
\sfW ^{r_1,r_2}_{d_1,d_2}(\mascB )
&=
\sfW ^{r_1,r_2}_{d_1,d_2}(\omega ,\ell _{\mascB})
\quad \text{when}\quad
\omega \equiv 1,
\label{Eq:WienerNot2}
\\[1ex]
\sfW ^{r_1,r_2}_{d_1,d_1}&=\sfW ^{r_1,r_2},
\quad \text{and}\quad
\sfW ^{r_1,r_1}=\sfW ^{r_1},
\notag
\end{align}
in Definition \ref{Def:WienAm}.

\par

We observe that
$\sfW ^r(\omega ,\ell ^{p})$ is often 
denoted by
$W(L^r,L^p_{(\omega )})$ or 
$W(L^r, \ell ^p_{(\omega ) })$ in
the literature (see e.{\,}g. 
\cite{Fe1992,FeiGro1,GaSa,Rau1}).
The space $\sfW ^r(\omega ,\ell _\mascB )$,
when $\mascB$ is equal to
$L^{p} (\rr {d})$, 
$L^{p,q} (\rr {2d})$ or
$L^{p,q} _\ast (\rr {2d})$
are special cases of so-called
coorbit spaces,
cf. \cite{FeiGro1,FeiGro2,Rau2}.

\par

\begin{rem}\label{Rem:WienerEmb}
Obviously,
$\sfW ^{r_1,r_2}(\omega ,\ell _\mascB )$ 
decreases with $r_1$ and $r_2$,
because 
for any compact set $K \subseteq \rr d$,
$$
\nm f{L^{r_1,r_2}(K)}
\lesssim
\nm f{L^{r_3,r_4}(K)},
\quad r_1\le r_3
\quad \text{and}\quad r_2\le r_4.
$$
In particular,
\begin{equation}
\label{Eq:WienerEmbTwoInd}
\sfW ^{\max (r_1,r_2)}(\omega ,\ell _\mascB )
\hookrightarrow
\sfW ^{r_1,r_2}(\omega ,\ell _\mascB )
\hookrightarrow
\sfW ^{\min (r_1,r_2)}(\omega ,\ell _\mascB ).
\end{equation}
Furthermore,
$\sfW ^{r_1,r_2}(\omega ,\ell ^{p,q})$ 
and 
$\sfW ^{r_1,r_2}(\omega ,\ell ^{p,q} _\ast)$
increase with $p,q$. We have
\begin{alignat}{1}
\sfW ^\infty (\omega ,\ell ^{p,q})
&\hookrightarrow
L^{p,q}_{(\omega )}(\rr {2d}) \cap \Sigma _1'(\rr {2d})
\hookrightarrow
L^{p,q}_{(\omega )}(\rr {2d})
\hookrightarrow
\sfW ^r(\omega ,\ell ^{p,q})
\intertext{and}
\nm \cdo {\sfW ^r(\omega ,\ell ^{p,q})}
&\le
\nm \cdo {L^{p,q}_{(\omega )}}
\le
\nm \cdo {\sfW  ^\infty (\omega ,\ell ^{p,q})},
\qquad
r\le\min (1,p,q),
\end{alignat}
which follows by straight-forward
applications of H{\"o}lder's inequality.
\end{rem}

\par

By Lemma 2.5 and Theorem 2.6 in 
\cite{Rau1} we get the following.

\par

\begin{prop}\label{Prop:ComplWienSpaces}
Suppose that $r_1,r_2\in (0,\infty ]$, $d,d_1,d_2\in \mathbf N$
is such that $d=d_1+d_2>0$, $\mascB$
is an invariant QBF space on $\rr d$
of order $r_0\in (0,1]$,
and let $\omega \in \mascP _E(\rr d)$.
Then the following is true:
\begin{enumerate}
\item $\ell _{\mascB}(\zz d)$ is a quasi-Banach
space of order $r_0$;

\vrum

\item $\sfW ^{r_1,r_2}_{d_1,d_2}(\omega , \mascB )$
is a quasi-Banach space of order $\min (r_0,r_1,r_2)$.
\end{enumerate}
\end{prop}

\par

\begin{rem}
\label{Rem:EquivWienerNorms}
Let $r_1,r_2\in (0,\infty ]$,
$d_1,d_2\in \mathbf N$ be such that $d=d_1+d_2>0$,
$\omega \in \mascP _E(\rr d)$,
$\mascB$ be an invariant QBF space on
$\rr d$,
$\Lambda \subseteq \rr d$
be a lattice, and let $\Omega
\subseteq \rr d$ be a bounded
convex set such that
$$
\bigcup _{j\in \Lambda}(j+\Omega)
=\rr d.
$$
Also let $a_{0,\Omega}
=
\{ a_{0,\Omega}(j) \} _{j\in \Omega}$,
$$
a_{0,\Omega}(f,j)\equiv
\nm {f\, \omega }{L^{r_1,r_2}_{d_1,d_2}(j+\Omega )},
\quad
j\in \Lambda ,
$$
and consider the quasi-norm
\begin{equation}\label{Eq:NormAltern}
\nmm f = \NM {\sum _{j\in \Lambda}
a_{0,\Omega }(f,j)\chi _{j+\Omega }}
{\mascB},
\end{equation}
when $f \in \maclM (\rr d)$.
Then $\sfW ^{r_1,r_2}_{d_1,d_2}(\omega ,\mascB)$
consists of all measurable $f \in \maclM (\rr d)$
such that $\nmm f$
is finite, and the quasi-norm
in \eqref{Eq:NormAltern} is equivalent
to $\nm \cdo
{\sfW ^{r_1,r_2}_{d_1,d_2}(\omega ,\mascB)}$.
These properties follow by straight-forward
computations, using the fact
that $\omega $ is moderate
(see e.{\,}g. \cite{Fei1981,Fei1}).
\end{rem}

\par

Next we define a broad class of 
modulation spaces.

\par

\begin{defn}\label{Def:GenModSpace}
Let $\mascB$ be a normal QBF space on $\rr {2d}$, 
$\omega \in\mascP _E(\rr {2d})$,
and 
$\phi \in \Sigma _1(\rr d)\setminus 0$. 
Then the \emph{modulation space} $M(\omega ,\mascB )$ consists
of all $f\in \Sigma _1'(\rr d)$ such that
$$
\nm f{M(\omega ,\mascB )}
\equiv \nm {V_\phi f\cdo \omega }{\mascB} <\infty .
$$
\end{defn}

\par

\begin{example}
Let
$
\mascB =L^{p,q}(\rr {2d})
$ 
or
$
\mascB =L_{\ast}^{p,q}(\rr {2d}).
$ 
$p,q \in (0,\infty]$. If
$\omega \in \mascP _E(\rr {2d})$,
then 
$
M(\omega ,\mascB )
$
in Definition \ref{Def:GenModSpace}
becomes the classical modulation
space
$
M^{p,q}_{(\omega )}(\rr d),
$
and
$
W^{p,q}_{(\omega )}(\rr d),
$
respectively,
see Definition \ref{Def:ClassModSp}.
\end{example}

\par

Note that $M(\omega ,\mascB )$ possess
several convenient properties that are
true for the classical modulation spaces. 
For example, the definition of
$M(\omega ,\mascB )$ is independent
of the choice of
$\phi \in \Sigma _1(\rr d)\setminus 0$ 
when $\mascB$ is a Banach space (see 
e.{\,}g. \cite{FeiGro1}).
Additionally in
\cite[Proposition 2.10]{PfTo19}
it is proved that 
$M(\omega ,\mascB )$ is independent
of the choice of
$\phi \in M ^1_{(v_0 v)}(\rr d)\setminus 0$.

\par

In the  next section we show that
such properties in broader forms
are extendable to the case when
$\mascB$ is a normal QBF space.
We refer to 
\cite{Fei1, Fei6, FeiGro1, FeiGro2, FG4, GaSa, 
Gro2, GroZim, RSTT, Teofanov2,
Teofanov3,Toft10}
for more information about modulation 
spaces.

\par 

\section{Some properties of normal QBF spaces}
\label{sec2}

\par

In this section we show some
properties of $\mascB$,
$\ell _{\mascB}(\zz d)$ and
$\sfW ^r(\omega ,\mascB )$
when $\mascB$ is a normal
QBF space.

\par

We start with the following,
which explain embedding properties
between $\mascB _{(\omega )}$ and
$\ell _{\mascB}(\zz d)$ on one hand,
and classical Wiener spaces and
discrete Lebesgue spaces on the other
hand.

\par

\begin{prop}\label{Prop:BasicEmbWSp}
Let $\mascB$ be a normal QBF space
on $\rr d$ with respect to
$r_0\in (0,1]$ and
$v_0\in \mascP _E(\rr d)$, and let
$\omega \in \mascP _E(\rr d)$.
Then the following is true:
\begin{enumerate}
\item $\ell ^{r_0}_{(v_0)}(\zz d)
\hookrightarrow
\ell _{\mascB}(\zz d)
\hookrightarrow
\ell ^{\infty}_{(1/v_0)}(\zz d)$,
and
\begin{equation}\label{Eq:SeqIneq}
C^{-1}\nm a{\ell ^\infty _{(1/v_0)}}
\le
\nm a{\ell _{\mascB}}
\le
C\nm a{\ell ^{r_0}_{(v_0)}},
\qquad
a\in \ell _0'(\zz d),
\end{equation}
where the constant $C$ only depends
on $d$ and the constant
in \eqref{translmultprop1};

\vrum

\item $\sfW ^\infty
(\omega ,\mascB)
\hookrightarrow
\mascB _{(\omega )}
\hookrightarrow
\sfW ^{r_0}
(\omega ,\mascB )$,
and
\begin{equation}\label{Eq:WSpIneq}
C^{-1}\nm f{\sfW ^{r_0}
(\omega ,\mascB)}
\le
\nm f{\mascB _{(\omega )}}
\le
C\nm f{\sfW ^{\infty}
(\omega ,\mascB)},
\qquad
f\in \maclM (\rr d),
\end{equation}
where the constant $C$ only depends
on $\omega$, $d$, and the constant
in \eqref{translmultprop1}.
\end{enumerate}
\end{prop}

\par

\begin{proof}
Let
$$
Q=[0,1]^d
\quad \text{and}\quad
Q_r=[-r,r]^d,
\quad r\ge 0.
$$
The assertion (1) follows if we
prove \eqref{Eq:SeqIneq}. 
Therefore,
suppose that
$a\in \ell ^{\infty}_{(1/v_0)}(\zz d)$.
Then
\begin{align*}
\nm a{\ell _{\mascB}}^{r_0}
&=
\NM {\sum _{j\in \zz d}a(j)\chi _{j+Q}}
{\mascB}^{r_0}
\le
\sum _{j\in \zz d}
|a(j)|^{r_0}
\nm {\chi _{j+Q}}{\mascB}^{r_0}
\\[1ex]
&\le
C\sum _{j\in \zz d}
|a(j)v_0(j)|^{r_0}
\nm {\chi _{Q}}{\mascB}^{r_0}
\asymp
C\sum _{j\in \zz d}
|a(j)v_0(j)|^{r_0},
\end{align*}
and the second inequality in
\eqref{Eq:SeqIneq} follows.

\par

We also have
\begin{align*}
\nm a{\ell ^\infty _{(1/v_0)}}
&\asymp
\sup _{j\in \zz d}
\left (
|a(j)/v_0(j)|\nm {\chi _Q}{\mascB}
\right )
\le
C\sup _{j\in \zz d}
\left (
\nm {|a(j)|\chi _{j+Q}}{\mascB}
\right )
\\[1ex]
&\lesssim
C
\NM {\sum _{j\in \zz d}
|a(j)|\chi _{j+Q}}{\mascB}
=
C\nm a{\ell _{\mascB}},
\end{align*}
giving the first inequality in
\eqref{Eq:SeqIneq}, and (1)
follows.

\par

The assertion (2) follows if we
prove \eqref{Eq:WSpIneq}. 
We observe that the map
$f\mapsto f\, \omega$ is homeomorphic
from $\mascB _{(\omega )}$ to $\mascB$,
and from $\sfW ^r(\omega ,\mascB)$
to $\sfW ^r(\mascB)$, for every
$r\in (0,\infty ]$. This reduces
ourselves to the case when $\omega \equiv 1$.

\par

Since
$$
|f(x)| \le \sum _{j\in \zz d}
\nm {f}
{L^\infty (j+Q)}\chi _{j+Q}(x),
$$
we obtain
\begin{equation*}
\nm f{\mascB}
\lesssim
\NM {\sum _{j\in \zz d}
\nm {f}{L^\infty (j+Q)}
\chi _{j+Q}}
{\mascB}
=
\nm f{\sfW ^{\infty}
(\mascB)},
\end{equation*}
where the last inequality follows
from the fact that $\mascB$ is solid.
This gives the second inequality in
\eqref{Eq:WSpIneq}.

\par

It remains to prove the first
inequality in \eqref{Eq:WSpIneq}.
By solid properties of the involved
spaces, it is no restriction to
assume that $f\ge 0$.
First we consider the case when
$r_0=1$.
We have
$$
\nm f{\sfW ^{r_0}
(\mascB)}
=
\nm h{\mascB},
$$
where
\begin{align*}
h(x)
&=
\sum _{j\in \zz d}
\nm f{L^1(j+Q)}\chi _{j+Q}(x)
\\[1ex]
&=
\sum _{j\in \zz d}
\left (
\int _{Q} f(y+j)
\chi _{Q_2}(x-y-j)\, dy
\right )
\chi _{Q}(x-j).
\end{align*}
Here we have taken $y-j$ as new
variables of integrations, and
used the fact that
$\chi _{Q_2}(x-y-j)=1$ when
$y\in Q$ and $x-j\in Q$. By
taking $y+j-x$ as new variables
of integrations, and using the fact
that $\chi _{Q_2}$ is even,
we obtain
\begin{align*}
h(x)
&\le
\sum _{j\in \zz d}
\left (
\int _{j-x+Q} f(x+y )
\chi _{Q_2}(y)\, dy
\right )
\chi _{Q}(x-j)
\\[1ex]
&\le
\sum _{j\in \zz d}
\left (
\int _{Q_1} f(x+y )
\chi _{Q_2}(y)\, dy
\right )
\chi _{Q}(x -j),
\end{align*}
where the last inequality follows
from the fact that $y\in Q_1$ when
$y\in j-x+Q$ and $x-j\in Q$.
Since $\mascB$ is solid, we obtain
$$
\nm f{\sfW ^{r_0}(\mascB)}
=
\nm h{\mascB}
\lesssim
\NM {\sum _{j\in \zz d}
\left (
\int _{Q_3} f(\cdo +y )
\chi _{Q_2}(y)\, dy
\right )
\chi _{Q}(\cdo -j)}{\mascB}.
$$
An application of Minkowski's
inequality now gives
\begin{align*}
\nm f{\sfW ^{r_0}
(\mascB)}
&\le
\int _{Q_3}
\NM {\sum _{j\in \zz d}
f(\cdo +y )
\chi _{Q_2}(y)
\chi _{Q}(\cdo -j)}{\mascB}\, dy
\\[1ex]
&\asymp
\int _{Q_3}
\NM {\sum _{j\in \zz d}
f(\cdo +y)\cdot 
\chi _{Q}(\cdo -j)}{\mascB}
\chi _{Q_2}(y) \, dy
\\[1ex]
&\asymp
\int _{Q_3}
\nm {f}{\mascB}
\chi _{Q_2}(y) \, dy
\asymp \nm {f}{\mascB}.
\end{align*}
This gives the first
inequality in \eqref{Eq:WSpIneq},
and the result follows in the case
$r_0=1$.

\par

Next suppose that $r_0\in (0,1]$
is arbitrary, and let $\mascB _0$
be the same as in Definition
\ref{Def:BFSpaces}. Then
\begin{align*}
\nm f{\sfW ^{r_0}
(\mascB)}^{r_0}
&=
\NM {\sum _{j \in \zz d}\nm f{L^{r_0}(j+Q)}
\chi _{j+Q}}{\mascB}^{r_0}
=
\NM {\sum _{j \in \zz d}\nm f{L^{r_0}(j+Q)}^{r_0}
\chi _{j+Q}}{\mascB _0}
\\[1ex]
&=
\nm {|f|^{r_0}}{\sfW ^1
(\mascB _0)}
\lesssim
\nm {|f|^{r_0}}{\mascB _0}
=
\nm f{\mascB}^{r_0},
\end{align*}
where the inequality follows from
the previous part of the proof.
This gives the first
inequality in \eqref{Eq:WSpIneq},
and the result follows.
%
%
%
%
\end{proof}

\par

As a consequence of 
Proposition \ref{Prop:BasicEmbWSp} we get

\par

\begin{cor}\label{Cor:BasicEmbWSp}
Let $\mascB$ be a normal QBF space
on $\rr d$ with respect to
$r_0\in (0,1]$ and
$v_0\in \mascP _E(\rr d)$, and
let $\omega \in \mascP _E(\rr d)$.
Then
$$
\sfW ^\infty (\omega v_0,\ell ^{r_0}
(\zz d))
\hookrightarrow
\mascB _{(\omega )}
\hookrightarrow
\sfW ^{r_0}(\omega /v_0,\ell ^{\infty}
(\zz d)),
$$
and
\begin{equation}\label{Eq:WSpIneq2}
\begin{aligned}
C^{-1}\nm f
{\sfW ^{r_0}(\omega /v_0,\ell ^{\infty})}
&\le
\nm f{\mascB _{(\omega )}}
\le
C\nm f
{\sfW ^\infty (\omega v_0,\ell ^{r_0})},
\\[1ex]
f
&\in
\sfW ^{r_0}(\omega/v_0,\ell ^{\infty}
(\zz d)),
\end{aligned}
\end{equation}
where the constant $C$ only depends
on $d$, $\omega$ and the constant
in \eqref{translmultprop1}.
\end{cor}

\par

\begin{proof}
If $\omega _1,\omega _2
\in \mascP _E(\rr d)$ and
$p,q\in (0,\infty ]$, then the facts
that $\omega _j$, $ j = 1,2$, are moderate imply
\begin{equation}
\begin{aligned}
\sfW ^p(\omega _1\omega _2 ,
\ell ^q(\zz d))
&=
\sfW ^p(\omega _1 ,
\ell ^q_{(\omega _2)}(\zz d))
\\[1ex]
&=
\sfW ^p(\omega _2 ,
\ell ^q_{(\omega _1)}(\zz d))
=
\sfW ^p(1,
\ell ^q_{(\omega _1\omega _2)}(\zz d)),
\end{aligned}
\end{equation}
with equivalent quasi-norms.
By combining these identities with
Proposition \ref{Prop:BasicEmbWSp},
we obtain
\begin{align*}
\sfW ^\infty (\omega v_0,\ell ^{r_0}
(\zz d))
&=
\sfW ^\infty (\omega ,
\ell _{(v_0)}^{r_0}
(\zz d))
\hookrightarrow
\sfW ^\infty (\omega ,
\ell _{\mascB}(\zz d))
\hookrightarrow
\mascB _{(\omega )}
\\[1ex]
&\hookrightarrow
\sfW ^{r_0} (\omega ,
\ell _{\mascB}(\zz d))
\hookrightarrow
\sfW ^{r_0} (\omega ,
\ell _{(1/v_0)}^{\infty}(\zz d))
\\[1ex]
&=
\sfW ^{r_0} (\omega /v_0,
\ell ^{\infty}(\zz d)).\qedhere
\end{align*}
\end{proof}

\par

\begin{cor} \label{cor:relation-to-Sigma}
Let $\mascB$ be a normal QBF space
on $\rr d$, $r\in (0,\infty ]$
and $\omega \in \mascP _E(\rr {d})$.
Then
\begin{alignat}{3}
\Sigma _1(\rr d)
&\subseteq
\sfW ^r(\omega ,\mascB)
\nsubseteq
\Sigma _1'(\rr d),&
\quad &\text{when}& \quad
r&<1,
\label{Eq:EmbWSpaceToDistSp1}
\intertext{and}
\Sigma _1(\rr d)
&\subseteq
\sfW ^r(\omega ,\mascB)
\subseteq
\Sigma _1'(\rr d),&
\quad &\text{when}& \quad
r&\ge 1,
\label{Eq:EmbWSpaceToDistSp2}
\end{alignat}
with continuous inclusions.
\end{cor}

\par

\begin{proof}
By \eqref{Eq:weight1}, \eqref{Eq:GSFtransfChar} and 
the definition of the Wiener amalgam space
it follows that $\Sigma _1(\rr d)$
is continuously embedded in
$\sfW ^r(\omega ,\mascB)$, which gives
left inclusions in
\eqref{Eq:EmbWSpaceToDistSp1}
and
\eqref{Eq:EmbWSpaceToDistSp2}.

\par

If $r<1$, $Q=(0,1)^d$, and 
$f(x)=x_1^{-1}\chi _Q(x)$, $x=(x_1, \ldots, x_d)\in \rr d$,
then we also have
$$
f\in \sfW ^r(\omega ,\mascB )\setminus
\Sigma _1'(\rr d),
$$
which gives 
\eqref{Eq:EmbWSpaceToDistSp1}.

\par

It remains to prove the second
inclusion in 
\eqref{Eq:EmbWSpaceToDistSp2}.
Let $v_\rho (x) = e^{\rho |x|}$,
$x\in \rr d$ and $\rho >0$. Then
$L^1_{(1/v_\rho )}(\rr d)
\hookrightarrow
\Sigma _1'(\rr d)$ (see
\eqref{Eq:vrhoEmb}). 
Also let
$v\in \mascP _E$ be chosen such that
$\omega$ is $v$-moderate,
$v_0$ be as in Definition 
\ref{Def:BFSpaces},
and redefine $Q$ as $Q=[0,1]^d$. For any
$f\in L^1_{(v_\rho )}(\rr d)$,
we have
\begin{align*}
\nm f{L^1_{(1/v_\rho )}}
&=
\nm {f/v_\rho }{L^1}
\lesssim 
\nm {f\omega v/v_\rho }{L^1}
\\[1ex]
&\le
\sum _{j\in \zz d}
\nm {f\omega v/v_\rho }{L^1(j+Q)}
\\[1ex]
&=
\sum _{j\in \zz d}
\nm {(f\omega v\eabs \cdo ^{d+1}
/v_\rho )\eabs \cdo ^{-(d+1)}}{L^1(j+Q)}
\\[1ex]
&\le 
\nm {\eabs \cdo ^{-(d+1)}}{\ell ^1}
\sup _{j\in \zz d}
\left (
\nm {f\omega v\eabs \cdo ^{d+1}
/v_\rho }{L^1(j+Q)}
\right )
\\[1ex]
&\lesssim
\sup _{j\in \zz d}
\left (
\nm {f\omega /v_0}{L^1(j+Q)}
\right ),
\end{align*}
provided $\rho$ is chosen large
enough such that
$$
v_0v\eabs \cdo ^{d+1}
\lesssim
v_\rho .
$$
An application of Proposition
\ref{Prop:BasicEmbWSp} (1) now gives
\begin{align*}
\nm f{L^1_{(1/v_\rho )}}
&\lesssim
\NM { \{ \nm {f\omega }{L^1(j+Q)}\}
_{j\in \zz d} }{\ell ^\infty (1/v_0)}
\\[1ex]
&\lesssim
\NM { \{ \nm {f\omega }{L^1(j+Q)}\}
_{j\in \zz d} }{\ell _{\mascB}}
\asymp
\nm f{\sfW ^1(\omega ,\mascB )}.
\end{align*}
A combination of these estimates and
\eqref{Eq:vrhoEmb} gives
$$
\sfW ^1(\omega ,\mascB )
\hookrightarrow
L^1_{(1/v_\rho )}(\rr d)
\hookrightarrow
\Sigma _1'(\rr d),
$$
which gives the second inclusion
in \eqref{Eq:EmbWSpaceToDistSp2},
and thereby the result.
\end{proof}

\par

\section{Equivalence of norms in modulation
and Wiener amalgam spaces}\label{sec3}

\par

In this section we deduce norm
equivalences
for general classes of
modulation spaces.
Especially we extend
\cite[Proposition 2$'$]{Toft2022} by 
relaxing conditions on window functions 
and by allowing more general classes of modulation
and Wiener amalgam spaces  $ M(\omega ,\mascB )$
and $\sfW ^r (\omega , \mascB )$, respectively.
The arguments are different compared to 
\cite{Toft2022}.
For related results, when $ \mascB=L^{p,q} (\rr {2d})$
respectively $\mascB=L ^{p,q}_{\ast} 
(\rr {2d}$),  $r=\infty$, and $p,q\in [1,\infty ]$,
see \cite{Gro2},
and for the same choice of $\mascB$
and $r$, and  $p,q\in (0,\infty ]$, see
\cite{GaSa,Toft13}. Here observe that 
the analysis in \cite{GaSa} is restricted
to weights which are moderated
by polynomially bounded submultiplicative
functions. 
%
%

\par

The main result in the section is the following. Here and in what
follows we set $\sfW ^{r_1,r_2} (\omega , \mascB)
=
\sfW ^{r_1,r_2}_{d,d} (\omega , \mascB)$, when $r_1,r_2\in (0,\infty ]$,
$\omega \in \mascP _E(\rr {2d})$ and $\mascB$ is an invariant QBF space
on $\rr {2d}$.

\par

\begin{thm}\label{Thm:EquivNorms2}
Let $\omega ,v,v_0\in \mascP _E(\rr 
{2d})$ be such that
$\omega$ is $v$-moderate, 
$\mascB$ be a normal QBF space
on $\rr {2d}$ with respect to $r_0\in (0,1]$ and $v_0$, $r_1,r_2\in [r_0,\infty ]$,
and let
$\phi _1,\phi _2\in M^{r_0}_{(v_0v)}
(\rr d)\setminus 0$.
Then
\begin{equation}\label{Eq:EquivNorms21}
f\in M(\omega ,\mascB )
\quad \Leftrightarrow \quad
V_{\phi _1}f\cdo \omega \in \mascB
\quad \Leftrightarrow \quad
V_{\phi _2}f\in \sfW ^{r_1,r_2}
(\omega , \mascB ),
\end{equation}
and
\begin{equation}\label{Eq:EquivNorms22}
\nm f{M(\omega ,\mascB )}
\asymp
\nm {V_{\phi _1}f\cdo \omega }{\mascB}
\asymp
\nm {V_{\phi _2}f}
{\sfW ^{r_1,r_2} (\omega , \mascB)}.
\end{equation}
\end{thm}

\par

Note that \eqref{Eq:EquivNorms22}
gives
\begin{equation}\label{Eq:EquivNorms23}
\nm {V_{\phi _2}f}
{\sfW ^{r_1} (\omega , \mascB)}
\asymp
\nm {V_{\phi _2}f}
{\sfW ^{r_2} (\omega , \mascB)},
\quad
r_1,r_2\in (0,\infty ].
\end{equation}

\par

In order to prove
Theorem \ref{Thm:EquivNorms2}, we
first show that the result holds
true when $r_1$ and $r_2$ belong to
the larger interval $(0,\infty ]$
instead of $[r_0,\infty ]$,
when $\phi _1(x)$ and $\phi _2(x)$
are equal to the (standard) Gaussian
$\phi _0(x) = \pi ^{-\frac d4}
e^{-\frac 12\cdot |x|^2}$.

\par

\begin{prop}\label{Prop:EquivNorms2}
Let $\omega ,v_0\in \mascP _E(\rr 
{2d})$ be such that $v_0$
is submultiplicative,
$\mascB$ be a normal QBF space
on $\rr {2d}$ with respect to $r_0\in 
(0,1]$ and $v_0$, $r_1,r_2\in (0,\infty ]$,
and let
$\phi _0(x)=\pi ^{-\frac d4}
e^{-\frac 12\cdot |x|^2}$.
Then
\begin{equation}\label{Eq:EquivNorms21G}
f\in M(\omega ,\mascB )
\quad \Leftrightarrow \quad
V_{\phi _0}f\omega \in \mascB
\quad \Leftrightarrow \quad
V_{\phi _0}f\in \sfW ^{r_1,r_2}
(\omega , \mascB ),
\end{equation}
and
\begin{equation}\label{Eq:EquivNorms22G}
\nm f{M(\omega ,\mascB )}
\asymp
\nm {V_{\phi _0}f\cdot \omega }{\mascB}
\asymp
\nm {V_{\phi _0}f}
{\sfW ^{r_1,r_2} (\omega , \mascB)}.
\end{equation}
\end{prop}

\par


\par

\par

\begin{proof}
Due to \eqref{Eq:WienerEmbTwoInd}
we reduce ourself to the case $r_1=r_2=r$.

\par

First we prove
\eqref{Eq:EquivNorms23} when $\phi _2 = \phi _0$.
Since
$$
\nm {V_{\phi _0}f}
{\sfW ^{r} (\omega , \mascB)}
\le \nm {V_{\phi _0}f}
{\sfW ^{\infty}
(\omega , \mascB)},
$$
it suffices to prove that
\begin{equation}\label{Eq:RevIneq1}
\nm {V_{\phi _0}f}
{\sfW ^{\infty}
(\omega , \mascB)}
\lesssim
\nm {V_{\phi _0}f}
{\sfW ^{r} (\omega , \mascB)},
\quad
r\in (0,\infty ].
\end{equation}
We shall essentially follow the
first part of the proof of
\cite[Propostion 2$'$]{Toft2022} to
show \eqref{Eq:RevIneq1}.

\par

Let 
$F_0=|V_{\phi _0} f\cdot \omega |$,
$Q=[0,1]^{2d}$, and $Q_r=[-r,r]^{2d}$, $r>0$.
By Lemma \ref{Lemma:EstSTFT}
and the fact that $\omega$ is
moderate it follows that there is an
$X_j=(x_j,\xi _j)\in j+Q$,
$j\in \zz {2d}$,
such that
\begin{align*}
\nm {F_0}{L^\infty (j+Q)}
&\asymp
\nm {V_{\phi _0}f}{L^\infty (j+Q)}
\omega (j)
\\[1ex]
&=
|{V_{\phi _0}f}(X_j)|\omega (j)
\lesssim
\nm {V_{\phi _0}f}{L^r (B_1(X_j))}
\omega (j)
\\[1ex]
&\le
\nm {V_{\phi _0}f}{L^r (j+Q_2)}
\omega (j)
\asymp
\nm {F_0}{L^r (j+Q_2)}.
\end{align*}
From this estimate and Remark
\ref{Rem:EquivWienerNorms} we get
\begin{align*}
\nm {V_{\phi _0}f}
{\sfW ^\infty (\omega ,\mascB)}
&=
\nm { \{ \nm {F_0}
{L^\infty (j+Q)}\}
_{j\in \zz {2d}}}{\ell _{\mascB}}
\\[1ex]
&\lesssim
\nm { \{ \nm {F_0}
{L^r (j+Q_2)}\}
_{j\in \zz {2d}}}{\ell _{\mascB}}
\\[1ex]
&\asymp
\nm { \{ \nm {F_0}
{L^r (j+Q)}\}
_{j\in \zz {2d} }}{\ell _{\mascB}}
=
\nm {V _{\phi _0} f}
{\sfW ^r(\omega ,\mascB)},
\end{align*}
and \eqref{Eq:RevIneq1} follows.

\par

Since 
$\nm {F_0}{\mascB}
\le
\nm {V_{\phi _0}f}
{\sfW ^\infty
(\omega , \mascB)}$ (see Proposition \ref{Prop:BasicEmbWSp}),
the result follows 
if we prove
that
\begin{equation}\label{Eq:RevIneq2}
\nm {V_{\phi _0}f}
{\sfW ^\infty
(\omega , \mascB)}
\lesssim
\nm {F_0}{\mascB}.
\end{equation}

\par

By Lemma \ref{Lemma:EstSTFT} we have for $X \in \rr {2d}$
\begin{align}
\nm {F_0}{L^\infty (j+Q)}^{r_0}
\chi _{j+Q}(X)
\notag
&\asymp
\nm {V_{\phi _0}f}{L^\infty (j+Q)}^{r_0}
\omega (j) ^{r_0}\chi _{j+Q}(X)
\notag
\\[1ex]
&\lesssim
\left (
\int _{Q_1}F_0(X_j+Y)^{r_0}\, dY
\right )
\chi _{j+Q}(X)
\notag
\\[1ex]
&=
\left (
\int _{X_j+Q_1}F_0(Y)^{r_0}\, dY
\right )
\chi _{j+Q}(X)
\notag
\\[1ex]
&\le
\left (
\int _{X+Q_3}F_0(Y)^{r_0}\, dY
\right )
\chi _{j+Q}(X)
\notag
\\[1ex]
&\le
\left (
\int _{Q_3}F_0(Y+X)^{r_0}\, dY
\right )
\chi _{j+Q}(X).
\label{Eq:F0LNormEst}
\end{align}
In the second inequality from
the end, we have used the facts that
$F_0\ge 0$, and that
$$
X_j+Q_1\subseteq X+Q_3
\quad \text{when}\quad
X\in j+Q.
$$

\par

Let $\mascB _0$ be an invariant
BF space, related to $\mascB$ as
in Definition \ref{Def:BFSpaces} (3).
By \eqref{Eq:F0LNormEst} we get
\begin{align*}
\nm {V_{\phi _0}f}
{\sfW ^\infty
(\omega , \mascB)}^{r_0}
&=
\NM { \sum _{j\in \zz {2d}}
\nm {F_0}{L^\infty (j+Q)}\chi _{j+Q}}
{\mascB}^{r_0}
\\[1ex]
&=
\NM { \sum _{j\in \zz {2d}}
\nm {F_0}{L^\infty (j+Q)}^{r_0}
\chi _{j+Q}}
{\mascB _0}
\\[1ex]
&\lesssim
\NM { \sum _{j\in \zz {2d}}
\left (
\int _{Q_3}F_0(Y+\cdo )^{r_0}\, dY
\right )
\chi _{j+Q}}
{\mascB _0}
\\[1ex]
&\le
\int _{Q_3}
\NM { \sum _{j\in \zz {2d}}
\left (
F_0(Y+\cdo )^{r_0}
\right )
\chi _{j+Q}}
{\mascB _0}\, dY
\\[1ex]
&\lesssim
\int _{Q_3}
\NM {F_0(Y+\cdo )^{r_0}}
{\mascB _0}\, dY
\\[1ex]
&\lesssim
\int _{Q_3}
\NM {F_0^{r_0}}
{\mascB _0}\, dY
\asymp
\NM {F_0}
{\mascB}^{r_0},
\end{align*}
and \eqref{Eq:RevIneq2} follows.
\end{proof}

\par

\begin{proof}[Proof of Theorem
\ref{Thm:EquivNorms2}]
By \eqref{Eq:WienerEmbTwoInd}, we may assume
that $r_1=r_2=r$, $r \in [r_0, \infty]$.
Let $\phi _0(x)=\pi ^{-\frac d4}e^{-\frac 12|x|^2}$,
$x\in \rr d$, be the standard Gaussian,
$\phi \in M^{r_0}_{(v _0 v)}(\rr d)\setminus 0$ be arbitrary,
and let $f\in \Sigma _1'(\rr d)$ be fixed. Using \eqref{Eq:WienerEmbTwoInd}
again, the result follows if we prove
\begin{equation}
\label{Eq:ProofDiagram}
\begin{matrix}
V_{\phi _0}f \in \sfW ^{\infty}(\omega ,\mascB ) &
\!\! \overset{(1)}{\phantom i\Leftrightarrow \phantom i}\!\! &
V_{\phi _0}f \in \mascB _{(\omega )} &
\!\! \overset{(2)}{\phantom i\Leftrightarrow \phantom i}\!\! &
V_{\phi _0}f \in \sfW ^{r_0}(\omega ,\mascB )
\\[1ex]
{\text{\scriptsize{$(3)$}}}
\Updownarrow & & & &
\Updownarrow
{\text{\scriptsize{$(4)$}}}
\\[1ex]
V_{\phi}f \in \sfW ^{\infty}(\omega ,\mascB ) &
\!\! \underset{(5)}{\phantom i\Rightarrow \phantom i}\!\! &
V_{\phi}f \in \mascB _{(\omega )} &
\!\! \underset{(6)}{\phantom i\Rightarrow \phantom i}\!\! &
V_{\phi}f \in \sfW ^{r_0}(\omega ,\mascB ).
\end{matrix}
\end{equation}

\par

By Propositions \ref{Prop:BasicEmbWSp} and \ref{Prop:EquivNorms2}
it follows that (1), (2), (5) and (6) in \eqref{Eq:ProofDiagram} hold.
It remain to prove (3) and (4) in \eqref{Eq:ProofDiagram}.

\par

We set
\begin{equation}\label{Eq:F0FDef}
\hspace{-1mm}
F=|V_\phi f\cdo \omega |,
\quad
F_0=|V_{\phi _0} f\cdo \omega |,
\quad
\Phi _1 = |V_{\phi _0}\phi \cdo v|,
\quad
\Phi _2 = |V_\phi \phi _0 \cdo v| ,
\end{equation}
\begin{alignat*}{2}
Q&=[0,1]^{2d},&\quad Q (j)&=j+Q,
\quad
\\[1ex]
Q_1&=
[-1,1]^{2d},&\quad
Q_1 (j )&=j+Q_1,
\\[1ex]
\alpha _{n,s}(j)& = \nm {\Phi _n}{L^s(Q 
(j))}, & &
\quad n = 1,2,
\\[1ex]
\beta _{0,s}(j )
&=
\nm {F_0}{L^s(Q (j ))}, &
\quad
\beta _{s}(j) &= \nm F{L^s(Q (j))},
\\[1ex]
\tilde \beta _{0,s}(j ) &=
\nm {F_0}{L^s(Q_1 (j))},&
\quad \text{and}\quad
\tilde \beta _{s}(j)
&=
\nm F{L^s(Q_1 (j ))},
\quad
j \in \zz {2d}, \ 
s \in (0,\infty].
\end{alignat*}
Since $|V_\phi \phi _0 (X)|
=
|V_{\phi _0} \phi (-X)|$, it follows that
\begin{equation}\label{Eq:TildeAlfaIdent}
\alpha _{1,s}(j)
\asymp
\alpha _{2,s}(-j).
\end{equation}

\par

We intend to prove that
\begin{gather}
\nm {V_{\phi }f}
{\sfW ^r (\omega , \mascB )}
\le
C\nm \phi{M^{r_0}_{(v_0v)}}
\nm {V_{\phi _0}f}
{\sfW ^r (\omega , \mascB )}
\label{Eq:EquivNorms2A}
\intertext{and}
\nm {V_{\phi _0}f}
{\sfW ^r (\omega , \mascB )}
\le
C\big (\nm \phi{M^{r_0}_{(v_0v)}}
\big )^{-1}
\left (
\frac {\nm {\phi}
{M^{r_0}_{(v_0v)}}}
{\nm \phi{L^2}}
\right )^{\! \! \vartheta (r)}
\!\!
\nm {V_{\phi }f}
{\sfW ^r (\omega , \mascB )},
\label{Eq:EquivNorms2B}
\end{gather}
for
$\vartheta(r) = \max(\frac 2r, 2)$ and
some constant $C\ge 1$ which is
independent of $f\in M(\omega ,\mascB)$
and $\phi \in M^{r_0}_{(v_0 v)}(\rr d)$.
This gives (3) and (4) in \eqref{Eq:ProofDiagram},
as special case, and thereby the result.

%

\par

By \eqref{Eq:TildeAlfaIdent}, Remark 
\ref{Rem:EquivWienerNorms} and
Proposition \ref{Prop:EquivNorms2} we obtain
\begin{align}
\nm {\alpha _{n,s}}
{\ell ^{r_0}_{( v _0)}}
&=
\nm {\Phi _n}
{\sfW ^s (\ell ^{r_0} _{( v _0)} ) }
\asymp
\nm \phi{M^{r_0}_{(v_0  v)}}, \quad n=1,2,
\label{NormAlpha} \\[1ex]
\nm {\tilde{\beta} _{0,s}}{\ell _\mascB}
&\asymp
\nm {\beta _{0,s}}{\ell _\mascB}
=
\nm {F_0}{\sfW ^s (\mascB )},
\notag
\intertext{and}
\nm {\tilde{\beta} _{s}}{\ell _\mascB}
&\asymp
\nm {\beta _{s}}{\ell_\mascB}
=
\nm F{\sfW ^s (\mascB )},
\qquad
s \in (0,\infty] .
\notag
\end{align}
It follows from
\eqref{Eq:change-of-window} that
$F \le C _0 F _0 \ast \Phi_2$, where
$C_0=\nm{\phi_0}{L ^2} ^{-2}$.
A combination of these
relations, \eqref{Eq:DiscConvEst} and 
\eqref{NormAlpha} gives
\begin{align*}
\nm {F} {\sfW ^{\infty} (\mascB )}^{r_0}
&= 
\nm {\beta _{\infty}}
{\ell _{\mascB}}^{r_0}
\\[1ex]
&\le
C C _0 ^{r_0} \nm { \alpha_{2,\infty}*\tilde{\beta}_{0,\infty} }
{\ell _{\mascB}}^{r_0} 
\le
C C_0 ^{r_0}
\nm { \alpha_{2,\infty} }
{\ell _{(v_0)}^{r_0}}^{r_0}
\nm { \tilde{\beta}_{0,\infty}  }
{\ell _{\mascB}}^{r_0}
\\[1ex]
&=
C C_0 ^{r_0}
\nm {\Phi_2}
{\sfW ^{\infty}
(\ell _{(v_0)}^{r_0})}^{r_0} 
\nm {F_0} {\sfW ^{\infty}(\mascB )}^{r_0}
\\[1ex]
&\le
C C_0 ^{r_0}
\nm { \phi } {M ^{r_0} _{(v_0v)}}^{r_0}
\nm { F _0} {\sfW ^{\infty}
(\mascB )}^{r_0} 
\end{align*}
for some constant $C\ge 1$ which is
independent of $f\in M(\omega ,\mascB)$
and $\phi \in M^{r_0}_{(v_0 v)}(\rr d)$.

\par 

A combination of the latter estimates
and Proposition \ref{Prop:EquivNorms2}
gives
\begin{equation} \label{eq:estimates-for-(2.10)}
\nm {F} {\sfW ^{r} (\mascB )}
\le
\nm {F} {\sfW ^{\infty} (\mascB )}
\lesssim
\nm { \phi } {M ^{r_0} _{(v_0v)}}
\nm { F _0} {\sfW ^{\infty}
(\mascB )}
\asymp
\nm { \phi } {M ^{r_0} _{(v_0v)}}
\nm { F _0} {\sfW ^{r} (\mascB )},
\end{equation}
and \eqref{Eq:EquivNorms2A} follows.

\par

Next we prove \eqref{Eq:EquivNorms2B}.
First we consider the case when
$r\in [r_0,1]$.
We redefine $C_0$ as $C_0=\nm \phi{L^2}^{-2r_0}$.
Then the inequality $F_0 \lesssim C_0^{1/r_0 } F 
\ast \Phi_1$ (cf. \eqref{Eq:change-of-window}) gives
\begin{align*}
\beta_{0,r} (j)^r
&\lesssim
C_0^{\frac r{r_0}}\int_{j+Q}
\left ( 
\int _{\rr {2d}} F(X-Y) \, 
\Phi _1 (Y) \, dY
\right )^r
\, dX
\nonumber
\\[1ex]
&=
C_0^{\frac r{r_0}}\int _{j+Q}
\left ( 
\sum_{k\in \zz {2d}}
\int_{k+Q} F(X-Y) \, 
\Phi_1 (Y) \, dY
\right ) ^r
\, dX
\nonumber
\\[1ex]
&\lesssim
C_0^{\frac r{r_0}}\sum_{k\in \zz {2d}} 
\alpha_{1,\infty} (k)^r
\gamma _r(j-k)^r ,
\end{align*}
where
$$
(\gamma _r(j))^{r}
=
\left (
\int _{j+Q} \left( 
\int _{Q} F(X-Y)  
\, dY
\right )^r \, dX
\right ).
$$
Since $\frac {r_0}r\le 1$, we obtain
$$
\beta_{0,r} (j)^{r_0}
\lesssim
C_0\left (
(\alpha _{1,\infty}^r*\gamma _r^r)(j)
\right )^{\frac {r_0}r}
\le
C_0(\alpha _{1,\infty}^{r_0}
*\gamma _r^{r_0})(j).
$$

\par

Let $\mascB _0$ be the invariant BF space of 
Definition \ref{Def:BFSpaces} related to $\mascB$.
By applying the $\ell _{\mascB}$ norm
and using
\eqref{Eq:DiscConvEst}, we obtain
\begin{equation}\label{Eq:beta0rEst}
\begin{aligned}
\nm {\beta _{0,r}}{\ell _{\mascB}}^{r_0}
&=
\nm {\beta _{0,r}^{r_0}}
{\ell _{\mascB _0}}
\le
C_0
\nm {\alpha _{1,\infty}^{r_0}
*\gamma _r^{r_0}}
{\ell _{\mascB _0}}
\\[1ex]
&\lesssim C_0
\nm {\alpha _{1,\infty}^{r_0}}
{\ell ^1_{(v_0)}}\nm {\gamma _r^{r_0}}
{\ell _{\mascB _0}}
=
C_0\nm {\alpha _{1,\infty}}
{\ell ^{r_0}_{(v_0)}}^{r_0}
\nm {\gamma _r ^{r_0} }{\ell _{\mascB _0}}.
\end{aligned}
\end{equation}
We need to estimate $\nm {\gamma _r ^{r_0}}
{\ell _{\mascB _0 }}$.
Let $\theta >0$ be arbitrary.
By 
$
F=F^r \, F^{1-r}
$
we have
\begin{align*}
\gamma _r (j)^{r_0}
&=
\left (
\int_{j+Q}
\left( 
\int_{Q} F(X-Y)  \, dY
\right )^r
\, dX
\right )^{\frac {r_0}r}
\\[1ex]
&\le 
\tilde{\beta}_{\infty}(j) ^{r_0(1-r)}
\left (
\int_{j+Q}
\left( 
\int_{Q} F(X-Y)^r
\, dY 
\right)^r
\, dX 
\right )^{\frac {r_0}r}
\\[1ex]
&\qquad 
\leq
\tilde{\beta}_{r}(j) ^{r_0r}
\tilde{\beta}_{\infty}(j) ^{r_0(1-r)}
\\[1ex]
&\qquad 
\leq
r\theta ^{\frac 1r}\tilde{\beta} _{r}(j)^{r_0}
+
(1-r)\theta ^{-\frac 1{1-r}}
\tilde{\beta}_{\infty}(j)^{r_0},
\end{align*}
where the last inequality follows
from the arithmetic-geometric
inequality.

\par

By using these estimate in
\eqref{Eq:beta0rEst}, we obtain
\begin{align*}
\nm {F_0}{\sfW ^r (\mascB )}^{r_0}
&=
\nm {\beta _{0,r}}{\ell _{\mascB}}^{r_0}
\\[1ex]
&\lesssim
C_0
\nm {\alpha_{1,\infty}}
{\ell ^{r_0}_{(v_0)}}^{r_0}
\left ( 
r\theta ^{\frac 1r}
\nm {\tilde{\beta} _{r}^{r_0}}
{\ell _{\mascB _0}}
+
(1-r)\theta ^{-\frac 1{1-r}}
\nm {\tilde{\beta}  _{\infty}^{r_0}}
{\ell _{\mascB _0}}
\right )
\\[1ex]
&=
C_0
\nm {\alpha_{1,\infty}}
{\ell ^{r_0}_{(v_0)}}^{r_0}
\left ( 
r\theta ^{\frac 1r}
\nm {\tilde{\beta} _{r}}
{\ell _{\mascB}}^{r_0}
+
(1-r)\theta ^{-\frac 1{1-r}}
\nm {\tilde{\beta}  _{\infty}}
{\ell _{\mascB}}^{r_0}
\right ).
\end{align*}
A combination of the latter
inequality with
\eqref{NormAlpha} gives
\begin{equation}\label{Eq:F0FEst}
\nm {F_0}{\sfW ^r (\mascB )}^{r_0}
\lesssim
C_0
\nm {\phi}
{M^{r_0}_{(v_0v)}}^{r_0}
\!\!
\left ( 
r\theta ^{\frac 1r}
\nm {F}{\sfW ^r (\mascB )}^{r_0}
+
\frac {1-r}{\theta ^{1/(1-r)}}
\nm {F}{\sfW ^\infty (\mascB )}^{r_0}
\right ).
\end{equation}

\par

If $r=1$ and $\theta =1$, then 
\eqref{Eq:F0FEst} gives 
\eqref{Eq:EquivNorms2B}, and the
result follows in this case.

\par

If instead $r\in [r_0,1)$,
then by \eqref{eq:estimates-for-(2.10)} it follows that
\begin{equation}\label{eq:EstF_0}
\nm {F}{\sfW ^\infty (\mascB )}
\le
C_1\nm {\phi}{M^{r_0}_{(v_0v)}}
\nm {F_0}{\sfW ^\infty (\mascB )}
\le
C_2\nm {\phi}{M^{r_0}_{(v_0v)}}
\nm {F_0}{\sfW ^r (\mascB )},
\end{equation}
for some constants $C_1$ and
$C_2$, which only
depend on $r$ and the dimension $d$.
A combination of \eqref{Eq:F0FEst}
and \eqref{eq:EstF_0} gives
$$
\nm {F_0}{\sfW ^r (\mascB )}^{r_0}
\le
C_0C
\nm {\phi}
{M^{r_0}_{(v_0v)}}^{r_0}
\left ( 
r\theta ^{\frac 1r}
\nm {F}{\sfW ^r (\mascB )}^{r_0}
+
\left(
\frac {1-r} {\theta ^{\frac 1{1-r}} } 
\right)
\nm {\phi}{M^{r_0}_{(v_0v)}}^{r_0}
\nm {F_0}{\sfW ^r (\mascB )}^{r_0}
\right ).
$$

\par

By choosing $\theta$ as
$$
\theta ^{\frac 1{1-r}}
=
2C_0C(1-r)
\nm \phi{M^{r_0}_{(v_0v)}}^{2r_0},
$$
we obtain
$$
\nm {F_0}{\sfW ^r (\mascB )}^{r_0}
\le
C_0C
r\theta ^{\frac 1r}
\nm {\phi}
{M^{r_0}_{(v_0v)}}^{r_0}
\nm {F}{\sfW ^r (\mascB )}^{r_0}
+
\frac 12
\nm {F_0}{\sfW ^r (\mascB )}^{r_0},
$$
which is the same as
$$
\nm {F_0}{\sfW ^r (\mascB )}^{r_0}
\le
2C_0C
r\theta ^{\frac 1r}
\nm {\phi}
{M^{r_0}_{(v_0v)}}^{r_0}
\nm {F}{\sfW ^r (\mascB )}^{r_0}.
$$
We also use
$$
\theta ^{\frac 1r} =
(2C_0C(1-r))^{(1-r)/r}
\nm \phi{M^{r_0}_{(v_0v)}}^{2r_0(1-r)/r}
\quad \text{and}\quad
C_0=\nm \phi{L^2}^{-2r_0}
$$
in the last inequality to get \eqref{Eq:EquivNorms2B} in the case 
$r\in [r_0,1)$.

\par

For $r\in (1,\infty]$,
Proposition \ref{Prop:EquivNorms2}, 
\eqref{Eq:EquivNorms2B}
and the fact that the Wiener amalgam spaces 
$\sfW ^r (\mascB )$ decrease with $r$ 
gives
$$
\nm {F_0}{\sfW ^r (\mascB )}
\asymp
\nm {F_0}{\sfW ^1 (\mascB )}
\lesssim
\nm F{\sfW ^1 (\mascB )}
\le
\nm {F}{\sfW ^r (\mascB )},
$$
giving the result.
\end{proof}

\par

\begin{rem}
Let $\mascB$ and $\omega$
be the same as in Theorem
\ref{Thm:EquivNorms2},
$\vartheta$ be as in
\eqref{Eq:EquivNorms2B},
$\phi _1,\phi _2\in M^{r_0}_{(v_0v)}
(\rr d)\setminus 0$, and let
$r_1,r_2\in [r_0,\infty ]$. Then
a combination of 
\eqref{Eq:EquivNorms2A}
and
\eqref{Eq:EquivNorms2B}
shows that
\begin{equation*}
\nm {V_{\phi _1}f}
{\sfW ^{r_1} (\omega ,\mascB )}
\le
C\left (
\frac {\nm {\phi _1}
{M^{r_0}_{(v_0v)}}}
{\nm {\phi _2}
{M^{r_0}_{(v_0v)}}}
\right )
\left (
\frac {\nm {\phi _2}
{M^{r_0}_{(v_0v)}}}{\nm {\phi _2}{L^2}}
\right )^{\!\! \vartheta (r_2)}
\!\!
\nm {V_{\phi _2} f}
{\sfW ^{r_2} (\omega ,\mascB )},
\end{equation*}
for every $f\in M(\omega ,\mascB )$,
where the constant $C$ only depends on
$\mascB$, $r_1$, $r_2$ and $d$.

\par

Here we observe that the factor
$$
\left (
\frac {\nm {\phi}
{M^{r_0}_{(v_0v)}}}
{\nm \phi{L^2}}
\right )^{\vartheta (r)}
$$
in \eqref{Eq:EquivNorms2B} is bounded
from below by a positive constant,
because $M^{r_0}_{(v_0v)}(\rr d)$
is continuously embedded in
$M^2(\rr d)=L^2(\rr d)$.
A question which appear is whether
\eqref{Eq:EquivNorms2B} is true
when this factor is omitted.
\end{rem}

\par

Theorem \ref{Thm:EquivNorms2} can also be
formulated in the following way.


\par

\begin{prop}
Suppose that
$\omega ,\omega _0,v_0
\in
\mascP _E(\rr {2d})$,
$\mascB$ is a normal QBF space
on $\rr {2d}$ with respect to
$r_0\in (0,1]$,
$r\in [r_0,\infty ]$,
and that
$f\in \Sigma _1'(\rr d)\setminus 0$.
Then the following conditions are equivalent:
\begin{enumerate}
\item $f\in M(\omega ,\mascB )$;

\vrum

\item for some $\phi \in M^{r_0}_{(v_0v)}(\rr d)\setminus 0$,
it holds $V_\phi f \cdot \omega \in \mascB$;

\vrum

\item for some $\phi \in M^{r_0}_{(v_0v)}(\rr d)\setminus 0$,
it holds $V_\phi f \in \sfW ^r(\omega ,\mascB )$;

\vrum

\item for every $\phi \in M^{r_0}_{(v_0v)}(\rr d) \setminus 0$,
it holds $V_\phi f \cdot \omega \in \mascB$;

\vrum

\item for every $\phi \in M^{r_0}_{(v_0v)}(\rr d) \setminus 0$,
it holds $V_\phi f \in \sfW ^r(\omega ,\mascB )$.
\end{enumerate}
\end{prop}

\par

\begin{rem}
Evidently, due to Example
\ref{Example:LebCase}, it follows that
all results in this section, as well as
in the remaining sections, are applicable
to classical modulation spaces 
$M^{p,q}_{(\omega)} (\rr d)$ and
$W^{p,q}_{(\omega)} (\rr d)$,
when
$p,q\in [r_0,\infty ]$, $r_0\in (0,1]$,
$\omega \in \mascP _E(\rr {2d})$.
\end{rem}

\par

\begin{rem}
\label{Rem:ProgressNormSpaces}
As remarked earlier, there are several 
contributions in the literature,
which are related to Theorem
\ref{Thm:EquivNorms2}. Several of these 
contributions deal with more general 
situations, e.{\,}g. in \cite{FeiGro1}
by Feichtinger and Gr{\"o}chenig, where the 
(basic) coorbit space theory (for Banach 
spaces) is established. In \cite{Rau2}, 
Rauhut extends the theory to allow suitable 
quasi-Banach spaces which are not Banach
spaces. Roughly 
speaking, the coorbit spaces are defined in
similar ways as for modulation spaces
in Definition \ref{Def:GenModSpace},
after that $L^2(\rr d)$ is replaced by
more general Hilbert spaces and the
phase-shifts $f\mapsto
f(\cdo -x)e^{i\scal \cdo \xi}$ are replaced
by more general group representations.
In this context, modulation spaces can
be identified with coorbit spaces by choosing
these representations as suitable
representations of the Heisenberg group.

\par

Table \ref{Table:ProgressModNorms} in the 
introduction give some overview on how far 
earlier results
on norm equivalences
for modulation spaces, are reaching compared 
to Theorem \ref{Thm:EquivNorms2}. Especially we focus
on extensions and generalizations with respect to
the parameters $r$, weight function $\omega$,
the window function $\phi$,
and the space $\mascB$ in Theorem \ref{Thm:EquivNorms2}.
\begin{itemize}
\item \emph{The parameter $r$}
in Table \ref{Table:ProgressModNorms}.
It is only allowed to
attain $\infty$ in the earlier
results (i.{\,}e. in 
\cite{FeiGro1,GaSa,Rau2}),
while in Theorem
\ref{Thm:EquivNorms2} it is allowed
to be any value in $[r_0,\infty]$.

\vrum

\item \emph{The weight function $\omega$.}
For the modulation spaces $M^{p,q}_{(\omega )}(\rr d)$
in \cite{GaSa} it is assumed that all weight functions
belong to $\mascP (\rr {2d})$, which is significantly
smaller than $\mascP _E(\rr {2d})$, the set of all moderate
weights on $\rr {2d}$. On the other hand, the analysis
in \cite{GaSa} should hold also after $\mascP$ is replaced
by the larger class $\mascP _E^0$. On the other hand
it seems to be difficult to include weights in
$\mascP _E\setminus \mascP _E^0$
in the pure analysis in \cite{GaSa}. In fact,
an essential part in \cite{GaSa} concerns
continuity of the inverse of
Gabor frame operators on  $M^{r_0}_{(v)}(\rr d)$,
and such properties can be guaranteed only for
$v\in \mascP _E^0(\rr d)$, while there are examples
that such properties do not hold when
$v\in \mascP _E(\rr d)\setminus \mascP _E^0(\rr d)$.
\\[1ex]
\emph{In particular it follows that weights which are exponential
functions are excluded in} \cite{GaSa} \emph{and its analysis, since
all such weights are in $\mascP _E\setminus \mascP _E^0$.}
\\[1ex]
Similar or even stronger restrictions seem to be present
in \cite{FeiGro1,Rau2}. For example, in Section 8 in \cite{Rau2},
it is assumed that all weights are symmetric, i.{\,}e., $\omega$
above should satisfy $\omega (-x,-\xi )=\omega (x,\xi )$.

\vrum

\item \emph{The window function $\phi$}
in the short-time Fourier transforms.
In the Banach space case, i.{\,}e.
in \cite{Fei1989,Fei1,FeiGro1} or when
$r_0=1$ in Theorem \ref{Thm:EquivNorms2},
$\phi$ is allowed to be any element in
$M^1_{(v)}\setminus 0$, for suitable $v$.
In the quasi-Banach cases, i.{\,}e.
in \cite{GaSa}, $\phi$ is allowed to be 
any element in $M^{r_0}_{(v)}$, provided 
there is a function $\psi \in M^{r_0}_{(v)}$
which is a dual Gabor atom to $\phi$.
In \cite{Rau2}, the problem with
existence of suitable dual Gabor atom
seems to be removed, provided the weights
(explained above)
possess the desired properties.

\vrum

\item \emph{The generality of $\mascB$.}
Evidently, \cite{GaSa} only deals with classical modulation spaces
$M^{p,q}_{(\omega )}(\rr d)$, which is a special case of
$M(\omega ,\mascB )$ spaces. If $\mascB$ is a Banach space,
then it seems that the assumptions in \cite{FeiGro1} and
in Theorem \ref{Thm:EquivNorms2} essentially agrees.
If $\mascB$ is a quasi-Banach space, then it seems that
stronger conditions on $\mascB$ are imposed in
Theorem \ref{Thm:EquivNorms2} compared to what
is the case in \cite{Rau2}. In fact, in Theorem
\ref{Thm:EquivNorms2} it is essential that $\mascB$
is \emph{normal}, while such a condition is
absent in \cite{Rau2}.
\end{itemize}
\end{rem}

\par

\section{Continuity, compactness,
unions and intersections for
modulation spaces}
\label{sec4}

\par

In this section we combine  results from
Section \ref{sec3} with some well-known
results for classical modulation spaces
to deduce continuity, compactness,
union and intersection properties for
modulation spaces of the form
$M(\omega ,\mascB )$.


\par

\subsection{Unions and intersections of
modulation spaces}

\par

First we have the following result, giving
general bounds of the modulation
$M(\omega ,\mascB )$ in terms of
classical modulation spaces.

\par

\begin{prop}\label{Prop:BasicEmbModSp1}
Let $\mascB$ be a normal QBF space
on $\rr {2d}$ with respect to
$r_0\in (0,1]$ and
$v_0\in \mascP _E(\rr {2d})$,
and let $\omega \in
\mascP _E(\rr {2d})$.
Then
$$
M^{r_0}_{(\omega v_0)}(\rr d)
\hookrightarrow
M(\omega ,\mascB)
\hookrightarrow
M^{\infty }_{(\omega /v_0)}(\rr d).
$$
\end{prop}

\par

\begin{proof}
The result follows by combining
Theorem \ref{Thm:EquivNorms2}
with Corollary
\ref{Cor:BasicEmbWSp}. The details
are left for the reader.
\end{proof}

\par

\begin{rem}
\label{Rem:BFtoQBF}
In Definition \ref{Def:BFSpaces} (3),
the main focus is
on a specific invariant QBF space $\mascB$
with respect to $r_0\in (0,1]$,
and one concludes that $\mascB$ is normal
if there exists an invariant BF space $\mascB _0$
such that the norm of $\mascB$ is given by
\eqref{Eq:NormQBFNorm}.

\par

On the other hand, by focusing on the invariant
BF space $\mascB _0$ instead, it follows that
for every such $\mascB _0$ and $r_0\in (0,1]$,
there is
a unique normal QBF space $\mascB _{[r_0]}$
of order $r_0$ with quasi-norm ($r_0$-norm)
given given by \eqref{Eq:NormQBFNorm},
when $\mascB =\mascB _{[r_0]}$.
Here we conclude that $\mascB _{[r_0]}$ essentially
increases with $r_0$.

\par

More precisely, let $r_1,r_2\in (0,1]$,
$\omega _1,\omega _2\in \mascP _E(\rr {2d})$,
$\mascB _0$
be an invariant solid BF space on $\rr {2d}$
with respect to $v_0\in \mascP _E(\rr {2d})$,
and suppose that
\begin{align}
r_1&\le r_2
\quad \text{and}\quad
\vartheta _2
\lesssim
\vartheta _1,
\qquad
\vartheta _j
=
\omega _jv_0^{-\frac 1{r_j}},
\quad
j=1,2.
\label{Eq:IncrParWeights}
\intertext{Then}
M(\omega _1,\mascB _{[r_1]})
&\hookrightarrow
M(\omega _2,\mascB _{[r_2]}),
\label{Eq:IncrParWeights1}
\intertext{and}
\nm f{M(\omega _2,\mascB _{[r_2]})}
&\lesssim
\nm f{M(\omega _1,\mascB _{[r_1]})},
\quad
f\in M(\omega _1,\mascB _{[r_1]}).
\label{Eq:IncrParWeights2}
\end{align}

\par

In fact, let $f\in M(\omega _1,\mascB _{[r_1]})$
and $F=V_\phi f$ for some
$\phi \in \Sigma _1(\rr d)\setminus 0$. Then,
\begin{align}
\nm f{M(\omega _2,\mascB _{[r_2]})}
&\asymp
\left (\nm { |F\cdot \omega _2 |^{r_2}}{\mascB _0}
\right )^{\frac 1{r_2}}
\notag
\\[1ex]
&\le
\left (\NM { \left |F\cdot \omega _2\cdot
v_0^{\frac 1{r_1}-\frac 1{r_2}} \right |^{r_1}
\cdot \left |
F\cdot \omega _2\cdot v_0^{-\frac 1{r_2}}
\right |^{r_2-r_1}}{\mascB _0}
\right )^{\frac 1{r_2}}
\notag
\\[1ex]
&\lesssim
\left (\NM { \left |F\cdot \omega _2\cdot
v_0^{\frac 1{r_1}-\frac 1{r_2}} \right |^{r_1}}{\mascB _0}
\NM {\left |F\cdot \omega _2\cdot v_0^{-\frac 1{r_2}}
\right |^{r_2-r_1}}{L^\infty}
\right )^{\frac 1{r_2}}
\notag
\\[1ex]
&\lesssim
\left (\nm f{M(\omega _1,\mascB _{[r_1]})}
\right )^{\frac {r_1}{r_2}}
\left (
\NM f{M^\infty _{(\vartheta _2)}}
\right )^{1-\frac {r_1}{r_2}}.
\label{Eq:SolidNormEst}
\end{align}
Here the last inequality follows from
\eqref{Eq:IncrParWeights} and the fact that
$\mascB _0$ is solid.

\par

By combining Proposition \ref{Prop:BasicEmbModSp1}
with \eqref{Eq:IncrParWeights}, and the fact that
$\mascB _{[r_1]}$
is a normal QBF space of order $r_1$
and $v_0^{1/{r_1}}$,
for the last factor in \eqref{Eq:SolidNormEst}
we obtain
\begin{align*}
\NM f{M^\infty _{(\vartheta _2)}}
\lesssim
\NM f{M(\vartheta _2v_0^{1/r_1},
\mascB _{[r_1]})}
\lesssim
\nm f{M(\omega _1,\mascB _{[r_1]})}.
\end{align*}
A combination of the latter inequality and
\eqref{Eq:SolidNormEst} now gives
\eqref{Eq:IncrParWeights1}
and \eqref{Eq:IncrParWeights2},
and the assertion follows.
\end{rem}

\par

\begin{rem}
By choosing $\mascB =
L^{p,q}_{(1/v_0)}(\rr {2d})$,
$p,q\in [1,\infty]$,
it follows that the assumptions on
the weights in \eqref{Eq:IncrParWeights}
cannot be improved, in order for
\eqref{Eq:IncrParWeights2} to hold.

\par

In fact, suppose that
\eqref{Eq:IncrParWeights2} holds true,
and
let $\phi \in \Sigma _1(\rr d)\setminus 0$,
$p_j=pr_j$ and $q_j=qr_j$,
$j=1,2$. Then $p_1\le p_2$,
$q_1\le q_2$, and
\begin{align*}
\nm f{M(\omega _j,\mascB _{[r_j]})}
&\asymp
\left (
\nm {\, |V_\phi f\cdot \omega _j|^{r_j} \, }
{L^{p,q}_{(1/v_0)}}
\right )^{\frac 1{r_j}}
\\[1ex]
&=
\nm {V_\phi f\cdot \vartheta _j}
{L^{p_j,q_j}}
\asymp
\nm {f}
{M^{p_j,q_j}_{(\vartheta _j)}}.
\end{align*}

\par

A combination with the latter
relations and \eqref{Eq:IncrParWeights2}
shows that
\begin{equation}
\label{Eq:IneqSTFTWeights}
\nm {V_\phi f\cdot \vartheta _2}
{L^{p_2,q_2}}
\lesssim
\nm {V_\phi f\cdot \vartheta _1}
{L^{p_1,q_1}}
\end{equation}
By choosing
$\phi (x)=\pi ^{-\frac d4}e^{-\frac 12\cdot |x|^2}$
and letting
$$
f(x)=f_{Y}(x)=\pi ^{-\frac d4}
e^{-\frac 12\cdot |x-y|^2}e^{i\scal x\eta},
\qquad Y=(y,\eta )\in \rr {2d},
$$
it follows that
$$
|V_\phi f_Y(X)| =(2\pi )^{-\frac d2}e^{-\frac 14\cdot |X-Y|^2},
\qquad X,Y\in \rr {2d}.
$$
This gives
\begin{equation}
\label{Eq:EqSTFTWeights}
\nm {V_\phi f_Y\cdot \vartheta _j}
{L^{p_j,q_j}}
\asymp
\nm {e^{-\frac 14\cdot |\cdo |^2}
\vartheta _j(\cdo +Y)}{L^{p_j,q_j}}
\asymp
\vartheta _j(Y),
\end{equation}
because
\begin{align*}
\nm {e^{-\frac 14\cdot |\cdo |^2}
\vartheta _j(\cdo +Y)}{L^{p_j,q_j}}
\lesssim
\nm {e^{-\frac 14\cdot |\cdo |^2}
\cdot v}{L^{p_j,q_j}}\vartheta _j(Y)
\asymp
\vartheta _j(Y)
\intertext{and}
\nm {e^{-\frac 14\cdot |\cdo |^2}
\vartheta _j(\cdo +Y)}{L^{p_j,q_j}}
\gtrsim
\nm {e^{-\frac 14\cdot |\cdo |^2}
/v}{L^{p_j,q_j}}\vartheta _j(Y)
\asymp
\vartheta _j(Y),
\end{align*}
if $v\in \mascP _E(\rr {2d})$
is chosen such that $\vartheta _j$
is $v$-moderate.

\par

A combination of
\eqref{Eq:IneqSTFTWeights}
and
\eqref{Eq:EqSTFTWeights}
gives
$$
\vartheta _2\lesssim \vartheta _1
\quad \Leftrightarrow \quad
\omega _2v_0^{-1/{r_2}}
\lesssim
\omega _1v_0^{-1/{r_1}},
$$
which is the same as
\eqref{Eq:IncrParWeights}.
\end{rem}

\par

In the next three propositions,
we deduce
general bounds of the modulation space
$M(\omega ,\mascB )$ in terms of
the Schwartz space, Gelfand-Shilov
spaces, and their distribution spaces.
Here recall Subsection \ref{subsec1.1}
for the definition of various types
of classes of weight functions.

\par

\begin{prop}
\label{Prop:BasicEmbModSp2Schw}
Let $\mascB$ be a normal
QBF space
on $\rr {2d}$ with respect to
$r_0\in (0,1]$ and
$v_0\in \mascP (\rr {2d})$,
and let $\omega \in
\mascP (\rr {2d})$.
Then
\begin{equation}\label{Eq:ModEmbSchw}
\begin{gathered}
\mascS (\rr d)
\hookrightarrow
M(\omega ,\mascB)
\hookrightarrow
\mascS '(\rr d),
\\[1ex]
\bigcap
M(\omega ,\mascB ) = \mascS (\rr d)
\quad \text{and}\quad
\bigcup
M(\omega ,\mascB ) = \mascS '(\rr d),
\end{gathered}
\end{equation}
where the intersection and union
are taken over all
$\omega \in \mascP (\rr {2d})$.
%
%
\end{prop}

\par

\begin{prop}
\label{Prop:BasicEmbModSp2Roum}
Let $s,\sigma \ge 1$, $\mascB$ be a normal
QBF space
on $\rr {2d}$ with respect to
$r_0\in (0,1]$ and
$v_0\in \mascP _{s,\sigma}^0(\rr {2d})$,
and let $\omega \in
\mascP _{s,\sigma}^0(\rr {2d})$.
Then
\begin{equation}\label{Eq:ModEmbRoum}
\begin{gathered}
\maclS _s^\sigma (\rr d)
\hookrightarrow
M(\omega ,\mascB)
\hookrightarrow
(\maclS _s^\sigma )'(\rr d),
\\[1ex]
\bigcap
M(\omega ,\mascB ) = \maclS _s^\sigma (\rr d)
\quad \text{and}\quad
\bigcup
M(\omega ,\mascB ) = (\maclS _s^\sigma )'(\rr d),
\end{gathered}
\end{equation}
where the intersection and union
are taken over all
$\omega \in \mascP _{s,\sigma}^0(\rr {2d})$.
%
%
\end{prop}

\begin{prop}
\label{Prop:BasicEmbModSp2Beurl}
Let $s,\sigma \ge 1$, $\mascB$ be a normal
QBF space
on $\rr {2d}$ with respect to
$r_0\in (0,1]$ and
$v_0\in \mascP _{s,\sigma}(\rr {2d})$,
and let $\omega \in
\mascP _{s,\sigma}(\rr {2d})$.
Then
\begin{equation}\label{Eq:ModEmbBeur}
\begin{gathered}
\Sigma _s^\sigma (\rr d)
\hookrightarrow
M(\omega ,\mascB)
\hookrightarrow
(\Sigma _s^\sigma )'(\rr d),
\\[1ex]
\bigcap
M(\omega ,\mascB ) = \Sigma _s^\sigma (\rr d)
\quad \text{and}\quad
\bigcup
M(\omega ,\mascB ) = (\Sigma _s^\sigma )'(\rr d),
\end{gathered}
\end{equation}
where the intersection and union
are taken over all
$\omega \in \mascP _{s,\sigma}(\rr {2d})$.
%
%
\end{prop}

\par

We only prove Proposition
\ref{Prop:BasicEmbModSp2Beurl}.
Propositions \ref{Prop:BasicEmbModSp2Schw}
and \ref{Prop:BasicEmbModSp2Roum} follow
by similar arguments and are left for the reader.

\par

\begin{proof}[Proof of Proposition
\ref{Prop:BasicEmbModSp2Beurl}]
The embeddings in \eqref{Eq:ModEmbBeur}
follows from Proposition
\ref{Prop:BasicEmbModSp1} and
the embeddings
\begin{equation}
\label{eq:EmbeddingsModSpAndGelfShilSp}
\Sigma _s^\sigma (\rr d)
\hookrightarrow
M^{r_0}_{(\omega )}(\rr d)
\hookrightarrow
M^{\infty}_{(\omega )}(\rr d)
\hookrightarrow
(\Sigma _s^\sigma )'(\rr d),
\end{equation}
for every
$\omega \in \mascP _{s,\sigma}(\rr {2d})$.
(See e.{\,}g. \cite[Theorem 3.9]{Toft10}.
For weights moderated by polynomials,
see also \cite{GaSa}.)

\par

Finally, the equalities in
\eqref{Eq:ModEmbBeur} follow from 
Corollary 
\ref{Prop:BasicEmbModSp1}, the fact
that $\mascP _{s,\sigma}(\rr {2d})$ is a group
under multiplication,
and the identities
\begin{equation} \label{eq:mod-spaces-sigma_1}
 \bigcap
M_{(\omega )}^\infty (\rr d)
= \Sigma _s^\sigma (\rr d)
\quad \text{and}\quad
\bigcup
M_{(\omega )}^{r_0} (\rr d)
= (\Sigma _s^\sigma )'(\rr d),
 \end{equation}
where the intersection and union
are taken over all
$\omega \in \mascP _{s,\sigma}(\rr {2d})$.
We refer to
Proposition 6.5 in \cite{Toft17} or
Theorem 4.1 in \cite{Teofanov3}
for the proof of
\eqref{eq:mod-spaces-sigma_1}.
%
\end{proof}

\par

Finally we observe that
$M(\omega ,\mascB)$ is a quasi-Banach
space (of the same order as $\mascB$),
which is explained in the following proposition.

\par

\begin{prop}\label{Prop:ComplModSp}
Let $\mascB$ be a normal invariant
QBF space on $\rr {2d}$ of order
$r_0\in (0,1]$, and let
$\omega \in \mascP _E(\rr {2d})$.
Then $M(\omega ,\mascB)$ is a
quasi-Banach space of order $r_0$.
\end{prop}

\par

Proposition \ref{Prop:ComplModSp}
follows by similar arguments as for
analogous results in 
\cite{FeiGro1,Gro2,Rau1}. In order to
be self-contained we give a proof in
Appendix \ref{App:A}.

\par

\subsection{Continuity and compactness
for modulation spaces}

\par

In various contexts it is shown
that if $\mascB$ and
$\omega _1,\omega _2\in \mascP _E(\rr {2d})$
are suitable, then the inclusion map
\begin{align}\label{Embedding}
i: M(\omega _1, \mascB) 
\rightarrow
M(\omega _2, \mascB)
\end{align}
is continuous when
${\omega _2}\slash {\omega _1}$ is bounded,
and compact when ${\omega _2}\slash {\omega _1}$
tends to zero at infinity. 
In \cite{PfTo19} these properties were shown for 
invariant BF spaces $\mascB$ fulfilling the additional condition 
\begin{equation}\label{Eq:CondForQBFSpaces}
\Sigma _1(\rr {2d})\hookrightarrow \mascB ,
\end{equation}
and for QBF spaces of Lebesgue-type.
In Propositions 1.20 and  1.21 in 
\cite{BimTof},
these properties were generalized
to all invariant QBF spaces $\mascB$
fulfilling \eqref{Eq:CondForQBFSpaces}. 
Due to 
Proposition \ref{Prop:BasicEmbWSp} (2) 
and Corollary \ref{cor:relation-to-Sigma}, 
condition \eqref{Eq:CondForQBFSpaces} holds 
for all normal QBF spaces $\mascB$. 
Hence by Propositions 1.20 and 1.21
in \cite{BimTof} we immediately get the following 
continuity respectively compactness result.

\par

\begin{thm}\label{thm:CompAndContProp}
    Let $\mascB$ be a normal QBF space 
    with respect to $r_0 \in (0,1]$ 
    and $v_0 \in \mascP _E (\rr d)$. 
    Also let 
    $\omega _1, \omega_2 \in \mascP _E (\rr {2d})$. 
    Then the following is true:
    \begin{enumerate}
        \item 
        if $\omega _2 \lesssim \omega _1$, 
        then $M(\omega_1, \mascB) 
        \subseteq M(\omega_2, \mascB)$
        and the map \eqref{Embedding}
        is continuous;
        
        \vrum

        \item 
        if $\lim _{|X|\to \infty}
\frac {\omega _2(X)}{\omega _1(X)} =0$,
 then $M(\omega_1, \mascB) 
        \subseteq M(\omega_2, \mascB)$
        and the map \eqref{Embedding}
        is compact;
        
        \vrum
        
        \item if in addition $v_0$ is bounded, then 
        $\omega _2 \lesssim \omega _1$ if and only
        if the map \eqref{Embedding}
        is continuous;
        \item 
        if in addition $v_0$ is bounded, then
$\lim _{|X|\to \infty}
\frac {\omega _2(X)}{\omega _1(X)} =0$,
if and only if the map \eqref{Embedding}
        is compact.
    \end{enumerate}
\end{thm}

\par

The corresponding result to (2)
in Theorem \ref{thm:CompAndContProp}
is not explicitly
given in Proposition 1.21 of 
\cite{BimTof}. However, it is employed in
its proof.
In order to be self-contained, we present
a proof of (2) and (4) in Appendix 
\ref{App:A}. 


\par

\section{Convolutions and multiplications in
Wiener amalgam and modulation spaces}\label{sec5}

\par

In this section we deduce convolution and multiplication properties
for Wiener amalgam spaces, which extends analogous results in
\cite{Toft04,Toft26} to more general modulation spaces. In the first
part we discuss an extended class of non-commutative convolutions,
which appears when applying short-time Fourier transforms on
convolutions and multiplications. Thereafter we deduce estimates
of various types of non-commutative convolutions and multiplications
in background of Wiener amalgam spaces and modulation spaces.

\par



\subsection{Wiener amalgam estimates for
non-commutative convolutions}

We observe that the twisted
convolution (as well as the ordinary
convolution) is a mapping of the form
\begin{equation}\label{Eq:QConv}
(f*_\vartheta g)(x,y)
=
\int _{\rr {d_1}} f(x-z,y)g(z,y)
e^{i\vartheta (x,y,z)}\, dz,
\quad
(x,y)\in \rr {d_1+d_2},
\end{equation}
where $\vartheta$ is a
real-valued measurable function on
$$
\rr {d_1}\times \rr {d_2}\times \rr {d_1}
\asymp \rr {2d_1+d_2}.
$$
It is obvious that
the map $(f,g)\mapsto f*_\vartheta g$ is continuous
from $\Sigma _1(\rr d)\times
\Sigma _1(\rr d)$ to $L^1(\rr d)$,
with $d=d_1+d_2$.

\par

Evidently, the corresponding discrete version of \eqref{Eq:QConv}
is given by
\begin{equation}\label{Eq:QConvDisc}
(f*_\vartheta g)(j,k)
=
\sum _{n\in \zz {d_1}} f(j-n,k)g(n,k)
e^{i\vartheta (j,k,n)},
\quad
(j,k)\in \zz {d_1+d_2},
\end{equation}
which, for example, is well-defined and 
continuous
from 
$\ell ^{p_1,q}_{(\omega _1 )}(\zz d)
\times
\ell ^{p_2,\infty}_{(\omega _2)}(\zz d)$
to $\ell ^{p_0,q}_{(\omega _0)}(\zz d)$
when
\begin{equation}\label{Eq:YoungCond}
\frac 1{p_1}+\frac 1{p_2} = 1+\frac 1{p_0},
\end{equation}
and
\begin{equation}\label{Eq:ConvWeightsEst}
\omega _0(x_1+x_2,y)\lesssim \omega _1(x_1,y)\omega _2(x_2,y),
\quad
x_1,x_2\in \rr {d_1},\ y\in \rr {d_2},
\end{equation}
holds.

\par


The following result deals with Wiener amalgam
estimates of mappings in \eqref{Eq:QConv}.

\par

\begin{prop}\label{Prop:ConvWienerSp}
Suppose $d=d_1+d_2>0$, $d_1,d_2\in \mathbf N$, $\mascB$ is an invariant
QBF space on $\rr d$ of order
$r_0\in (0,1]$ and
$v_0\in \mascP _E(\rr {d})$,
$\omega _j,v\in \mascP _E(\rr {d})$, $j=0,1,2$,
be such that $v(x,y)=v_0(x,0)$ and \eqref{Eq:ConvWeightsEst} holds,
and let
$$
p_0,p_1,p_2\in [1,\infty ]
$$
be such that
\eqref{Eq:YoungCond} holds.
Also let $\vartheta$ be a real-valued
and measurable function on $\rr d$, and let $*_\vartheta$ from
$\Sigma _1(\rr d)\times \Sigma _1(\rr d)$
to $L^1(\rr d)$ be the non-commutative
convolution, given in \eqref{Eq:QConv}.
Then $(f,g)\mapsto f*_\vartheta g$ extends to a continuous map
from
$\sfW ^{p_1,r_0}_{d_1,d_2}(\omega _1,\mascB )
\times \sfW ^{p_2,\infty}_{d_1,d_2,*}
(\omega _2v,\ell ^{r_0,\infty}_*(\zz d))$
to
$\sfW ^{p_0,r_0}_{d_1,d_2}(\omega _0,\mascB )$, and
\begin{equation}\label{Eq:ConvEstWiener}
\begin{aligned}
\nm {f*_\vartheta g}{\sfW ^{p_0,r_0}_{d_1,d_2}(\omega _0,\mascB )}
&\lesssim
\nm f{\sfW ^{p_1,r_0}_{d_1,d_2}(\omega _1,\mascB )}
\nm g{\sfW ^{p_2,\infty}_{d_1,d_2,*}(\omega _2v,\ell ^{r_0,\infty}_*)},
\\[1ex]
f&\in \sfW ^{p_1,r_0}_{d_1,d_2}(\omega _1,\mascB ),
\ g\in \sfW ^{p_2,\infty}_{d_1,d_2,*}
(\omega _2v,\ell ^{r_0,\infty}_*(\zz d)).
\end{aligned}
\end{equation}
\end{prop}

\par

Proposition \ref{Prop:ConvWienerSp} follows by
related results,
available in e.{\,}g. \cite{Rau1}. In order
to be self-contained we here present a proof.

\par

\begin{proof}
We first prove \eqref{Eq:ConvEstWiener}
when $f,g\ge 0$ are measurable and $\vartheta (x,y,z)=0$,
leading to that $*_\vartheta $ is the ordinary convolution
in the $x$-variable.
Let $Q_{n}=[0,1]^{n}$, $Q_{n,r}=[-r,r]^{n}$, $r\ge 0$,
$n \in \mathbf Z_+$, 
\begin{alignat*}{2}
f_\omega
&=
f\cdot \omega ,&
\quad f_{\omega ,j}
&=
f_{\omega} \cdot \chi _{j+Q_{d,1}},
\\[1ex]
g_{\omega}
&=
g\cdot {\omega} ,&
\quad g_{\omega ,j}
&=
g_{\omega} \cdot \chi _{j+Q_d},
\\[1ex]
\beta _0(j)
&=
\nm {(f*_0g)\, \omega _0}{L^{p_0,r_0}(j+Q_d)}&
\quad \text{and}\quad
\beta (j)
&=
\nm {f_{\omega _1 ,j}}{L^{p_1,r_0}},
\quad
j\in \zz d,
\end{alignat*}
when $\omega \in \mascP _E(\rr d)$.
Also let
$$
h_j(y)
=
\nm {g_{\omega _2,j}(y,\cdo )}{L^\infty (\rr {d_2})},
\quad \text{and}\quad
\gamma (j)
=
\nm {h_j}{L^{p_2}(\rr {d_1})},\quad
j\in \zz d,\ y\in \rr {d_1}.
$$
By \eqref{Eq:ConvWeightsEst}, Minkowski's inequality and
Young's inequality, we obtain
\begin{align}
\nm {(f*_0g)(\cdo ,y)\, & \omega _0(\cdo ,y)}{L^{p_0}(j_1+Q_{d_1})}
\\[1ex]
&\le
\NM {\sum _{k\in \zz {d}}\int _{\rr {d_1}}f_{\omega _1}(\cdo -z,y) h_{k}(z)\, dz}
{L^{p_0}(j_1+Q_{d_1})}
\notag
\\[1ex]
&\le
\sum _{k\in \zz {d}}
\NM {\int _{\rr {d_1}}f_{\omega _1}(\cdo -z,y) h_{k}(z)\, dz}
{L^{p_0}(j_1+Q_{d_1})}
\notag
\\[1ex]
&\le
\sum _{k\in \zz {d}}
\NM {\int _{\rr {d_1}}f_{\omega _1,(j_1-k_1,k_2)}(\cdo -z,y) h_{k}(z)\, dz}
{L^{p_0}(\rr {d_1})}
\notag
\\[1ex]
&\le
\sum _{k\in \zz {d}}
\nm {f_{\omega _1,(j_1-k_1,k_2)}(\cdo ,y)}{L^{p_1}(\rr {d_1})}
\nm  {h_{k}}{L^{p_2}(\rr {d_1})}
\notag
\\[1ex]
&\le
\sum _{k\in \zz {d}}
h_{(j_1-k_1,k_2)}(y)
\gamma (k),
\label{Eq:MixedDiscrConv}
\intertext{where}
h_{(k_1,k_2)}(y)
&=
\nm {f_{\omega _1,(k_1,k_2)}(\cdo ,y)}{L^{p_1}(\rr {d_1})},
\quad (k_1,k_2)\in \zz {d_1}\times \zz {d_2}.
\end{align}

\par

Let $\Omega =Q_{d_2,2}\bigcap \zz {d_2}$.
By applying the $L^{r_0}(j_2+Q_{d_2})$ norm in \eqref{Eq:MixedDiscrConv},
we obtain
\begin{align*}
\beta _0(j)^{r_0}
&\le
\NM {\sum _{k\in \zz {d}}
h_{(j_1-k_1,k_2)}
\gamma (k)}{L^{r _0} (j_2+Q_{d_2})}^{r_0}
\\[1ex]
&\le
\sum _{k\in \zz {d}}
\nm {f_{\omega _1,(j_1-k_1,k_2)}}{L^{p_1,r_0}(\rr {d_1}\times (j_2+Q_{d_2}))}^{r_0}
\gamma (k)^{r_0}
\\[1ex]
&\le
\sum _{k_2\in j_2+\Omega}
\left (
\sum _{k_1\in \zz {d_1}}
\nm {f_{\omega _1,(j_1-k_1,k_2)}}{L^{p_1,r_0}(\rr d)}^{r_0}
\gamma (k)^{r_0}
\right )
\\[1ex]
&\le
\sum _{k_2\in j_2+\Omega}
\left (\nm {f_{\omega _1,(\cdo ,k_2)}}{L^{p_1,r_0}(\rr d)}^{r_0}
*\gamma (\cdo ,k_2)^{r_0}\right )(j_1)
\\[1ex]
&=
\sum _{k_2\in \Omega}
\left (\beta (\cdo ,j_2+k_2)^{r_0}
*\gamma (\cdo ,j_2+k_2)^{r_0}\right )(j_1)
\\[1ex]
&\le
\sum _{k_2\in \Omega}
\left (\beta (\cdo ,j_2+k_2)^{r_0}
*\gamma _0^{r_0}\right )(j_1),
\end{align*}
where
$$
\gamma _0(k_1) = \nm {\gamma (k_1,\cdo )}{\ell
^{\infty}(\zz {d_2})}.
$$

\par

Since $\Omega$ is a finite set, we obtain
$$
\NM
{\sum _{k_2\in \Omega}
a(\cdo +(0,k_2))}{\ell _{\mascB}} \lesssim \nm a{\ell _{\mascB}},
\quad a\in \ell _0'(\zz d).
$$
This gives
\begin{align*}
\nm {\beta _0}{\ell _{\mascB}}^{r_0}
&\asymp
\nm {\beta _0^{r_0}}{\ell _{\mascB _0}}
\\[1ex]
&\lesssim
\NM {\left (
\sum _{k_1\in \zz {d_1}}\beta (\cdo -(k_1,0))^{r_0}
\gamma _0(k_1)^{r_0}
\right )}{\ell _{\mascB _0}}
\\[1ex]
&\le
\sum _{k_1\in \zz {d_1}}
\nm {\beta (\cdo -(k_1,0))^{r_0}}{\ell _{\mascB _0}}
\gamma _0(k_1)^{r_0}
\\[1ex]
&\lesssim
\nm {\beta ^{r_0}}{\ell _{\mascB _0}}
\left (
\sum _{k_1\in \zz {d_1}} \gamma _0(k_1)^{r_0}v(k_1,0)^{r_0}\right )
\\[1ex]
&\asymp
\nm f{\sfW ^{p_1,r_0} _{d_1,d_2}(\omega _1,\mascB )}^{r_0}
\nm g{\sfW ^{p_2,\infty} _{d_1,d_2,*}(\omega _2v,\ell ^{r_0,\infty}_*)}^{r_0}.
\end{align*}
This gives \eqref{Eq:ConvEstWiener}
and thereby the result,
when $\vartheta =0$ and $f,g\ge 0$.

\par

Suppose that $f\in \sfW ^{p_1,r_0}_{d_1,d_2}(\omega _1,\mascB )$
and $g\in \sfW ^{p_2,\infty} _{d_1,d_2,*}(v_0v,\ell ^{r_0,\infty}_*(\zz d))$
are arbitrary, and decompose $f$, $g$ 
and $e^{i\vartheta}$
into their positive and negative real and imaginary 
parts. That is,
$$
f = \sum _{n=1}^4 i^{n-1}f_n,
\quad
g = \sum _{n=1}^4 i^{n-1}g_n,
\quad \text{and}\quad
e^{i\vartheta}
=
\sum _{n=1}^4 i^{n-1}K_n,
$$
where $f_n,g_n,K_n\ge 0$, $n=1,\dots ,4$,
are chosen as small as possible. Then
\begin{equation}\label{Eq:ConvDecomp}
f*_\vartheta g =\sum _{n_1=1}^4\sum _{n_2=1}^4
\sum _{n_3=1}^4 i^{n_1+n_2+n_3-3}h_{n_1,n_2,n_3},
\end{equation}
where
$$
h_{n_1,n_2,n_3}(x)
=
\int _{\rr d}f_{n_1}(x-z,y)g_{n_2}(y,z)K_{n_3}(x,y,z)\,dy
$$
has only non-negative factors in the integral. We 
observe that
$$
0\le
h_{n_1,n_2,n_3}
\le
|f|*|g|.
$$

\par

The first part of the proof
now shows that
$h_{n_1,n_2,n_3}$ is a well-defined 
element in
$\sfW ^{p_0,r_0} _{d_1, d_2}(\omega _0,\mascB )$, and that
$$
\nm {h_{n_1,n_2,n_3}}
{\sfW ^{p_0,r_0} _{d_1, d_2}(\omega _0,\mascB )}
\lesssim
\nm f{\sfW ^{p_1,r_0} _{d_1, d_2}(\omega _1,\mascB )}
\nm g{\sfW ^{p_2,\infty} _{d_1, d_2,*}(\omega _2v,\ell ^{r_0,\infty}_*)},
$$
for every $n_1$, $n_2$ and $n_3$. Hence,
by defining $f*_\vartheta g$ by \eqref{Eq:ConvDecomp},
it follows that $f*_\vartheta g$ is a well-defined  element in $\sfW ^{p_0,r_0}_{d_1, d_2}(\omega _0,\mascB )$,
and that \eqref{Eq:ConvEstWiener} holds.
This shows that the requested extension
exists.
\end{proof}

\par

By letting $d_2=0$ or $d_1=0$ in Proposition \ref{Prop:ConvWienerSp}, we
obtain the following results. The details are left for the reader.

\par

\begin{cor}\label{Cor:ConvWienerSp1}
Suppose $\mascB$ is an invariant
QBF space on $\rr d$ of order
$r_0\in (0,1]$ and
$v_0\in \mascP _E(\rr {d})$,
$\omega _j\in \mascP _E(\rr {d})$, $j=0,1,2$,
be such that 
\begin{equation}
\label{Eq:PureConvWeightsEst}
\omega _0(x_1+x_2)\lesssim \omega _1(x_1)\omega _2(x_2),
\quad
x_1,x_2\in \rr {d},
\end{equation}
holds, and let
$$
p_0,p_1,p_2\in [1,\infty ]
$$
be such that
\eqref{Eq:YoungCond} holds.
Also let $\vartheta$ be a real-valued
and measurable function on $\rr d$, and let $*_\vartheta$ from
$\Sigma _1(\rr d)\times \Sigma _1(\rr d)$
to $L^1(\rr d)$ be
$$
(f*_\vartheta g)(x)
=
\int _{\rr {d_1}} f(x-z)g(z)
e^{i\vartheta (x,z)}\, dz,
\quad
x\in \rr d.
$$
Then $(f,g)\mapsto f*_\vartheta g$ extends to a continuous map
from
$\sfW ^{p_1}(\omega _1,\mascB )\times \sfW ^{p_2}
(\omega _2v_0,\ell ^{r_0}(\zz d))$
to
$\sfW ^{p_0}(\omega _0,\mascB )$. If in addition $p_2<\infty$,
then the extension is unique.
\end{cor}

\par

\begin{cor}\label{Cor:ConvWienerSp2}
Suppose $\mascB$ is an invariant QBF space on $\rr d$ of order
$r_0\in (0,1]$ and
$v_0\in \mascP _E(\rr {d})$,
$\omega _j\in \mascP _E(\rr {d})$, $j=0,1,2$,
and
\begin{equation}
\label{Eq:PureMultWeightsEst}
\omega _0(y)\lesssim \omega _1(y)\omega _2(y),
\quad
y\in \rr {d},
\end{equation}
holds.
Also let $\vartheta$ be a real-valued
and measurable function on $\rr d$.
Then $(f,g)\mapsto f\cdot g$ extends to a continuous map
from
$\sfW ^{r_0}(\omega _1,\mascB )\times \sfW ^{\infty}
(\omega _2\, v_0,\ell ^{\infty}_*(\zz d))$
to
$\sfW ^{r_0}(\omega _0,\mascB )$.
\end{cor}

\par

\begin{rem}\label{Rem:IndepChapters}
For future references we observe that
no results from Sections \ref{sec2} and
\ref{sec3} were used in the deductions of
Proposition \ref{Prop:ConvWienerSp},
Corollary \ref{Cor:ConvWienerSp1}
and Corollary \ref{Cor:ConvWienerSp2}.
\end{rem}

\par

\subsection{Modulation space estimates for convolutions
and multiplications}

\par

Next we shall use the previous results to deduce convolution
and multiplication properties for the modulation spaces under
considerations. For the norm estimates we shall use
Proposition \ref{Prop:ConvWienerSp}. For the existence and
uniqueness we shall use the first of the following two lemmas
for convolution, and the second lemma for multiplications.
Here the involved weight functions should satisfy
\begin{alignat}{2}
\omega _0(x_1+x_2,\xi )
&\lesssim
\omega _1(x_1,\xi )\omega _2(x_2,\xi ), &
\quad
x_1,x_2,\xi &\in \rr d,
\label{Eq:ModConvWeightsEst}
\intertext{or}
\omega _0(x,\xi _1+\xi _2)
&\lesssim
\omega _1(x,\xi _1)\omega _2(x,\xi _2), &
\quad
x,\xi _1,\xi _2 &\in \rr d.
\label{Eq:ModMultWeightsEst}
\end{alignat}
We observe that \eqref{Eq:ModConvWeightsEst}
agree with \eqref{Eq:ConvWeightsEst}
when $d_1=d_2$.

\par

\begin{lemma}\label{Lemma:AssistConv}
Let $r_0\in (0,1]$, $v_0\in \mascP _E(\rr {2d})$ be submultiplicative,
$v(x,\xi )=v_0(x,0)$, and suppose that
$\omega _j\in \mascP _E(\rr {2d})$ satisfy \eqref{Eq:ModConvWeightsEst}.
Then the following is true:
\begin{enumerate}
\item the map
$(f,g)\mapsto f*g$ from $\Sigma _1(\rr d)\times \Sigma _1(\rr d)$
to $\Sigma _1(\rr d)$ is uniquely extendable to a continuous map
from $M^\infty _{(\omega _1/v_0)}(\rr d)\times W^{r_0,\infty}_{(\omega _2v)}(\rr d)$
to $M^\infty _{(\omega _0/v_0)}(\rr d)$;

\vrum

\item if $f\in M^\infty _{(\omega _1/v_0)}(\rr d)$,
$g\in W^{r_0,\infty}_{(\omega _2v)}(\rr d)$, and $\phi ,\psi \in \Sigma _1(\rr d)$,
then
$$
(x,\xi )\mapsto \big (|V_\phi f (\cdo ,\xi )|*|V_\psi g (\cdo ,\xi )|\big )(x)
$$
belongs to $\sfW ^{r_0}(\omega _0/v_0,\ell ^\infty (\zz {2d}))$, and
$$
V_{\phi *\psi} (f*g)(x,\xi )
=
(2\pi )^{\frac d2} \big (V_\phi f (\cdo ,\xi )*V_\psi g (\cdo ,\xi )\big )(x).
$$
\end{enumerate}
\end{lemma}

\par

\par

\begin{lemma}\label{Lemma:AssistMult}
Let $r_0\in (0,1]$, $v_0\in \mascP _E(\rr {2d})$ be submultiplicative,
$v(x,\xi )=v_0(0,\xi )$, and suppose that
$\omega _j\in \mascP _E(\rr {2d})$ satisfy \eqref{Eq:ModMultWeightsEst}.
Then the following is true:
\begin{enumerate}
\item the map
$(f,g)\mapsto f\cdot g$ from $\Sigma _1(\rr d)\times \Sigma _1(\rr d)$
to $\Sigma _1(\rr d)$ is uniquely extendable to a continuous map
from $M^\infty _{(\omega _1/v_0)}(\rr d)\times M^{\infty ,r_0}_{(\omega _2v)}(\rr d)$
to $M^\infty _{(\omega _0/v_0)}(\rr d)$;

\vrum

\item if $f\in M^\infty _{(\omega _1/v_0)}(\rr d)$,
$g\in M^{\infty ,r_0}_{(\omega _2v)}(\rr d)$, and $\phi ,\psi \in \Sigma _1(\rr d)$,
then
$$
(x,\xi )\mapsto \big (|V_\phi f (x,\cdo )|*|V_\psi g (x,\cdo )|\big )(\xi )
$$
belongs to $\sfW ^{r_0}(\omega _0/v_0,\ell ^\infty (\zz {2d}))$, and
$$
V_{\phi \cdot \psi} (f\cdot g)(x,\xi )
=
(2\pi )^{-\frac d2} \big (V_\phi f (x,\cdo )*V_\psi g (x,\cdo )\big )(\xi ).
$$
\end{enumerate}
\end{lemma}

%
%
%

\par

\begin{proof}[Proof of Lemma \ref{Lemma:AssistConv}]
The assertion (1) essentially follows
from \cite[Theorem 3.8]{Toft26}. Let
\begin{align*}
\vartheta _0(x,\xi )
&=
\frac {\omega _0(x,\xi )}{v_0(x,\xi )},
\quad
\vartheta _1(x,\xi )
=
\frac {\omega _1(x,\xi )}{v_0(x,\xi )}
\intertext{and}
\vartheta _2(x,\xi )
&=
\omega _2(x,\xi )v(x,\xi )
=
\omega _2(x,\xi )v_0(x,0),
\quad x,\xi \in \rr d.
\end{align*}
Since $v_0$ is submultiplicative, \eqref{Eq:ModConvWeightsEst}
gives
\begin{align*}
\vartheta _0(x_1+x_2,\xi )
&=\frac {\omega _0(x_1+x_2,\xi )}{v_0(x_1+x_2,\xi )}
\lesssim
\frac {\omega _1(x_1,\xi )\omega _2(x_2,\xi )}{v_0(x_1+x_2,\xi )}
\\[1ex]
&\lesssim
\frac {\omega _1(x_1,\xi )\omega _2(x_2,\xi )v_0(x_2,0)}{v_0(x_1,\xi )}
=
\vartheta _1(x_1,\xi )\vartheta _2(x_2,\xi ).
\end{align*}

\par

By \cite[Theorem 3.8]{Toft26} it follows that the map
$(f,g)\mapsto f*g$ from $\Sigma _1(\rr d)\times \Sigma _1(\rr d)$
to $\Sigma _1(\rr d)$ extends uniquely to a continuous map
from $M^\infty _{(\vartheta _1)}(\rr d)
\times
M^{\infty ,r_0}_{(\vartheta _2)}(\rr d)$ to
$M^\infty _{(\vartheta _0)}(\rr d)$.
This is the same as (1).

\par

The assertion (2) follows from
\cite[Theorem 3.4]{Toft26},
and its proof. The details are
left for the reader.
\end{proof}

\par

Lemma \ref{Lemma:AssistMult} follows by
similar arguments as for
Lemma \ref{Lemma:AssistConv}, by using
\cite[Theorem 3.3]{Toft26} and its proof instead of
\cite[Theorem 3.4]{Toft26} and its proof. The
details are left for the reader.

\par

We have now the following convolution and multiplication
results for modulation spaces.

\par

\begin{thm}\label{Thm:MainConvMod}
Suppose $\mascB$ is an invariant QBF space on $\rr {2d}$ of order
$r_0\in (0,1]$ and $v_0\in \mascP _E(\rr {2d})$, and 
$\omega _j,v\in \mascP _E(\rr {2d})$, $j=0,1,2$, be such that
$v(x,\xi )=v_0(x,0)$ and  \eqref{Eq:ModConvWeightsEst} holds.
Then the map $(f,g)\mapsto f*g$ from $\Sigma _1(\rr d)\times \Sigma _1(\rr d)$
to $\Sigma _1(\rr d)$ is uniquely extendable to a continuous map
from $M(\omega _1,\mascB )\times W^{r_0,\infty}_{(\omega _2v)}(\rr d)$
to $M(\omega _0,\mascB )$, and
\begin{equation}\label{Eq:MainConvMod}
\begin{aligned}
\nm {f*g}{M(\omega _0,\mascB )}
&\le
C \nm f{M(\omega _1,\mascB )}
\nm g{W^{r_0,\infty}_{(\omega _2v)}},
\\[1ex]
f &\in M(\omega _1,\mascB ),\ g\in W^{r_0,\infty}_{(\omega _2v)}(\rr d),
\end{aligned}
\end{equation}
where the constant $C>0$ only depends
on the involved weight functions, $r_0$ and $\mascB$.
\end{thm}

\par

\par

\begin{thm}\label{Thm:MainMultMod}
Suppose $\mascB$ is an invariant QBF space on $\rr {2d}$ of order
$r_0\in (0,1]$ and $v_0\in \mascP _E(\rr {2d})$, and 
$\omega _j,v\in \mascP _E(\rr {d})$, $j=0,1,2$, be such that
$v(x,\xi )=v_0(0,\xi )$ and  \eqref{Eq:ModMultWeightsEst} holds.
Then the map $(f,g)\mapsto f\cdot g$ from $\Sigma _1(\rr d)\times \Sigma _1(\rr d)$
to $\Sigma _1(\rr d)$ is uniquely extendable to a continuous map
from $M(\omega _1,\mascB )\times M^{\infty ,r_0}_{(\omega _2v)}(\rr d)$
to $M(\omega _0,\mascB )$, and
\begin{equation}\label{Eq:MainMultMod}
\begin{aligned}
\nm {f\cdot g}{M(\omega _0,\mascB )}
&\le
C \nm f{M(\omega _1,\mascB )}
\nm g{W^{r_0,\infty}_{(\omega _2v)}},
\\[1ex]
f &\in M(\omega _1,\mascB ),\ g\in M^{\infty ,r_0}_{(\omega _2v)}(\rr d),
\end{aligned}
\end{equation}
where the constant $C>0$ only depends
on the involved weight functions, $r_0$ and $\mascB$.
\end{thm}

\par

\begin{proof}[Proof of Theorem \ref{Thm:MainConvMod}]
Suppose $f \in M(\omega _1,\mascB )$ and
$g\in W^{r_0, \infty}_{(\omega _2v)}(\rr d)$.
Since $M(\omega _1,\mascB )$ is continuously embedded in
$M^\infty _{(\omega _1/v_0)}(\rr d)$, in view of
Corollary \ref{Prop:BasicEmbModSp1}, Lemma
\ref{Lemma:AssistConv} shows that $f*g$ is
uniquely defined as an element in
$M^\infty _{(\omega _0/v_0)}(\rr d)$. Hence the result
follows if we prove \eqref{Eq:MainConvMod}.

\par

Let $\phi \in \Sigma _1(\rr d)\setminus 0$,
$\psi =\phi *\phi \in  \Sigma _1(\rr d)\setminus 0$,
and let $d_1=d_2=d$
and $\vartheta =0$ in \eqref{Eq:QConv}. Also let
$p_j$, $j=0,1,2$, be the same as in
Proposition \ref{Prop:ConvWienerSp}.
By Theorem \ref{Thm:EquivNorms2},
Proposition \ref{Prop:ConvWienerSp},
and Lemma \ref{Lemma:AssistConv} (2) we get
\begin{align*}
\nm {f*g}{M(\omega _0,\mascB )}
&\lesssim
\nm {|V_\phi f|*_0 |V_\phi g|}{\sfW ^{p_0,r_0}(\omega _0,\mascB)}
\\[1ex]
&\lesssim
\nm {|V_\phi f|}{\sfW ^{p_1,r_0}(\omega _1,\mascB)}
\nm {|V_\phi g|}{\sfW ^{p_2,\infty}(\omega _2v,\ell ^{r_0,\infty}_*(\zz {2d}))}
\\[1ex]
&\asymp
\nm f{M(\omega _1,\mascB )}\nm g{W^{r_0,\infty}_{(\omega _2v)}}.
\end{align*}
This gives \eqref{Eq:MainConvMod}, and thereby the result.
\end{proof}

\par

Theorem \ref{Thm:MainMultMod} follows by similar arguments
as in the proof of Theorem \ref{Thm:MainConvMod}, where
the roles of $x$ and $y$ in \eqref{Eq:QConv} are 
interchanged,
and Lemma \ref{Lemma:AssistMult} is used instead of
Lemma \ref{Lemma:AssistConv}. The details are left for the 
reader.

\par

\begin{rem}
\label{Rem:ConvClassicMod}
Theorem \ref{Thm:MainConvMod} is not sharp
when we restrict ourselves to classical modulation
spaces. For example, (1) in Lemma \ref{Lemma:AssistConv}
can be improved into that the map $(f_1,f_2)
\mapsto f_1*f_2$
is continuous from
$M^\infty _{(\omega _1)}(\rr d)\times
M^{1,\infty}_{(\omega _2)}(\rr d)$ to
$M^\infty _{(\omega _0)}(\rr d)$, provided
$\omega _j\in \mascP _E(\rr {2d})$,
$j=0,1,2$, satisfy \eqref{Eq:ModConvWeightsEst}
(see e.{\,}g. \cite{Toft04} and the analysis
in \cite{Fei1}).
These continuity properties are special
cases of \cite[Theorem 3.4]{Toft2022}. 

\par

In fact,
suppose that $\omega _j\in \mascP _E(\rr {2d})$ and
$p_j,q_j\in (0,\infty ]$,
$j=0,1,2$, satisfy \eqref{Eq:ModConvWeightsEst}
and
\begin{equation}
\label{Eq:HolderYoungCond}
\frac 1{p_0} = \frac 1{p_1}+\frac 1{p_2}
-
\max \left (
1,\frac 1{p_1},\frac 1{p_2}
\right ),
\quad
\frac 1{q_0} = \frac 1{q_1} + \frac 1{q_2}.
\end{equation}
Then it follows from
\cite[Theorem 3.4]{Toft2022}
that the map $(f_1,f_2)\mapsto f_1*f_2$
is continuous from
$M^{p_1,q_1} _{(\omega _1)}(\rr d)\times
M^{p_2,q_2}_{(\omega _2)}(\rr d)$ to
$M^{p_0,q_0} _{(\omega _0)}(\rr d)$, and
\begin{equation}
\label{Eq:ConvEstClassicModSp}
\nm {f_1*f_2}{M^{p_0,q_0}_{(\omega _0)}}
\lesssim
\nm {f_1}{M^{p_1,q_1}_{(\omega _1)}}
\nm {f_2}{M^{p_2,q_2}_{(\omega _2)}},
\quad
f_j\in M^{p_j,q_j} _{(\omega _j)}(\rr d),
\ j=1,2.
\end{equation}
We also refer to \cite{FeiGro1,Rau1}
for overlapping convolution results
established for coorbit space.
\end{rem}

\par

\section{Continuity properties for pseudo-differential
operators}
\label{sec6}

\par

In this section we make a review of some facts about
pseudo-differential operators. Thereafter we deduce
continuity for pseudo-differential operators with symbols
in suitable modulation spaces, when acting between
other modulation spaces.

\par


\par

Let $\GL (d,\mathbf R)$ be the set of $d\times d$-matrices with
entries in $\mathbf R$, $\fka \in \Sigma _1 
(\rr {2d})$, and let $A\in \GL (d,\mathbf R)$ be fixed. Then the
pseudo-differential operator $\op _A(\fka )$
is the linear and continuous operator on $\Sigma _1 (\rr d)$, given by
\begin{equation}\label{Eq:PseudoDef}
(\op _A(\fka )f)(x)
=
(2\pi  ) ^{-d}\iint \fka (x-A(x-y),\xi )f(y)e^{i\scal {x-y}\xi }\,
dyd\xi , \quad  x\in \rr d
\end{equation}
(see e.{\,}g. Chapter XVIII in \cite{Ho1}).
For general $\fka \in \Sigma _1'(\rr {2d})$, the
pseudo-differential operator $\op _A(\fka )$ is defined as the continuous
operator from $\Sigma _1(\rr d)$ to $\Sigma _1'(\rr d)$ with
distribution kernel
\begin{equation}\label{Eq:KernelPsDO}
K_{\fka ,A}(x,y)=(2\pi )^{-d/2}(\mascF _2^{-1}\fka )(x-A(x-y),x-y), 
\quad x,y \in \rr d.
\end{equation}
Here $\mascF _2F$ is the partial Fourier transform of $F(x,y)\in
\Sigma _1'(\rr {2d})$ with respect to the $y$ variable. This
definition makes sense since the mappings
\begin{equation}\label{Eq:HomeoF2tmap}
\mascF _2\quad \text{and}\quad F(x,y)\mapsto F(x-A(x-y),x-y)
\end{equation}
are homeomorphisms on
$\Sigma _s(\rr {2d})$, $\maclS _s(\rr {2d})$,
$\mascS (\rr {2d})$, and their duals, $s\ge 1$.
In particular we have the following.

\par

\begin{lemma}
\label{Lemma:SymbKernelHom}
Let $A\in \GL (d,\mathbf R )$,
and $K_{\fka ,A}$ be as in \eqref{Eq:KernelPsDO}, 
when $\fka \in \Sigma _1'(\rr {2d})$.
Then the map $\fka \mapsto K_{\fka ,A}$ is a homeomorphism on
$\Sigma _1'(\rr {2d})$, which restricts to homeomorphisms on
$$
\Sigma _s(\rr {2d}),
\quad
\maclS _s(\rr {2d}),
\quad
\mascS (\rr {2d}),
\quad
\mascS '(\rr {2d}),
\quad
\maclS _s'(\rr {2d})
\quad \text{and}\quad
\Sigma _s'(\rr {2d}).
$$
\end{lemma}

\par

The standard (Kohn-Nirenberg) representation, $\fka (x,D)=\op (\fka )$, and
the Weyl quantization $\op ^w(\fka )$ of $\fka$ are obtained by choosing
$A=0$ and $A=\frac 12 I$, respectively, in \eqref{Eq:PseudoDef}
and \eqref{Eq:KernelPsDO}, where $I$ is the identity matrix.

\par

\begin{rem}\label{Rem:BijKernelsOps}
Let $s\ge 1$. By Fourier's inversion formula, \eqref{Eq:KernelPsDO} and the kernel theorem
\cite[Theorem 2.2]{LozPerTask} for operators from
Gelfand-Shilov spaces to their duals,
it follows that the map $\fka \mapsto \op _A(\fka )$ 
is bijective from $\Sigma _s'(\rr {2d})$
to the set of all linear and continuous operators 
from $\Sigma _s(\rr d)$
to $\Sigma _s'(\rr {2d})$. The same holds true
with $\maclS _s$ or $\mascS$ in place of $\Sigma _s$
at each occurrence.
\end{rem}

\par

By Remark \ref{Rem:BijKernelsOps}, it follows that for every $\fka _1\in \Sigma _1'(\rr {2d})$
and $A_1,A_2\in \GL (d,\mathbf R)$, there is a unique $\fka _2\in \Sigma _1'(\rr {2d})$ such that
$\op _{A_1}(\fka _1) = \op _{A_2} (\fka _2)$.
Indeed, the relation between $\fka _1$
and $\fka _2$ is given by
\begin{equation}
\label{Eq:CalculiTransform}
\op _{A_1}(\fka _1) = \op _{A_2}(\fka _2)
\quad \Leftrightarrow \quad
e^{i\scal {A_2D_\xi}{D_x}}\fka _2(x,\xi )=e^{i\scal {A_1D_\xi}{D_x}}\fka _1(x,\xi )
\end{equation}
(see e.{\,}g. \cite[Section 18.5]{Ho1}
or \cite[Proposition 1.1]{Toft20}).
Here we note that the operator
$e^{i\scal {AD_\xi}{D_x}}$ is homeomorphic
on $\Sigma _s(\rr {2d})$, $\maclS _s(\rr {2d})$,
$\mascS (\rr {2d})$, and their duals
(cf. \cite{CaTo,Tr}).

\par

We also recall that $\op _A(\fka )$ is a rank-one operator, i.{\,}e.
\begin{equation}\label{Eq:RankOneSymb}
\op _A(\fka )f=(2\pi )^{-\frac d2}
(f,f_2)f_1, \qquad f\in \Sigma _1(\rr d),
\end{equation}
for some $f_1,f_2\in \Sigma _1'(\rr d)$,
if and only if $\fka$ is equal to the \emph{$A$-Wigner distribution}
\begin{equation}\label{Eq:WignerAdef}
W_{f_1,f_2}^{A}(x,\xi ) \equiv \mascF (f_1(x+A\cdo
)\overline{f_2(x-(I-A)\cdo )} )(\xi ),
\quad x,\xi \in \rr d, 
\end{equation}
of $f_1$ and $f_2$. If in addition $f_1,f_2\in L^2(\rr d)$, then $W_{f_1,f_2}^{A}$
takes the form
\begin{equation}\label{wignertdef2}
W_{f_1,f_2}^{A}(x,\xi ) =
(2\pi )^{-\frac d2}
\int _{\rr d} f_1(x+Ay)
\overline{f_2(x-(I-A)y)}e^{-i\scal y\xi} \, dy,
\quad x,\xi \in \rr d.
\end{equation}
(Cf. e.{\,}g. \cite{Toft20}.) Since
the Weyl case is peculiar interesting,
we also set $W_{f_1,f_2}=W_{f_1,f_2}^{A}$
when $A=\frac 12 I$. A straight-forward combination of
\eqref{Eq:CalculiTransform} and the fact that $\fka$ in \eqref{Eq:RankOneSymb}
equals $W_{f_1,f_2}^{A}$ gives
$$
e^{i\scal {A_2D_\xi}{D_x}}W_{f_1,f_2}^{A_2}
=
e^{i\scal {A_1D_\xi}{D_x}}W_{f_1,f_2}^{A_1}.
$$

\par


\par

We aim at extending the following result which is a special case of Theorem 3.1 in
\cite{Toft18}. Here the involved weight functions should be related as
\begin{equation}
\label{Eq:WeightPseudoRel}
\frac {\omega _2(x,\xi )}{\omega _1(y,\eta)}
\lesssim
\omega (x+A(y-x),\eta +A^*(\xi -\eta ),
\xi -\eta ,y-x),
\quad x,y,\xi ,\eta \in \rr d.
\end{equation}

\par

\begin{thm}
\label{Thm:PseudoCont1}
Suppose $A\in \GL (d,\mathbf R)$, $r_0\in (0,1]$, $p,q\in [r_0,\infty ]$,
$\omega \in \mascP _E(\rr {4d})$,
and $\omega _1,\omega _2\in \mascP _E(\rr {2d})$, be such that
\eqref{Eq:WeightPseudoRel} holds. If
$\fka \in M^{\infty ,r_0}_{(\omega _0)}(\rr {2d})$, then $\op _A(\fka )$
from $\Sigma _1(\rr d)$ to $\Sigma _1'(\rr d)$ is uniquely
extendable to a continuous map from
$M^{p,q}_{(\omega _1)}(\rr d)$ to $M^{p,q}_{(\omega _2)}(\rr d)$,
and
\begin{equation}
\label{Eq:PseudoCont1}
\begin{aligned}
\nm {\op _A(\fka )f}{M^{p,q}_{(\omega _2)}}
&\le
C\nm {\fka }{M^{\infty ,r_0}_{(\omega _0)}}\nm f{M^{p,q}_{(\omega _1)}},
\\[1ex]
\fka &\in M^{\infty ,r_0}_{(\omega _0)}(\rr {2d}),\ 
f\in M^{p,q}_{(\omega _1)}(\rr d),
\end{aligned}
\end{equation}
for some constant $C>0$, which only depends on
$r_0$ and $\omega _j$, $j=0,1,2$.
\end{thm}

\par

Our extension of the previous result is the following.
Here the condition \eqref{Eq:WeightPseudoRel} is replaced by
\begin{multline}
\tag*{(\ref{Eq:WeightPseudoRel})$'$}
\frac {\omega _2(x,\xi )v_0(x-y,\xi -\eta )}
{\omega _1(y,\eta)}
\lesssim
\omega _0(x+A(y-x),\eta +A^*(\xi -\eta ),\xi -\eta ,y-x),
\\[1ex]
x,y,\xi ,\eta \in \rr d.
\end{multline}

\par


\begin{thm}
\label{Thm:PseudoCont2}
Suppose $A\in \GL (d,\mathbf R)$, $\mascB$ is a normal
QBF space on $\rr {2d}$ with respect to $r_0\in (0,1]$ and
$v_0\in \mascP _E(\rr {2d})$,
$\omega _0\in \mascP _E(\rr {4d})$,
and $\omega _1,\omega _2\in \mascP _E(\rr {2d})$, be such that
\eqref{Eq:WeightPseudoRel}$'$ holds. If
$\fka \in M^{\infty ,r_0}_{(\omega _0)}(\rr {2d})$,
then $\op _A(\fka )$
from $\Sigma _1(\rr d)$ to $\Sigma _1'(\rr d)$ is uniquely
extendable to a continuous map from
$M(\omega _1,\mascB )$ to $M(\omega _2,\mascB )$,
and
\begin{equation}
\label{Eq:PseudoCont2}
\begin{aligned}
\nm {\op _A(\fka )f}{M(\omega _2,\mascB )}
&\le
C\nm {\fka}{M^{\infty ,r_0}_{(\omega _0)}}\nm f{M(\omega _1,\mascB )},
\\[1ex]
\fka &\in M^{\infty ,r_0}_{(\omega _0)}(\rr {2d}),\ 
f\in M(\omega _1,\mascB ),
\end{aligned}
\end{equation}
for some constant $C>0$, which only depends on
$r_0$ and $\omega _j$, $j=0,1,2$.
\end{thm}

\par

For the proof we observe that if $\fka \in \Sigma _1'(\rr {2d})$
and $K_{\fka ,A}$ is given by \eqref{Eq:KernelPsDO}, then
straight-forward applications of the Fourier inversion formula
gives
\begin{equation}
\label{Eq:STFTKernelPsDO}
\begin{aligned}
(V_\Phi K_{\fka ,A} )(x,y,\xi ,\eta )
&=
(2\pi )^{-\frac d2}e^{i\scal {y-x}{A^*(\xi +\eta )-\eta}}
(V_\Psi \fka )(T_A(x,\xi ,\eta ,y)),
\\[1ex]
T_A(x,\xi ,\eta ,y)
&=
(x+A(y-x),A^*(\xi +\eta )-\eta ,\xi +\eta ,y-x),
\\[1ex]
\Psi
&=
K_{\Phi ,A},
\qquad
a\in \Sigma _1'(\rr {2d}),\ 
\Phi \in \Sigma _1(\rr {2d}).
\end{aligned}
\end{equation}
(See \cite[Lemma 3.12]{AbCaTo}.)
By Moyal's identity \eqref{Eq:Moyal},
it also follows that if in addition $\Psi$ is
chosen such that $\Phi$ in 
\eqref{Eq:STFTKernelPsDO} is given
by
\begin{equation}\label{Eq:TensorChoice}
\Phi =\phi \otimes \phi ,
\quad
\nm \phi{L^2}=1
\quad \text{and}\quad
0\le \phi \in \Sigma _1(\rr d),
\end{equation}
then
\begin{multline}
\label{Eq:STFTKernelAction}
\big (V_\phi (\op _A(\fka )f) \big )(x,\xi )
\\[1ex]
=
\iint _{\rr {2d}} \big ( V_\Phi K_{\fka ,A}\big )(x,y,\xi ,-\eta )
\cdot (V_\phi f)(y,\eta )\, dyd\eta .
\end{multline}

\par

\begin{proof}[Proof of Theorem
\ref{Thm:PseudoCont2}]
Let $\fka \in M^{\infty ,r_0}
_{(\omega _0)}(\rr {2d})$,
and
$$
\vartheta _j(x,\xi ) = \frac {\omega _j(x,\xi )}{v_0(x,\xi )},
\quad x,\xi \in \rr d, 
\quad
j=1,2.
$$
We claim that $\op _A(\fka )$
from $\Sigma _1(\rr d)$ to $\Sigma _1'(\rr d)$ extends uniquely to a continuous
operator from $M^\infty _{(\vartheta _1)}(\rr d)$ to $M^\infty _{(\vartheta _2)}(\rr d)$.

\par

In fact, since $v_0$ is even, we obtain
\begin{align*}
\frac {\vartheta _2(x,\xi )}{\vartheta _1(y,\eta )}
&=
\frac {\omega _2(x,\xi )v_0(y,\eta )}{\omega _1(y,\eta )v_0(x,\xi )}
\lesssim
\frac {\omega _2(x,\xi )}{\omega _1(y,\eta )}v_0(x-y,\xi -\eta )
\\[1ex]
&\lesssim
\omega _0(x+A(y-x),\eta +A^*(\xi -\eta ),\xi -\eta ,y-x), 
\quad x,\xi \in \rr d.
\end{align*}
Hence \eqref{Eq:WeightPseudoRel} holds with $\vartheta _j$
in place of $\omega _j$, at each occurrence, $j=1,2$. The
claim now follows from Theorem \ref{Thm:PseudoCont1}.

\par

Since $M(\omega _1,\mascB ) \hookrightarrow M^\infty _{(\vartheta _1)}(\rr d)$,
in view of Corollary \ref{Prop:BasicEmbModSp1}, it follows from the previous
claim that $\op _A(\fka )$ is uniquely defined as a continuous
operator from $M(\omega _1,\mascB )$ to $M^\infty _{(\vartheta _2)}(\rr d)$.
The result therefore follows if we prove \eqref{Eq:PseudoCont2}.

\par

Let $f\in M(\omega _1,\mascB )$, $\phi$, $\Phi$ and $\Psi$
be chosen such that
\eqref{Eq:STFTKernelPsDO} and \eqref{Eq:TensorChoice}
are fulfilled. Also let
\begin{align*}
F(x,\xi ) &= |V_\phi f(x,\xi )|\cdot \omega _1(x,\xi ),
\quad
G(x,\xi ) = |V_\phi (\op _A(\fka )f)(x,\xi )|\cdot \omega _2(x,\xi ),
\intertext{ for $x,\xi \in \rr d$ and}
H(y,\eta ) &= \sup _{x,\xi \in \rr d}
\left (|V_\Psi \fka (x,\xi ,\eta ,-y)|\cdot
\omega _0(x,\xi ,\eta ,-y)\right ), 
\quad x, \eta \in \rr d.
\end{align*}
Then \eqref{Eq:STFTKernelPsDO} and
\eqref{Eq:STFTKernelAction} gives
\begin{equation}\label{Eq:FGHEst}
\begin{aligned}
0\le G(X)
&\lesssim
\int _{\rr {2d}} H(X-Y)v_0(X-Y)^{-1}F(Y)\, dY
\\[1ex]
&= ((H/v_0)*F)(X),
\\[1ex]
X &= (x,\xi )\in \rr {2d},\ Y = (y,\eta )\in \rr {2d}.
\end{aligned}
\end{equation}

\par

If $p_0$, $p_1$ and $p_2$ satisfy \eqref{Eq:YoungCond}, then
a combination of Theorem \ref{Thm:EquivNorms2},
Corollary \ref{Cor:ConvWienerSp1} and
\eqref{Eq:FGHEst} gives
\begin{align*}
\nm {\op _A(\fka )f}{M(\omega _2,\mascB )}
&\asymp
\nm G{\sfW ^{p_0}(\mascB)}
\lesssim
\nm {(H/v_0)*F}{\sfW ^{p_0}(\mascB)}
\\[1ex]
&\lesssim
\nm {H/v_0}{\sfW ^{p_2}(v_0,\ell ^{r_0})}
\nm F{\sfW ^{p_1}(\mascB )}
\\[1ex]
&\lesssim
\nm {V_\Psi \fka}{\sfW ^{\infty ,p_2}(\omega _0,\ell ^{\infty ,r_0})}
\nm {V_\phi f}{\sfW ^{p_1}(\omega _1,\mascB )}
\\[1ex]
&\asymp
\nm {\fka}{M^{\infty ,r_0}_{(\omega _0)}}\nm f{M(\omega _1,\mascB )},
\end{align*}
which gives \eqref{Eq:PseudoCont2}, and thereby the result.
\end{proof}

\par

\begin{rem}
\label{Rem:OtherPseudoResults}
It seems that Theorem \ref{Thm:PseudoCont2}
is the most general result so far concerning
pseudo-differential operators with symbols in
weighted $M^{\infty ,r_0}$ spaces, when acting
on modulation spaces. On the other hand,
Theorem \ref{Thm:PseudoCont2} is not suitable,
or even applicable when dealing with
pseudo-differential operators with symbols in
weighted $M^{p,q}$ spaces, for \emph{general}
$p,q\in (0,\infty ]$. A classical result here is given in
\cite{GroHei2,Toft04}, which in extended form in \cite{Toft20}
state that if $\omega _j$ are the same as in
Theorem \ref{Thm:PseudoCont1}, $j=0,1,2$,
and $p,p_1,p_2,q,q_1,q_2\in [1,\infty ]$ satisfy
\begin{equation}
\label{Eq:GenLebCond1}
1-\frac 1{p}-\frac 1{q}
=
\frac 1{p_1}-\frac 1{p_2}
=
\frac 1{q_1}-\frac 1{q_2},
\quad
q\le p_2,q_2\le p,
\end{equation}
and $\fka \in M^{p,q}_{(\omega _0)}(\rr {2d})$,
then $\op _A(\fka )$ is continuous from
$M^{p_1,q_1}_{(\omega _1)}(\rr d)$ to
$M^{p_2,q_2}_{(\omega _2)}(\rr d)$. By combining
this result with the embeddings in Proposition
\ref{Prop:EmbModSp}, Cordero and Nicola
in \cite{CorNic2} were able
to improve the latter result by proving that the same
continuity properties hold after the condition
\eqref{Eq:GenLebCond1} is relaxed into
\begin{equation}
\label{Eq:GenLebCond2}
1-\frac 1{p}-\frac 1{q} \le \frac 1{p_1}-\frac 1{p_2},
\quad
1-\frac 1{p}-\frac 1{q}\le \frac 1{q_1}-\frac 1{q_2},
\quad
q\le p_2,q_2\le p.
\end{equation}
Furthermore, for trivial weights, they also proved that
\eqref{Eq:GenLebCond2} is sharp. That is
\eqref{Eq:GenLebCond2} is not only sufficient,
but also necessary in order for $\op _A(\fka )$
to be continuous from $M^{p_1,q_1}(\rr d)$ to
$M^{p_2,q_2}(\rr d)$, for every $\fka \in M^{p,q}(\rr {2d})$.

\par

The sufficiently part in \cite{CorNic2} were further
generalized in \cite{GuRaToUs} to the case of Orlicz
modulation spaces.

\par

There are also several results on Schatten-von
Neumann properties for
pseudo-differential operators with symbols in
modulation spaces, when acting on other
modulation spaces (see e.{\,}g.
\cite{GroHei1,Sjo,Toft04,Toft18,Toft22} and the references
therein).
\end{rem}

\par

\section{Continuity properties for
Toeplitz operators}
\label{sec7}

\par

In this section we make a review of some facts about
Toeplitz operators. Thereafter we deduce
continuity for Toeplitz operators with symbols
in suitable modulation spaces, when acting 
between other modulation spaces. As in
\cite{CorGro,Toft04,GroTof1,Toft10,ToBo08}, 
the main idea is to
express the Toeplitz operators as
pseudo-differential operators
with symbols being a convolution
between the Toeplitz symbol
and a convenient function.

\par

%

Let $\fka \in \Sigma _1 (\rr {2d})$. 
Then the Toeplitz 
operator $\tp _{\phi _1,\phi _2}(\fka)$,
with symbol $\fka$, and window functions 
$\phi _1,\phi _2\in
\Sigma _1 (\rr d)$, is the continuous
operator on $\Sigma _1(\rr d)$, defined by
the formula
\begin{equation}\label{Eq:ToepDef}
(\tp _{\phi _1,\phi _2}(\fka )f_1,f_2)
_{L^2(\rr d)}
= 
(\fka \cdot V_{\phi _1}f_1,
V_{\phi _2}f_2)_{L^2(\rr {2d})}
\end{equation}
when $f_1,f_2\in \Sigma _1 (\rr d)$.
The definition of $\tp _{\phi _1,\phi _2}(\fka )$
extends uniquely to a continuous operator from
$\Sigma _1'(\rr d)$ to $\Sigma _1(\rr d)$.
For convenience we write
$\tp _{\phi _1}(\fka )$ 
for
$\tp _{\phi _1,\phi _1}(\fka )$.

\par

We may interpret the operator in 
\eqref{Eq:ToepDef}
as a pseudo-differential operator, using
the fact that
\begin{equation}\label{Eq:ToeplWeyl}
\begin{aligned}
\operatorname{Tp}_{\phi _1,\phi _2}(\fka)
&=
\op ^w(\fka \ast u)
\quad \text{when}
\\[1ex]
u(X) &= (2\pi)^{-\frac d2}
W_{\phi _2,\phi _1}(-X),
\end{aligned}
\end{equation}
which follows by straight-forward application of
Fourier's inversion formula. By using
\eqref{Eq:ToepDef} and \eqref{Eq:ToeplWeyl},
the map
\begin{equation}
\label{Eq:ToeplMapSymbWind}
(\fka ,\phi _1,\phi _2)\mapsto \tp
_{\phi _1,\phi _2}(\fka )
\end{equation}
from $\Sigma _1(\rr {2d})\times
\Sigma _1(\rr d)\times \Sigma _1(\rr d)$
to $\maclL (\Sigma _1(\rr d),\Sigma _1(\rr d))$
extends in several ways, e.{\,}g. from
$\Sigma _1'(\rr {2d})\times
\Sigma _1(\rr d)\times \Sigma _1(\rr d)$
to $\maclL (\Sigma _1(\rr d),\Sigma _1(\rr d))$.
(cf. e.{\,}g. 
\cite{CorGro,Fol,GroTof1,Toft04,Toft20,ToBo08}).

\par

\begin{rem}
As in \cite{GroTof1,AbCoTo}, our
most part of our analysis
of Toeplitz operators is based on the
pseudo-differential operator representation 
given by \eqref{Eq:ToeplWeyl}.
Furthermore, any extension of the definition
of Toeplitz operators to cases which are not 
covered by straight-forward extensions
of the definition in \eqref{Eq:ToepDef},
is based on this representation.
Roughly speaking, in main part of our
analysis of Toeplitz operators, the
reader may consider \eqref{Eq:ToeplWeyl}
as the definition of Toeplitz operators.

\par

We observe that
this leads to situations in
which certain mapping properties for the
pseudo-differential operator
representation make sense, whereas similar 
interpretations are difficult or impossible
to make in the framework of \eqref{Eq:ToepDef}.
%
%
%
\end{rem}

\par

We shall combine \eqref{Eq:ToeplWeyl}
with Theorem \ref{Thm:MainConvMod}
and Theorem \ref{Thm:PseudoCont2} 
in various contexts to extend the 
definition of Toeplitz operators,
and to deduce
continuity properties for such
operators. Especially
we perform such extensions to
allow symbols to belong to suitable
modulation
spaces and the window functions
to belong to admissible
weighted $M^{r_0}$ spaces. Here
$r_0\in (0,1]$ is the same as in Theorem
\ref{Thm:MainConvMod} and
Theorem \ref{Thm:PseudoCont2}.

\par

The following lemma explains some
properties of the Wigner distribution
in the convolution in
\eqref{Eq:ToeplWeyl}.

\par

\begin{lemma}\label{Lemma:PropWigDistr}
Let $A\in \GL (d,\mathbf R)$, $p\in (0,\infty ]$,
$\phi _j,\psi _j
\in
\Sigma _1'(\rr d)$, $j=1,2$,
$\Phi =W_{\phi _2,\phi _1}^A$
and
$\Psi =W_{\psi _2,\psi _1}^A$.
Also let
$\vartheta _1,\vartheta _2\in
\mascP _E(\rr {2d})$
and
\begin{equation}
\label{Eq:STFTWignerWeight}
\omega (x,\xi ,\eta ,y)
=
\vartheta _1
(x+(I-A)y,\xi -A^*\eta )
\vartheta _2
(x-Ay,\xi +(I-A^*)\eta )
\end{equation}
Then the following is true:
\begin{enumerate}
\item it holds
\begin{multline*}
V_\Psi \Phi (x,\xi ,\eta ,y)
\\[1ex]
= e^{-i\scal y\xi}
\overline{V_{\psi _1}\phi _1
(x+(I-A)y,\xi -A^*\eta )}
V_{\psi _2}\phi _2
(x-Ay,\xi +(I-A^*)\eta ),
\end{multline*}
with equality in $\Sigma _1'(\rr {4d})$;

\vrum

\item if, more restrictive,
$\phi _j\in M^p_{(\vartheta _j)}(\rr d)$,
$j=1,2$, then $\Phi
\in M^p_{(\omega )}(\rr {2d})$, and
$$
\nm \Phi{M^p_{(\omega )}}
\asymp
\nm {\phi _1}{M^p_{(\vartheta _1)}}
\nm {\phi _2}{M^p_{(\vartheta _2)}}.
$$
\end{enumerate}
\end{lemma}

\par

The result can essentially be found
in e.{\,}g. \cite{Gro2} (see also e.{\,}g. 
Proposition 2.4 and Lemma 2.6
in \cite{Toft20}).
The details are left for the reader.

%
%

\par

The function spaces to which the Toeplitz and
Weyl symbol of the
Toeplitz operator
$\tp _{\phi _1,\phi _2}(\fka)$ belong to,
are connected in the following way.


\par

\begin{thm}
\label{Thm:ToeplPseudoSymbClassMod}
Let $q,r_0,r\in (0,\infty ]$ be such that
\begin{equation}
\label{Eq:LebExpToeplToPseudo}
1\le \frac 1r\le \frac 1{r_0}
=
\frac 1{q}+\frac 1{r},
\quad \text{or}\quad
\frac 12 \le \frac 1r= \frac 1{r_0}
=1-\frac 1{2q},
\end{equation}
and let $\omega _0,\omega
\in \mascP _E(\rr {4d})$
and $\vartheta _1,\vartheta _2
\in \mascP _E(\rr {2d})$,
be such that
\begin{multline}
\label{Eq:WeylToeplWeightCondClass}
\omega_0(x_1-x_2, \xi_1-\xi_2,\eta ,y)
\\[1ex]
\lesssim 
\omega (x_1,\xi_1,\eta ,y)
\cdot \vartheta _1 (x_2 -{\textstyle{\frac y2}}
,\xi_2 + {\textstyle{\frac \eta 2}})
\cdot \vartheta _2 (x_2+{\textstyle{\frac y2}} ,
\xi_2 - {\textstyle{\frac \eta 2}}),
\end{multline}
where 
$x_j, y, \xi_j, \eta \in \rr d$, $j=1,2$.
Then the map \eqref{Eq:ToeplMapSymbWind}
from $M^{\infty ,q}_{(\omega )}(\rr {2d})
\times \Sigma _1(\rr d)
\times \Sigma _1(\rr d)$ to
$\op ^w(\Sigma _1'(\rr {2d}))$
is uniquely extendable to a
continuous map
from $M^{\infty ,q}_{(\omega )}(\rr {2d})
\times M^{r}_{(\vartheta _1)}(\rr d)
\times M^{r}_{(\vartheta _2)}(\rr d)$
to $\op ^w(M^{\infty ,r_0}_{(\omega _0)}
(\rr {2d}))$.
\end{thm}

\par

\begin{proof}
First suppose that the second condition
in \eqref{Eq:LebExpToeplToPseudo} holds.
Then $q,r,r_0\ge 1$, and the result
follows from \cite[Lemma 5.5]{AbCoTo}.

\par

Next suppose that the first condition in
\eqref{Eq:LebExpToeplToPseudo} holds,
and let $\phi _j\in
M^{r}_{(\vartheta _j)}(\rr d)$, $j=1,2$.
By Lemma \ref{Lemma:PropWigDistr}, 
$W_{\phi _2,\phi _1} \in
M ^{r} _{(\omega _3)} (\rr {2d})$,
or equivalently,
$ u \in M ^{r} _{(\omega _4)} (\rr {2d})$,
when
\begin{align} \label{eq:omega_3}
\omega _3(x,\xi ,\eta ,y)
=
\vartheta _1
(x+{\textstyle{\frac y2}},
\xi -{\textstyle{\frac \eta 2}})
\cdot
\vartheta _2
(x-{\textstyle{\frac y2}},
\xi +{\textstyle{\frac \eta 2}}),
\quad 
x,y, \xi, \eta \in \rr d, 
\end{align}
\begin{equation}
\label{Eq:omega_4}
 u(X)\equiv (2\pi)^{-\frac d2} \cdot 
W_{\phi _2,\phi _1}(-X)
\quad \text { and } \quad
\omega _4(X)\equiv \omega _3(-X), 
\quad
X \in \rr {2d}.
\end{equation}
Due to Remark \ref{Rem:ConvClassicMod}
it follows that $\fkb 
=\fka *u\in M^{\infty ,r_0}_{(\omega _0)}
(\rr {2d})$.

\par

Now suppose that $\phi _1,\phi _2
\in \Sigma _1(\rr d)$, and let
$$
\maclR
=
\sets {\op ^w(\fkb)}{\fkb \in
M^{\infty ,r_0}_{(\omega _0)}(\rr {2d})},
$$
with quasi-norm
\begin{equation}
\label{Eq:OpSpecNormClass}
\nm T{\maclR} = \nm {\fkb}
{M^{\infty ,r_0}_{(\omega _0)}},
\qquad
T=\op ^w(\fkb ).
\end{equation}
Then $\maclR$ is contained in the set of 
all linear and continuous operators from 
$\Sigma _1(\rr d)$
to $\Sigma _1'(\rr d)$.
By identifying linear and continuous 
operators with their distribution
kernels, using Schwartz kernel theorem
(see e.{\,}g. Remark 
\ref{Rem:BijKernelsOps}),
it follows from \eqref{Eq:OpSpecNormClass}
that $\maclR$ is  continuously
embedded in $\Sigma _1'(\rr {2d})$.

\par

Since $\phi _1,\phi _2\in \Sigma _1(\rr d)$ we have
$\tp _{\phi _1,\phi _2}(\fka )\in \maclR$
when $\fka \in M^{\infty ,q}_{(\omega )}(\rr {2d})$,
and
\begin{equation}
\label{Eq:ToeplSymbWindEstClass}
\nm {\tp _{\phi _1,\phi _2}(\fka )}{\maclR}
\lesssim
\nm {\fka}{M^{\infty ,q}_{(\omega )}}
\nm {\phi _1}{M^{r}(\vartheta _1)}
\nm {\phi _2}{M^{r}(\vartheta _2)}.
\end{equation}
In fact, by Remark \ref{Rem:ConvClassicMod} and
Lemma \ref{Lemma:PropWigDistr} we obtain
\begin{align}
\nm {\tp _{\phi _1,\phi _2}(\fka )}{\maclR}
&\asymp
\nm {\op ^w(\fka *\check W_{\phi _2,\phi _1})}{\maclR}
\notag
\\[1ex]
&\lesssim
\nm {\fka}{M^{\infty ,q}_{(\omega )}}
\nm {W_{\phi _2,\phi _1}}{M^{r}_{(\omega _4)}}
\notag
\\[1ex]
&=
\nm {\fka}{M^{\infty ,q}_{(\omega )}}
\nm {\phi _1}{M^{r}_{(\vartheta _1)}}
\nm {\phi _2}{M^{r}_{(\vartheta _2)}},
\label{Eq:ToeplSymbWindEstClass2}
\end{align}
and \eqref{Eq:ToeplSymbWindEstClass} 
follows. Here, $\check f$ is defined by 
$\check f(x)=f(- x)$ when
$x \in \rr d$ and $f$ is 
a distribution on $\rr d$.

\par

The extension of the map \eqref{Eq:ToeplMapSymbWind}
to a continuous map from
$M^{\infty ,q}_{(\omega )}(\rr {2d})
\times M^{r}_{(\vartheta _1)}(\rr d)
\times M^{r}_{(\vartheta _2)}(\rr d)$ to $\maclR$
now follows from Remark \ref{Rem:ConvClassicMod},
Lemma \ref{Lemma:PropWigDistr} and
\eqref{Eq:ToeplSymbWindEstClass2}.

\par

Since $r<\infty$ it follows that
$\Sigma _1(\rr d)$ is dense in
$M^{r}_{(\vartheta _j)}(\rr d)$, $j=1,2$.
Hence, a combination of \eqref{Eq:OpSpecNormClass} and
\eqref{Eq:ToeplSymbWindEstClass} shows that the map
\eqref{Eq:ToeplMapSymbWind} from
$M(\omega ,\mascB)\times \Sigma _1(\rr d)
\times \Sigma _1(\rr d)$ to $\maclR$, is uniquely
extendable to a continuous map from
$M(\omega ,\mascB)\times M^{r_0}_{(\vartheta _1)}(\rr d)
\times M^{r_0}_{(\vartheta _2)}(\rr d)$ to $\maclR$.
This gives the result.
\end{proof}

\par

\begin{thm}
\label{Thm:ToeplPseudoSymbMod}
Let 
$
\mascB
$
be an invariant QBF space on $\rr {4d}$
with respect to 
$r_0 \in (0,1]$
and 
$v_0 \in \mascP_E(\rr {4d} )$. 
Also let 
$\omega _0,\omega \in \mascP _E (\rr {4d})$, 
$\vartheta _1, \vartheta _2 \in \mascP_E(\rr {2d})$
be such that
\begin{multline}
\label{Eq:WeylToeplWeightCond}
\omega_0(x_1-x_2, \xi_1-\xi_2,\eta ,y)
v_0(x_2,\xi_2,0,0)
\\[1ex]
\lesssim 
\omega (x_1,\xi_1,\eta ,y)
\cdot \vartheta _1 (x_2 -{\textstyle{\frac y2}}
,\xi_2 + {\textstyle{\frac \eta 2}})
\cdot \vartheta _2 (x_2+{\textstyle{\frac y2}} ,
\xi_2 - {\textstyle{\frac \eta 2}}).
\end{multline}
Then the map \eqref{Eq:ToeplMapSymbWind}
from $M(\omega ,\mascB )\times \Sigma _1(\rr d)
\times \Sigma _1(\rr d)$ to $\op ^w(\Sigma _1'(\rr {2d}))$
is uniquely extendable to a continuous map
from $M(\omega ,\mascB )
\times M^{r_0}_{(\vartheta _1)}(\rr d)
\times M^{r_0}_{(\vartheta _2)}(\rr d)$
to $\op ^w(M(\omega _0,\mascB))$.
\end{thm}

\par

\begin{proof}
Due to Lemma \ref{Lemma:PropWigDistr} 
we have 
$W_{\phi _2,\phi _1} \in
M ^{r_0} _{(\omega _3)} (\rr {2d})$,
where $u, \omega _3,$ and $\omega _4$ are given in 
\eqref{eq:omega_3} and \eqref{Eq:omega_4}.
An application of Proposition \ref{Prop:EmbModSp}
gives
$ u \in M ^{r_0} _{(\omega _4)} (\rr {2d})
\subseteq W  ^{r_0, \infty} _{(\omega _4)} (\rr {2d})
$.
By letting
$$
\omega _1=\omega
\quad \text{and}\quad
\omega _2(x,\xi, \eta, y)
=
\frac{\omega _4(x,\xi, \eta, y)}
{v_0(x, \xi, 0, 0)},
\quad x,y,\xi ,\eta \in \rr d,
$$
in Theorem \ref{Thm:MainConvMod}, it follows that
$\fkb =\fka *u\in M(\omega_0, \mascB)$.

\par

Now suppose that $\phi _1,\phi _2\in \Sigma _1(\rr d)$,
and let
$$
\maclR
=
\sets {\op ^w(\fkb)}{\fkb \in M(\omega _0,\mascB)},
$$
with quasi-norm
\begin{equation}
\label{Eq:OpSpecNorm}
\nm T{\maclR} = \nm {\fkb}{M(\omega _0,\mascB)},
\qquad
T=\op ^w(\fkb ).
\end{equation}
Then $\maclR$ is contained in the set of all
linear and continuous operators from $\Sigma _1(\rr d)$
to $\Sigma _1'(\rr d)$. Moreover, due to Schwartz
kernel theorem it follows that $\maclR$ is continuously
embedded in $\Sigma _1'(\rr {2d})$.

\par

Since $\phi _1,\phi _2\in \Sigma _1(\rr d)$ we have
$\tp _{\phi _1,\phi _2}(\fka )\in \maclR$
when $\fka \in M(\omega ,\mascB)$, and
\begin{equation}
\label{Eq:ToeplSymbWindEst}
\nm {\tp _{\phi _1,\phi _2}(\fka )}{\maclR}
\lesssim
\nm {\fka}{M(\omega ,\mascB)}
\nm {\phi _1}{M^{r_0}(\vartheta _1)}
\nm {\phi _2}{M^{r_0}(\vartheta _2)}.
\end{equation}
In fact, by Theorem \ref{Thm:MainConvMod} and
Lemma \ref{Lemma:PropWigDistr} we obtain
\begin{align*}
\nm {\tp _{\phi _1,\phi _2}(\fka )}{\maclR}
&\asymp
\nm {\op ^w(\fka *\check W_{\phi _2,\phi _1})}{\maclR}
\\[1ex]
&\lesssim
\nm {\fka}{M(\omega ,\mascB)}
\nm {W_{\phi _2,\phi _1}}{M^{r_0}_{(\omega _4)}}
\\[1ex]
&=
\nm {\fka}{M(\omega ,\mascB)}
\nm {\phi _1}{M^{r_0}(\vartheta _1)}
\nm {\phi _2}{M^{r_0}(\vartheta _2)},
\end{align*}
and \eqref{Eq:ToeplSymbWindEst} follows.

\par

Since $\Sigma _1(\rr d)$ is dense in
$M^{r_0}_{(\vartheta _j)}(\rr d)$, $j=1,2$,
it follows from \eqref{Eq:OpSpecNorm} and
\eqref{Eq:ToeplSymbWindEst} that the map
\eqref{Eq:ToeplMapSymbWind} from
$M(\omega ,\mascB)\times \Sigma _1(\rr d)
\times \Sigma _1(\rr d)$ to $\maclR$, is uniquely
extendable to a continuous map from
$M(\omega ,\mascB)\times M^{r_0}_{(\vartheta _1)}(\rr d)
\times M^{r_0}_{(\vartheta _2)}(\rr d)$ to $\maclR$.
This gives the result.
\end{proof}

\par

\begin{rem}
\label{Rem:ToeplPseudoSymbMod}
Let $r_0$, $\mascB$, $\omega$, $\omega _0$ and
$\vartheta _j$, $j=1,2$,
be the same as in Theorem \ref{Thm:ToeplPseudoSymbMod}.
If 
$\fka \in M(\omega , \mascB)$ and
$\phi _j  \in M ^{r _0} _{(\vartheta _j)}(\rr d)$, 
$j=1,2$, 
then Theorem \ref{Thm:ToeplPseudoSymbMod} and its
proof show that
there is $\fkb \in M(\omega_0, \mascB)$ such that
\begin{align*}
    \tp _{\phi _1,\phi _2}(\fka)
    = 
    \op ^w(\fkb),
\end{align*}
and
\begin{align*} 
    \nm {\fkb}{M(\omega _0, \mascB)} 
    \lesssim
    \nm {\fka}{M(\omega, \mascB)}
    \nm {\phi_1}{M ^{r_0} _{(\vartheta _1)}}
    \nm {\phi_2}{M ^{r_0} _{(\vartheta _2)}}.
\end{align*}
\end{rem}

\par

\begin{thm}
\label{Thm:GenContResForToepl}
Let $p,q,r_0,r\in (0,\infty ]$ be such that
\eqref{Eq:LebExpToeplToPseudo} holds, and suppose
$
\mascB
$
is an invariant QBF space on $\rr {2d}$
with respect to 
$r_0$
and 
$v_0 \in \mascP_E(\rr {2d} )$,
$\omega \in \mascP _E (\rr {4d})$,
and
$\omega _1, \omega _2, \vartheta_1, \vartheta _2
\in \mascP _E (\rr {2d})$, 
are such that
\begin{multline}
\label{Eq:ToepOpCond2}
    \frac {\omega _2 (x-z, \xi - \zeta)} 
        {\omega _1 (x-y, \xi-\eta)}
    \cdot
    v_0(y-z, \eta-\zeta)
\\[1ex]
    \lesssim
    \omega (x,\xi, \eta-\zeta, z-y)
    \vartheta _1 (y,\eta )
    \vartheta _2 (z, \zeta)
\end{multline}
for $x,y, z, \xi, \eta, \zeta \in \rr d$. 
Also suppose
$
\phi _j \in 
M ^{r} _{(\vartheta _j)} (\rr d)$ 
and
$
\fka \in 
M ^{\infty ,q} _{(\omega)}( \rr {2d})
$, $j =1,2$.
Then 
$
\operatorname{Tp}_{\phi _1,\phi _2}(\fka)
$
from $\Sigma _1(\rr d)$ to $\Sigma _1'(\rr d)$
extends uniquely to a continuous map from 
$M(\omega_1, \mascB)$ 
to 
$M(\omega_2, \mascB)$,
and 
\begin{equation}
\label{Eq:ToeplitzCont}
\begin{gathered}
\nm {\tp _{\phi _1,\phi _2}(\fka)f}
{M(\omega _2,\mascB )}
\le
C\nm {\fka}{M^{\infty ,q}_{(\omega)}}
\nm f{M(\omega _1,\mascB )}
\nm {\phi _1}{M^r_{(\vartheta _1)}}
\nm {\phi _2}{M^r_{(\vartheta _2)}},
\\[1ex]
\fka \in M^{\infty ,q}_{(\omega )}(\rr {2d}),\ 
f\in M(\omega _1,\mascB ),\ 
\phi _j\in M^r_{(\vartheta _j)}(\rr d),\ j=1,2.
\end{gathered}
\end{equation}
for some constant $C>0$, which only depends on
$\mascB$, $\omega$, and $\omega _j$, $j=1,2$.
\end{thm}

\par


\par

\begin{proof}
Let
\begin{equation}
\label{Eq:AssignWeightFunc}
\omega _0(x,\xi ,\eta ,y)
=
\frac {\omega _2(x-\frac y2,\xi +\frac \eta 2)
v_0(-y,\eta )}
{\omega _1(x+\frac y2,\xi -\frac \eta 2)}.
\end{equation}
A combination of \eqref{Eq:ToepOpCond2} and
\eqref{Eq:AssignWeightFunc} shows that 
\eqref{Eq:WeylToeplWeightCondClass} holds.
By Theorem \ref{Thm:ToeplPseudoSymbClassMod}
%
we obtain that $\tp _{\phi _1,\phi _2}(\fka )$
is equal to $\op ^w(\fkb )$, for some
$\fkb \in M^{\infty ,r_0}_{(\omega _0)}(\rr {2d})$.

\par

Since \eqref{Eq:AssignWeightFunc} gives
\eqref{Eq:WeightPseudoRel}$'$ in the Weyl
case $A=\frac 12\cdot I$,
the asserted continuity now follows from
Theorem \ref{Thm:PseudoCont2}. Moreover,
Theorem \ref{Thm:PseudoCont2}, 
\eqref{Eq:OpSpecNormClass}, and 
\eqref{Eq:ToeplSymbWindEstClass}
give
\begin{align*}
    \NM {\operatorname{Tp}_{\phi _1,\phi _2}
    (\fka) f}{M(\omega _2,\mascB )}
    &=
    \NM {\op ^w(\fkb) f}{M(\omega _2,\mascB )
    }
    \le
    C\nm {\fkb}{M^{\infty ,r_0}
    _{(\omega _0)}}
    \nm f {M(\omega _1,\mascB )}
    \\[1ex]
    &\le
    C \nm {\fka}{M ^{\infty, q} _{(\omega)}}
    \nm {\phi_1}{M ^{r} _{(\vartheta _1)}}
    \nm {\phi_2}{M ^{r} _{(\vartheta _2)}}
    \nm f{M(\omega _1,\mascB )},
\end{align*}
where the constant $C>0$ only depends on
$\mascB$ and $\omega, \omega _j$, $j=1,2$.
This gives the claim.
\end{proof}

\par

\begin{rem}
We observe that Theorems 
\ref{Thm:ToeplPseudoSymbClassMod} and
\ref{Thm:GenContResForToepl} generalize 
Proposition 1.34$'$
in \cite{AbCoTo} in several ways.
\end{rem}

\par

\section{Lifting properties
for modulation spaces}
\label{sec8}

\par

In this section we extend the lifting
properties given in \cite{GroTof1,AbCoTo}
to the extended class of modulation
spaces under consideration. 
In the first part we perform a general
review and explanations concerning
lifting properties on modulation spaces, given
in \cite{AbCoTo}. Thereafter
we deduce the extensions of these properties
to our extended class of modulation spaces.

\par

We recall from \cite{GroTof1,AbCoTo}
that the topological vector spaces $V_1$ and $V_2$ 
are said to possess
lifting property if there exists a "convenient" 
homeomorphisms
(a lifting) between them. For example, let
$p\in [1,\infty]$, $s\in \mathbf R$, and
set
$$
\eabs x=(1+|x|^2)^{\frac 12},
\qquad
x\in \rr d.
$$
Then the mappings
\begin{equation}
\label{Eq:SimpleLiftings}
f\mapsto \eabs \cdo ^s \cdot f
\quad \text{and}\quad
f\mapsto \eabs \Delta ^sf
\end{equation}
are homeomorphic from $L^p_s(\rr d)$ 
and
the Sobolev space $H^p_s(\rr d)$, respectively,
into $L^p(\rr d)=H^p_0(\rr d)$,
with inverses $f\mapsto \eabs \cdo ^{-s}\cdot f$ 
and $f\mapsto \eabs \Delta ^{-s}f$,
respectively. Here $L^p_s=L^p_{(\omega )}$
when $\omega =\eabs \cdo ^s$.

\par

By slightly more cumbersome arguments, one
deduces similar mapping properties for 
classical modulation spaces. That is,
suppose that $p,q\in [1,\infty ]$, 
$\omega _1(x,\xi )=\eabs x^s$,
and $\omega _2(x,\xi )=\eabs \xi ^s$. Then
the mappings in \eqref{Eq:SimpleLiftings}
are homeomorphisms from
$M^{p,q}_{(\omega _1)}(\rr d)$
and
$M^{p,q}_{(\omega _2)}(\rr d)$, respectively,
into $M^{p,q}(\rr d)$. (See e.{\,}g.
\cite[Theorem 2.2]{Toft04}.)

\par

We observe that the symbols for the
operators in \eqref{Eq:SimpleLiftings}
are given by $\eabs x^s$ and $\eabs \xi ^s$,
and thereby only depends purely on
the configuration variable $x$ or the
momentum variable $\xi$.
The situation becomes drastically more
complicated if the involved operators
depend on both $x$ and $\xi$ variables,
because of (heavy) interactions between
multiplications and differentiations.

\par

Nevertheless, in \cite{AbCoTo,GroTof1} one 
reach lifting results for modulation
spaces, with more general
pseudo-differential
operators and Toeplitz operators,
using delicate combinations of various
methods in time-frequency analysis, micro-local
analysis and spectral theory. 

\par

In order to explain these results, we need some
preparations.

\par

\subsection{Classes of smooth
symbols}

\par

Next we recall from \cite{AbCoTo,CaTo}
the definition of
spaces of smooth functions
which in several ways are suitable
as symbol classes for pseudo-differential
operators.

\par

\begin{defn}\label{Def:SymbClExp}
Let $s\ge 0$ and $\omega \in \mascP _E(\rr {d})$. 
\begin{enumerate}
\item The class $S^{(\omega )}(\rr {d})$
consists of all $f\in C^\infty (\rr {d})$
such that
$
\nm {\partial ^\alpha f}
{L^\infty _{(1/\omega )}}
$
is finite, for every $\alpha \in \nn d$.

\vrum

\item The class $\Gamma ^{(\omega )}_s(\rr {d})$
($\Gamma ^{(\omega )}_{0,s}(\rr {d})$) consists of all $f\in C^\infty (\rr {d})$ such that
$$
\sup _{\alpha \in \nn d}
\left (
\frac {\nm {\partial ^\alpha f}
{L^\infty _{(1/\omega )}}}
{h^{|\alpha |}\alpha !^s}
\right )
$$
is finite, for \emph{some} $h>0$
(for \emph{every} $h>0$).
\end{enumerate}
\end{defn}

\par

Evidently, we have
$$
\Gamma ^{(\omega )}_{0,s}(\rr {d})
\subsetneq
\Gamma ^{(\omega )}_{s}(\rr {d})
\subsetneq
S^{(\omega )}(\rr {d}),
\quad \text{when}\quad
\omega \in \mascP _E(\rr d).
$$
On the other hand, one usually
impose that $\omega$ should belong
to $\mascP _{s}(\rr d)$,
$\mascP _{s}^0(\rr d)$ and
$\mascP (\rr d)$ when considering
$\Gamma ^{(\omega )}_{0,s}(\rr {d})$,
$\Gamma ^{(\omega )}_s(\rr {d})$
and
$S^{(\omega )}(\rr {d})$, respectively.
It follows that the families
\begin{alignat}{3}
&\Gamma ^{(\omega )}_{0,s}(\rr {d})&
\quad &\text{when} & \quad
\omega &\in \mascP _{s}(\rr d),
\label{Eq:SmoothSymbClasses1}
\\[1ex]
&\Gamma ^{(\omega )}_{s}(\rr {d}) &
\quad &\text{when} & \quad
\omega &\in \mascP _{s}^0(\rr d),
\label{Eq:SmoothSymbClasses2}
\\[1ex]
\text{and}\quad
&S^{(\omega )}(\rr {d}) &
\quad &\text{when} & \quad
\omega &\in \mascP (\rr d),
\label{Eq:SmoothSymbClasses3}
\end{alignat}
do not cover each others. This follows
from Proposition 
\ref{Prop:SmoothSymbModSp} below in
combination with estimates of coefficients 
of the Gabor expansions for the involved
functions (see e.{\,}g. \cite{Gro2}).

\par

The following result shows that
the spaces in
\eqref{Eq:SmoothSymbClasses1}--\eqref{Eq:SmoothSymbClasses3}
are obtained
as unions and intersections of
suitable modulation spaces. We omit the
proof since the result follows by
combining (6.16) in \cite{Toft10} 
with Propositions 1.26 and 1.27
in \cite{AbCoTo}, and the fact that
$$
M^{p,q_1}_{(\omega )}(\rr d)
\subseteq
M^{p,q_2}_{(\omega )}(\rr d)
\subseteq
M^{p,q_1}_{(\omega _0)}(\rr d),
$$
when
$$
q_1\le q_2,
\quad \text{and}\quad
\omega _0(x,\xi )
\lesssim
\omega (x,\xi )\eabs {\xi}^{-s},
\quad
s>d\left (
\frac 1{q_1}-\frac 1{q_2}
\right ).
$$

\par

\begin{prop}
\label{Prop:SmoothSymbModSp}
Let $s\ge 1$, $q\in (0,\infty ]$,
$\omega _0\in \mascP (\rr d)$,
$\omega _1\in \mascP _{s}(\rr d)$
and
$\omega _2\in \mascP _{s}^0(\rr d)$.
Also let
\begin{multline*}
\omega _{0,r}(x,\xi )
=
\omega _0(x)\eabs \xi ^{-r},
\quad \text{and}\quad
\omega _{j,r}(x,\xi )
=
\omega _j(x)e^{-r|\xi |^{\frac 1{s}}},
\\[1ex]
j=1,2,\ r>0,\ x,\xi \in \rr d. 
\end{multline*}
Then
\begin{align*}
S^{(\omega )}(\rr {d})=\bigcap _{r>0}
M^{\infty ,q}_{(1/\omega _{0,r})}(\rr d),
\quad
\Gamma _{0,s}^{(\omega _1)}(\rr d)
&=
\bigcap _{r>0}
M^{\infty ,q}_{(1/\omega _{1,r})}(\rr d)
\intertext{and}
\Gamma _{s}^{(\omega _2)}(\rr d)
&=
\bigcup _{r>0}
M^{\infty ,q}_{(1/\omega _{2,r})}(\rr d).
\end{align*}
\end{prop}

\par

The next proposition shows that
pseudo-differential operators
with symbols in the spaces
in
\eqref{Eq:SmoothSymbClasses1}--\eqref{Eq:SmoothSymbClasses3}
possess convenient mapping
properties on Gelfand-Shilov spaces.

\par

\begin{prop}
\label{Prop:PseudoContSmoothSymb}
Let $s\ge 1$ and $A\in \GL (d,\mathbf R)$.
\begin{enumerate}
\item If $\omega \in \mascP (\rr {2d})$
and $\fka \in S^{(\omega )}(\rr {2d})$,
then $\op _A(\fka )$ is continuous on
$\mascS (\rr d)$, and on $\mascS '(\rr d)$.

\vrum

\item If $\omega \in \mascP _s^0(\rr {2d})$
and $\fka \in \Gamma _s^{(\omega )}(\rr {2d})$,
then $\op _A(\fka )$ is continuous on
$\maclS _s (\rr d)$, and on $\maclS _s'(\rr d)$.

\vrum

\item If $\omega \in \mascP _s(\rr {2d})$
and $\fka \in \Gamma _{0,s}^{(\omega )}(\rr {2d})$,
then $\op _A(\fka )$ is continuous on
$\Sigma _s (\rr d)$, and on $\Sigma _s'(\rr d)$.
\end{enumerate}
\end{prop}

\par

The assertion (1) in Proposition
\ref{Prop:PseudoContSmoothSymb} follows
by choosing $g$ in
\cite[Theorem 18.6.2]{Ho1}
as the constant Euclidean metric on
$W=\rr {2d}$. The assertions (2) and
(3) in Theorem \ref{Prop:PseudoContSmoothSymb}
follows from Theorems 4.10 and 4.11 in
\cite{CaTo}.

\par

Due to Proposition
\ref{Prop:BasicEmbModSp2Schw}--\ref{Prop:BasicEmbModSp2Beurl},
the following
result is in some sense an extension
of the previous proposition. Here the
involved weight functions should satisfy
\begin{equation}
\label{Eq:WeightCondPsedoMap}
\omega (x,\xi )
\lesssim
\frac {\omega _1(x,\xi )}{\omega _2(x,\xi )}.
\end{equation}

\par

\begin{prop}
\label{Prop:PseudoContModSmoothSymb}
Let $s\ge 1$, $A\in \GL (d,\mathbf R)$
and $\mascB$ be a normal QBF space on $\rr {2d}$
with respect to $r_0\in (0,1]$ and
$v_0\in \mascP _E(\rr {2d})$.
\begin{enumerate}
\item If
$\omega ,\omega _j\in \mascP (\rr {2d})$
satisfy \eqref{Eq:WeightCondPsedoMap},
$\fka \in S^{(\omega )}(\rr {2d})$ and,
more restrictive, $v_0\in \mascP (\rr {2d})$,
then $\op _A(\fka )$ is continuous from
$M(\omega _1,\mascB )$ to $M(\omega _2,\mascB )$.

\vrum

\item If
$\omega ,\omega _j\in \mascP _s^0(\rr {2d})$
satisfy \eqref{Eq:WeightCondPsedoMap},
$\fka \in \Gamma _s^{(\omega )}(\rr {2d})$ and,
more restrictive, $v_0\in \mascP _s^0(\rr {2d})$,
then $\op _A(\fka )$ is continuous from
$M(\omega _1,\mascB )$ to $M(\omega _2,\mascB )$.

\vrum

\item If
$\omega ,\omega _j\in \mascP _s(\rr {2d})$
satisfy \eqref{Eq:WeightCondPsedoMap},
$\fka \in \Gamma _{0,s}^{(\omega )}(\rr {2d})$ and,
more restrictive, $v_0\in \mascP _s(\rr {2d})$,
then $\op _A(\fka )$ is continuous from
$M(\omega _1,\mascB )$ to $M(\omega _2,\mascB )$.
\end{enumerate}
\end{prop}

\par

\begin{proof}
We only prove (2). The assertions (1)
and (3) follow by similar arguments
and are left for the reader.

\par

Since all involved weights are in
$\mascP _s^0(\rr {2d})$, it follows
from the definitions that for some
$v_1,v_2,v_3\in \mascP _s^0(\rr {2d})$, we have
\begin{align*}
\frac 1{\omega (x,\xi )}
&\lesssim
\frac 1{\omega (x,\eta )}\cdot v_1(0,\xi -\eta )
\\[1ex]
&\lesssim
\frac 1{\omega (x+A(y-x),\eta +A^*(\xi -\eta))}
\cdot v_2(y-x,\xi -\eta ),
\intertext{and}
\frac {\omega _2(x,\xi)}{\omega _1(x,\xi )}
&\lesssim
\frac {\omega _2(x,\xi)}{\omega _1(y,\eta )}
\cdot v_3(y-x,\xi -\eta ).
\end{align*}
A combination of these inequalities
gives \eqref{Eq:WeightPseudoRel}$'$
with $\omega _0\in \mascP _s^0(\rr {4d})$
given by
$$
\omega _0(x,\xi ,\eta ,y)
=
\frac {v_0(-y,\eta )v_2(y,\eta )v_3(y,\eta)}
{\omega (x,\xi )}.
$$
The result now follows by combining
Theorem \ref{Thm:PseudoCont2} with
Proposition \ref{Prop:SmoothSymbModSp}.
\end{proof}

\par

\subsection{Isomorphisms
of pseudo-differential operators}

\par

Several isomorphisms and lifting properties
rely on existence of
pairs of pseudo-differential
operators, which are inverses to each others
and with symbols in suitable
$\Gamma ^{(\omega )}_{s}$ classes. This
is, in the Roumieu case guaranteed by
the following result. (See Remark
\ref{Rem:IsomPseudoOtherCases}
below for the other cases.)

\par

\begin{thm}\label{Thm:Identification}
Let $s\ge 1$, $\mascB$ be a normal QBF space on
$\rr {2d}$ with respect to $r_0\in (0,1]$ and
$v_0\in \mascP _s^0(\rr {2d})$,
$A\in \GL (d,\mathbf R)$, and
$\omega ,\omega _0\in \mathscr P_s^0(\rr {2d})$. Then
there exist
$\fka \in \Gamma ^{(\omega _0)}_{s}(\rr {2d})$
and
$\fkb \in \Gamma ^{(1/\omega _0)}_{s}(\rr {2d})$
such that
\begin{equation}\label{Eq:abInverse1}
\op _A(\fka )\circ \op _A(\fkb )
=
\op _A(\fkb)\circ \op _A(\fka )
=
\operatorname {Id}_{\maclS _s'(\rr d)}.
\end{equation}
Furthermore, $\op _A(\fka )$ is a
homeomorphism from $M(\omega ,\mathscr B)$
onto $M(\omega /\omega _0,\mathscr B )$, for
every $\omega _0\in \mathscr P_s^0(\rr {2d})$.
\end{thm}

\par

\begin{proof}
The existence of
$\fka \in \Gamma ^{(\omega _0)}_s(\rr {2d})$ 
and
$\fkb \in \Gamma ^{(1/\omega _0)}_s(\rr {2d})$
such that \eqref{Eq:abInverse1} holds, follows
from \cite[Theorem 4.1]{AbCoTo}. The asserted
homeomorphism properties now follow by combining
Proposition \ref{Prop:BasicEmbModSp2Roum},
Proposition \ref{Prop:PseudoContModSmoothSymb}
and \eqref{Eq:abInverse1}.
%
%
\end{proof}


\par

If $\fka ,\fkb \in \Sigma _1'(\rr {2d})$
are such that $\op _A(\fka )\circ \op _A(\fkb )$
is well-defined as continuous map from
$\Sigma _1(\rr d)$ to $\Sigma _1'(\rr d)$,
then the $A$-twisted product $\fka \wpr _A \fkb$
of $\fka$ and $\fkb$ is defined by the formula
$$
\op _A(\fka \wpr _A \fkb )
=
\op _A(\fka )\circ \op _A(\fkb ).
$$
It follows that $\fka \wpr _A \fkb$
is uniquely defined as an element
in $\Sigma _1'(\rr {2d})$, in view of
Remark \ref{Rem:BijKernelsOps}.
Since the Weyl case is peculiar
interesting, we put $\wpr _A=\wpr$,
when $A=\frac 12\cdot I$

\par

From the fact that
the symbol for the identity operator
equals $1$, it follows that
\eqref{Eq:abInverse1} is equivalent to
\begin{equation}\tag*{(\ref{Eq:abInverse1})$'$}
\fka \wpr _A \fkb 
=
\fkb \wpr _A \fka 
=
1.
\end{equation}

\par

We also have the following  extension
of the second parts of Theorems 4.1 and 4.2
in \cite{AbCoTo}.

\par

\begin{prop}
\label{Prop:PseudoIsomEquiv}
Let $s\ge 1$,
$A\in \GL (d,\mathbf R)$,
$\omega
\in \mathscr P_s^0(\rr {2d})$, and
$\fka \in \Gamma ^{(\omega )}_s(\rr {2d})$.
Then the
following conditions are equivalent:
\begin{enumerate}
\item $\op _A(\fka )$ is an 
isomorphism from $M^2_{(\omega _1)}(\rr d)$ to $M^2_{(\omega
_1/\omega )} (\rr d)$ for some $\omega _1\in \mathscr P_{s}^0(\rr
{2d})$;

\vrum

\item $\op _A(\fka )$
is an isomorphism from $M(\omega
_2 ,\mathscr B)$ to $M(\omega _2/\omega ,\mathscr B)$ for
every $\omega _2\in \mathscr P_{s}^0(\rr {2d})$
and normal QBF space $\mascB$ on $\rr {2d}$
with respect to $r_0\in (0,1]$ and
$v_0\in \mascP _s^0(\rr {2d})$
\end{enumerate}
If {\rm{(1)}} or {\rm{(2)}} hold, then
the inverse of $\op _A(\fka )$
is given by $\op _A(\fkb )$ for some
$\fkb \in \Gamma ^{(1/\omega )}_{s}(\rr {2d})$,
and \eqref{Eq:abInverse1} holds.
%
%
\end{prop}

\par

\begin{proof}
Evidently, (2) implies (1). Suppose that
(1) holds. Then \cite[Theorem 4.1]{AbCoTo}
shows that the inverse is given by
$\op _A(\fkb)$, for some
$\fkb \in \Gamma ^{(1/\omega )}_{s}(\rr {2d})$,
and \eqref{Eq:abInverse1} holds.
The assertion (2) now follows from
Proposition \ref{Prop:PseudoContModSmoothSymb},
giving the result.
\end{proof}

\par

\begin{rem}
\label{Rem:IsomPseudoOtherCases}
For $s>1$, Theorem \ref{Thm:Identification}
and Proposition \ref{Prop:PseudoIsomEquiv}
hold true also with $\mascP _s$,
$\Gamma ^{(\omega )}_{0,s}$,
$\Gamma ^{(1/\omega )}_{0,s}$ and
$\Sigma _s'$, or with
$\mascP$,
$S^{(\omega )}$,
$S^{(1/\omega )}$ and
$\mascS '$,
in place of $\mascP _s^0$,
$\Gamma ^{(\omega )}_{s}$,
$\Gamma ^{(1/\omega )}_{s}$ and
$\maclS _s'$, respectively, at each
occurrence.
These cases follow by similar arguments,
using \cite[Theorem 4.2]{AbCoTo} and
Proposition \ref{Prop:BasicEmbModSp2Beurl},
or \cite[Theorem 2.1]{GroTof1} and
Proposition \ref{Prop:BasicEmbModSp2Schw},
in place of \cite[Theorem 4.1]{AbCoTo}
and Proposition \ref{Prop:BasicEmbModSp2Roum}.
The details are left for the reader.
\end{rem}

\par

\subsection{Liftings with Toeplitz operators}

\par

Our next goal is to
extend Theorems 0.1$'$, 5.1$'$ and 5.2$'$
in \cite{AbCoTo}. (See Theorems
\ref{Thm:ToeplLift1} and \ref{Thm:ToeplLift2}
below.)

\par

For convenience, we put
$$
\splM ^{p,q}_{(\omega )}(\rr {2d})
=
M^{p,q}_{(\omega _0)}(\rr {2d})
\quad \text{when}\quad
\omega (x,\xi ,y,\eta )
=
\omega _0(x,\xi ,-2\eta ,2y),
$$
and note that in some of our lifting results,
involved weights satisfy conditions
of the form
\begin{equation}\label{Eq:omega0t}
\omega _{0,t}(X,Y)=v_1(2Y)^{t-1}\omega _0(X) 
\qquad \text{ for } X,Y
\in \rr {2d}, \ t\in [0,1]
\end{equation}
(see \cite{AbCoTo}).

\begin{rem}
The space
$\splM ^{p,q}_{(\omega )}(\rr {2d})$ 
is usually called 
symplectic modulation space, which
appear when using the symplectic
short-time Fourier transform in place
of (ordinary) short-time Fourier transform
in the definition of the modulation space
norm.
We refer
to \cite[Section 3]{Toft18}
and the references therein,
for details.
\end{rem}

\par


\begin{thm}
\label{Thm:ToeplLift1}
Let $s\ge 1$, $\omega ,\omega _0,v\in
\mascP _{s}^0(\rr {2d})$ be
such that $\omega _0$ is $v$-moderate,
and let $\mascB$ be a normal QBF space on
$\rr {2d}$ with respect to $v_0 \equiv 1$.
Also suppose that one of the following
conditions hold:
\begin{enumerate}
\item $\phi \in M^2_{(v)}(\rr d)\setminus 0$, and
that in addition $\omega _0
\in \Gamma ^{(\omega _0)}_{s}(\rr {2d})$;

\vrum

\item $\phi \in \maclS _s(\rr d)\setminus 0$.
\end{enumerate}
Then $\tp _\phi (\omega _0)$ is an
isomorphism from $M(\omega ,\mathscr B)$ to 
$M(\omega /\omega _0,\mathscr B)$.
\end{thm}

\par

\begin{thm}
\label{Thm:ToeplLift2}
Let $s> 1$, $0\le t\le 1$,
$q_0\in [1,\infty]$,
$\omega ,\omega _0,v,v_1\in
\mascP _{s}(\rr {2d})$ be such that
$\omega _0$ is
$v$-moderate and  $\omega $ is
$v_1$-moderate, and let
$\mascB$ be an invariant BF space
with respect to $v_0 \equiv 1$. Set
$$
r_0=\frac {2q_0}{(2q_0-1)},\quad
\vartheta =\omega _0^{1/2},
$$
and
let $\omega _{0,t}$ be the same as in 
\eqref{Eq:omega0t}.
If  $\phi \in M^{r_0}_{(v_1^tv)}(\rr d)$
and $\omega _0\in \splM ^{\infty ,q_0}
_{(1/\omega _{0,t})}(\rr {2d})$,  then
$\tp _\phi (\omega _0)$ is
an isomorphism from $M(\vartheta
\omega ,\mascB )$ to
$M(\omega /\vartheta,\mascB )$.
\end{thm}

\par

We observe that Theorem
\ref{Thm:ToeplLift2Intro}
follows by choosing
$q_0=\infty$
and $t=1$ in Theorem
\ref{Thm:ToeplLift2}.

\par

\begin{rem}
If we choose $\mascB$ to be a mixed
norm space of Lebesgue type in
Theorem \ref{Thm:ToeplLift2},
we obtain
\cite[Theorem 5.2$'$]{AbCoTo}
as a special case. Hence
Theorem \ref{Thm:ToeplLift2},
is more general than
\cite[Theorem 5.2$'$]{AbCoTo}.

\par

At first glance, it might seem that
Theorems \ref{Thm:Identification}
is a pure restatement of
\cite[Theorem 4.1]{AbCoTo}, and that
Theorem \ref{Thm:ToeplLift1}
is a pure unification of
Theorems 0.1$'$ and Theorem 5.1$'$ in
\cite{AbCoTo}, since they are almost
literary
the same. However, this is an illusion,
because the definition of
\emph{normal QBF space}
in \cite{AbCoTo} is more restrictive
compared to Theorems
\ref{Thm:Identification} and
\ref{Thm:ToeplLift1}.
In fact, in
\cite{AbCoTo}, the only normal QBF spaces
which fails to be Banach spaces, are
suitable mixed norm spaces of Lebesgue
types, while in Theorems
\ref{Thm:Identification} and
\ref{Thm:ToeplLift1}, normal QBF spaces
are of such general kinds, given in
Definition \ref{Def:BFSpaces}. For example,
in Theorems \ref{Thm:Identification}
and \ref{Thm:ToeplLift1}, we may choose
$\mascB$ such that the $M(\omega ,\mascB )$
is any kind of quasi-Banach Orlicz modulation
space, given in e.{\,}g. \cite{TofUst}.
\end{rem}

\par

We need some preparations for the proofs.
The following lemma is a restatement of
\cite[Lemma 5.4]{AbCoTo}. The proof is therefore
omitted.

\par

\begin{lemma}
\label{Lem:BijectionHilbCase}
Let $s\ge 1$,  $\omega _0,v\in
\mascP _{s}(\rr {2d})$ be such
that $\vartheta = \omega _0^{1/2}$
is $v$-moderate. Assume that
$\phi \in M^2_{(v)}$. Then
$\tp _\phi (\omega _0)$ is an
isomorphism from
$M^2_{(\vartheta )}(\rr d)$
onto $M^2_{(1/\vartheta )}(\rr d)$.
\end{lemma}

\par

The proof of the next result
is omitted, because it follows from
\cite[Proposition 5.6]{AbCoTo}
and its proof.

\par

\begin{lemma}
\label{Lem:Aomegaproperties}
Suppose $s\ge 1$, $\omega _0\in \mascP _{s}^0(\rr {2d})$ is such that
$\omega _0\in \Gamma ^{(\omega _0)}_{s}(\rr {2d})$,
$v \in \mascP _{s}^0(\rr {2d})$ is submultiplicative,  and that
$\omega _0 ^{1/2}$ is $v$-moderate.
If $\phi _1,\phi _2\in M^2_{(v)}(\rr d)$,
then $\tp _{\phi _1,\phi _2}
(\omega _0) =\op ^w(\fkb )$ for some
$\fkb \in \Gamma ^{(\omega _0)}_s(\rr {2d})$.
\end{lemma}

\par

\begin{proof}[Proof of Theorem
\ref{Thm:ToeplLift1}]
The results follow by combining
Theorem \ref{Thm:Identification},
Proposition \ref{Prop:PseudoIsomEquiv},
Lemma \ref{Lem:BijectionHilbCase}
and Lemma \ref{Lem:Aomegaproperties}. 
Note for example that if (2) is fulfilled,
then the Weyl symbol for
$\tp _\phi (\omega _0)$ belongs to
$\Gamma ^{(\omega _0)}_{s}(\rr {2d})$,
in view of \eqref{Eq:ToeplWeyl}.
\end{proof}


\par

We need some further preparations
for the proof of Theorem \ref{Thm:ToeplLift2}.
First we have the following result concerning 
mapping  properties of
modulation spaces under the Weyl product.
We omit the proof because the result
is a special case of
\cite[Theorem 2.1]{ChToWa}.

\par

\begin{prop}\label{Prop:Weylprodmod}
Let $r\in (0,1]$ and
$\omega _0,\omega _1,\omega _2
\in \mathscr P_{E}(\rr {2d}
\oplus \rr {2d})$
be such that
\begin{equation}\label{Eq:weightprodmod}
\omega _0(X,Y)
\le
C\omega _1(X-Y+Z,Z)\omega _2(X+Z,Y-Z),
\end{equation}
for some constant $C>0$ which is
independent of
$X,Y,Z\in \rr {2d}$.
Then the map $(\fka ,\fkb )
\mapsto \fka \wpr \fkb$ from
$\Sigma _1(\rr {2d}) \times \Sigma _1(\rr {2d})$
to $\Sigma _1(\rr {2d})$ extends uniquely to a
continuous mapping from $\splM
^{\infty ,r}_{(\omega _1)}(\rr {2d})\times
\splM ^{\infty
,r}_{(\omega _2)}(\rr {2d})$ to
$\splM ^{\infty ,r}_{(\omega
_0)}(\rr {2d})$. 
\end{prop}

\par

The next lemma follows from
\cite[Lemma 5.7]{AbCoTo} and its proof.
The proof is therefore omitted.

\par

\begin{lemma}
\label{Lemma:CorWeyl}
Let $s\ge 1$, $r\in (0,1]$,
$\omega _0, v, v_1 \in \mathscr{P}_{s}^0
(\rr {2d}\oplus \rr {2d})$ be such that
$\omega _0$ is $v$-moderate.
Set $\vartheta = \omega _0^{1/2}$, and  
\begin{align}
\omega _1(X,Y)&=\frac{v(2Y)^{1/2}v_1(2Y)}
{\vartheta (X+Y) \vartheta(X-Y)},
\label{Eq:Tomega}
\\[1ex]
\omega _2(X,Y) &=\vartheta (X-Y)\vartheta(X+Y)v_1(2Y),
\notag
\\[1ex]
v_2(X,Y) &= v_1(2Y), 
\quad X,Y \in \rr {2d}.
\end{align}
Then 
\begin{align}
\Gamma ^{(1/\vartheta )}_{s}
\wpr
\splM ^{\infty ,r} _{(\omega _1)}
\wpr
\Gamma ^{(1/\vartheta )}_{s}
&\subseteq
\splM ^{\infty , 1} _{(v_2)},
\label{Eq:ch123}
\\[1ex]
\Gamma ^{(1/\vartheta )}_{s}
\wpr
\splM ^{\infty , r} _{(v_2)}
\wpr
\Gamma ^{(1/\vartheta )}_{s}
&\subseteq
\splM ^{\infty , r}
_{(\omega _2)}.
\label{Eq:ch124}
\end{align}

\par

The same holds true with $\mascP _{s}$ and
$\Gamma _{0,s}^{(1/\vartheta )}$
in place of $\mascP _{s}^0$ and
$\Gamma _{s}^{(1/\vartheta )}$
respectively, at each occurrence.
\end{lemma}

\par

We also need the following restatement
of \cite[Lemma 5.8]{AbCoTo}. The proof is 
omitted.

\par

\begin{lemma}\label{Lemma:CorWeyl2}
Let $s$, $\omega _1$, $\omega _2$ 
and $\vartheta$ be 
the same as in Lemma \ref{Lemma:CorWeyl}.
Also let $p,q\in [1,\infty ]$
and $\fkb \in \splM ^{\infty ,1}_{(\omega _1)}(\rr {2d})$. Then the
following is true:
\begin{enumerate}
\item $\op ^w(\fkb )$ is continuous from $M^{p,q}
_{(\vartheta )}(\rr d)$ to $M^{p,q} _{(1/\vartheta )}(\rr d)$;

\vrum

\item if in addition $\op ^w(\fkb )$ is an 
isomorphism from 
$M^{2} _{(\vartheta )}(\rr d)$ to $M^{2}
_{(1/\vartheta )}(\rr d)$, then its inverse $\op ^w(\fkb )^{-1}$
equals $\op ^w(\fkc )$ for some
$\fkc \in \splM ^{\infty ,1}
_{(\omega _2)}(\rr {2d})$.
\end{enumerate}
\end{lemma}

\par

\begin{proof}[Proof of Theorem 
\ref{Thm:ToeplLift2}]
We shall mainly follow the proof of
\cite[Theorem 5.2$'$]{AbCoTo}.

\par

Let $\phi \in M ^{r_0} _{(v_1^t v)}
(\rr d) \subseteq M^2_{(v)}(\rr d)$.
First we note that the Toeplitz
operator $\tp _\phi (\omega _0)$ is an 
isomorphism from
$M^2_{(\vartheta )}(\rr d)$ to 
$M^2_{(1/\vartheta )}(\rr d)$ 
in view of Lemma 
\ref{Lem:BijectionHilbCase}.
With $\omega _1$ defined in
\eqref{Eq:Tomega}, Remark
\ref{Rem:BijKernelsOps} and Theorem
\ref{Thm:ToeplPseudoSymbClassMod}
imply that there exist
$\fkb \in \splM ^{\infty ,1}_{(\omega 
_1)}(\rr {2d})$ and
$\fkc \in \maclS _s'(\rr {2d})$
such that
$$
\tp _\phi (\omega _0) = \op ^w(\fkb )\quad 
\text{and}\quad \tp _\phi (\omega
_0)^{-1}=\op ^w(\fkc ) \, .
$$

For $X,Y \in \rr {2d}$, let
\begin{equation}
\label{Eq:dmod1A}
\omega _2(X,Y)
=
\vartheta (X-Y)\vartheta (X+Y) v_1(2Y)
\quad \text{and}\quad  
\omega _3(X,Y) = \frac {\vartheta (X+Y)}
{\vartheta (X-Y)}.
\end{equation}

\par

By Theorem \ref{Thm:PseudoCont2}
and Lemma \ref{Lemma:CorWeyl2}
it follows that
$\fkc
\in
\splM ^{\infty ,1}_{(\omega _2)}(\rr {2d})$,
and that the mappings
\begin{align}
\op ^w(\fkb )\, &:\, 
M(\omega \vartheta ,\mascB )
\to
M(\omega /\vartheta ,\mascB )
\label{op(b)(d)}
\intertext{and}
\op ^w(\fkc )\, &:\,
M(\omega /\vartheta ,\mascB )
\to
M(\omega \vartheta ,\mascB )
\end{align}
are continuous.


\par

This gives
\begin{eqnarray*}
\lefteqn{\omega _1(X-Y+Z,Z)\omega _2(X+Z,Y-Z)}
\\[1ex]
&=& \Big (\frac {v(2Z)^{1/2}v_1(2Z)}
{\vartheta (X-Y+2Z)\vartheta
(X-Y)}\Big )
\cdot
\big (\vartheta (X-Y+2Z)\vartheta (X+Y)
v_1(2(Y-Z)) \big ) 
\\[1ex]
&= &  \frac {v(2Z)^{1/2}v_1(2Z)v_1(2(Y-Z))
\,  \vartheta (X+Y)}{\vartheta (X-Y)}
\\[1ex]
&\gtrsim &
\frac {\vartheta (X+Y)}
{\vartheta (X-Y)} =\omega _3(X,Y)\, 
\qquad X,Y,Z \in \rr {2d}.
\end{eqnarray*}
Therefore Proposition \ref{Prop:Weylprodmod}  
shows that $\fkb \wpr \fkc \in \splM
^{\infty ,1}_{(\omega _3)}$. Since
$\op ^w(\fkb )$ is an isomorphism from
$M^2_{(\vartheta )}$ to $M^2_{(1/\vartheta )}$ 
with inverse $\op ^w(\fkc )$, it follows that 
$\fkb \wpr \fkc =1$ and that the constant
symbol $1$ belongs to
$\splM ^{\infty ,1}_{(\omega _3)}$. By similar 
arguments it follows that
$\fkc \wpr \fkb =1$. Therefore the identity 
operator $\mathrm{Id}= \op ^w (\fkb )
\circ \op ^w(\fkc ) $ on
$M(\omega /\vartheta ,\mascB )$ factors through
$M(\omega \vartheta ,\mascB )$,
and thus $\op ^w(\fkb )=
\tp _\phi (\omega _0)$ is an isomorphism
from $M(\omega \vartheta ,\mascB )$ to 
$M(\omega /\vartheta ,\mascB )$
with inverse $\op ^w(\fkc )$.
This gives the result.
\end{proof}

\par

\appendix

\par

\section{Proofs of some results
from Section \ref{sec4}}
\label{App:A}

\par

In this appendix we give proofs of
Proposition \ref{Prop:ComplModSp}
and the compactness 
statements (2) and (4) of
Theorem \ref{thm:CompAndContProp}.

\par

The proof of Proposition \ref{Prop:ComplModSp}
is based on the following
lemma, concerning continuity of
the operator
\begin{equation}\label{Eq:DefTFProj}
P_{\phi _1,\phi _2}
\equiv 
\big ( \nm {\phi _1}{L^2}
\nm {\phi _2}{L^2} \big )^{-1}
V_{\phi _2}\circ V_{\phi _1}^*,
\end{equation}
when $\phi _1,\phi _2\in L^2(\rr d)
\setminus 0$.

\par

\begin{prop}\label{Prop:TFProjCont}
Let $\mascB$ be a normal
QBF space on $\rr {2d}$ of order
$r_0\in (0,1]$ and
$v_0\in \mascP _E(\rr {2d})$,
$\omega ,v\in \mascP _E(\rr {2d})$
be such that $\omega$ is $v$-moderate,
$\phi _1,\phi _2
\in M^{r_0}_{(v_0v)}(\rr d)\setminus 0$,
and let $r\in [1,\infty ]$. Also let
$P_{\phi _1,\phi _2}$ be as in
\eqref{Eq:DefTFProj}. Then
$$
P_{\phi _1,\phi _2}:
\Sigma _1(\rr {2d})\to \Sigma _1'(\rr {2d})
$$
is uniquely extendable to a continuous
map
$$
P_{\phi _1,\phi _2}:
\sfW ^r(\omega ,\mascB)
\to 
V_{\phi _2}(M(\omega ,\mascB))
\hookrightarrow
\sfW ^\infty (\omega ,\mascB).
$$
\end{prop}

\par

For the proof we observe that
$P_{\phi _1,\phi _2}$ in
\eqref{Eq:DefTFProj} can be written as
\begin{align}
P_{\phi _1,\phi _2}F
&=
\big ( \nm {\phi _1}{L^2}
\nm {\phi _2}{L^2} \big )^{-1}
V_{\phi _2}\phi _1 *_V F,
\label{Eq:FormulaTFProj}
\intertext{where $*_V$ is the
twisted convolution, given by}
(F*_VG)(X)
&\equiv
(2\pi )^{-\frac d2}
\int _{\rr {2d}}F(X-Y)G(Y)
e^{i\kappa _1(X,Y)}\, dY,
\label{Eq:TwistConv}
\\[1ex]
&=
(2\pi )^{-\frac d2}
\int _{\rr {2d}}F(Y)G(X-Y)
e^{i\kappa _2(X,Y)}\, dY,
\intertext{with}
\kappa _1(X,Y) &= \scal y{\eta -\xi}
\quad \text{and}\quad
\kappa _2(X,Y) = \scal {y-x}\eta ,
\\
X &=(x,\xi )\in \rr {2d},\quad 
Y=(y,\eta )\in \rr {2d}.
\notag
\end{align}
This follows by a straight-forward application
of Fourier's inversion formula (see e.{\,}g.
the proof of \cite[Proposition 11.3.2]{Gro2}).
Especially we observe that
\eqref{Eq:FormulaTFProj} implies
that
\begin{equation}\label{Eq:TFProjEst}
|(P_{\phi _1,\phi _2}F)(X)|
\le
\big ( \nm {\phi _1}{L^2}
\nm {\phi _2}{L^2} \big )^{-1}
\big (
|V_{\phi _2}\phi _1| * |F|
\big )(X).
\end{equation}

\par

\begin{proof}[Proof of
Proposition \ref{Prop:TFProjCont}]
Let $F\in \sfW ^r(\omega ,\mascB)$. Since
$V_{\phi _2}\phi _1
\in
\sfW ^1(v_0v,\ell ^{r_0}(\zz {2d}))$ in view
of Theorem \ref{Thm:EquivNorms2}, it follows
from \eqref{Eq:FormulaTFProj}, Corollary \ref{Cor:ConvWienerSp1}
that
$P_{\phi _1,\phi _2}F
\in \sfW ^r(\omega ,\mascB)$.
Let $f
\equiv
(\nm {\phi _1}{L^2}\nm {\phi _2}
{L^2})^{-1}
V_{\phi _1}^*F$. 
Then
$$
V_{\phi _2}f = P_{\phi _1,\phi _2}F
\in \sfW ^r(\omega ,\mascB).
$$
By Theorem \ref{Thm:EquivNorms2} it
follows that $V_{\phi _2}f
\in \sfW ^r(\omega ,\mascB)$
is the same as $f\in M(\omega ,\mascB )$,
and $V _{\phi_2} f
\in \sfW ^{\infty} (\omega, \mascB)$,
and the result follows.
\end{proof}

\par

\begin{proof}[Proof of Proposition
\ref{Prop:ComplModSp}]
That $M(\omega ,\mascB)$ is of the same order as $\mascB$ follows from Definition \ref{Def:GenModSpace}, so we proceed with proving completeness.

Let $\phi \in \Sigma _1(\rr d)$
be such that $\nm \phi{L^2}=1$,
and let $\{ f_j\} _{j=1}^\infty$
be a Cauchy sequence in
$M(\omega ,\mascB)$.
Then Theorem \ref{Thm:EquivNorms2}
implies that
$$
\lim _{j,k\to \infty}
\nm {V_\phi f_j-V_\phi f_k}
{\sfW ^\infty (\omega ,\mascB )}=0.
$$
Since $\sfW ^\infty (\omega ,\mascB )$
is complete, there exists a unique
$F\in \sfW ^\infty (\omega ,\mascB )$
such that
$$
\lim _{j\to \infty}
\nm {V_\phi f_j-F}
{\sfW ^\infty (\omega ,\mascB )}
=0.
$$
Let $f=V_\phi ^*F$. By Proposition
\ref{Prop:TFProjCont} it follows
that $f\in M(\omega ,\mascB)$.
Since $P_{\phi ,\phi}\circ V_\phi =V_\phi$,
another application of Proposition
\ref{Prop:TFProjCont} shows that
\begin{align*}
\nm {f_j-f}{M(\omega ,\mascB)}
&\asymp
\nm {V_\phi f_j-V_\phi f}
{\sfW ^\infty (\omega ,\mascB )}
=
\nm {V_\phi f_j-P_{\phi ,\phi} F}
{\sfW ^\infty (\omega ,\mascB )}
\\[1ex]
&=
\nm {P_{\phi ,\phi}(V_\phi f_j- F)}
{\sfW ^\infty (\omega ,\mascB )}
\lesssim
\nm {V_\phi f_j- F}
{\sfW ^\infty (\omega ,\mascB )}
\to 0,
\end{align*}
as $j\to \infty$. This shows that
$$
\lim _{j\to \infty}
\nm {f_j-f}
{M(\omega ,\mascB )}=0,
$$
and the asserted completeness
follows. This gives the result.
\end{proof}

\par


\par

For the proof of (2) and (4) in Theorem 
\ref{thm:CompAndContProp},
we need the following lemma. 
We omit the proof because the result 
follows from the proof of Lemma 3.10 in 
\cite{PfTo19}.

\par

\begin{lemma}\label{Lemma:UniformSeq}
Let
$\phi_0 (x)=\pi ^{-\frac d4}e^{-\frac 12\cdot |x|^2}$,
$x\in \rr d$, $\omega \in \mascP _E (\rr {2d})$ and
let $\{ f_j\} _{j=1}^\infty \subseteq \Sigma _1'(\rr d)$
be a bounded set in $M^{\infty}_{(\omega)}(\rr d)$. Then there is a
subsequence $\{ f_{j_k}\} _{k=1}^\infty$ of $\{ f_j\}
_{j=1}^\infty$ such that $\{ V_{\phi_0} f_{j_k}\} _{k=1}^\infty$
is locally uniformly convergent.
\end{lemma}

\begin{proof}[Proof of Theorem
\ref{thm:CompAndContProp} {\rm{(2)}} and
{\rm{(4)}}]
Let 
$
\phi (x)
=\pi ^{-\frac d4}
e^{-\frac 12\cdot |x|^2}
$,
$x\in \rr d$,
be the standard Gaussian.

\par 

To show (2) we let 
$
\{f_j\}_{j=1}^{\infty} 
\subseteq 
M(\omega _1,\mascB )
$
be bounded. We need to show that 
$\{f_j\}_{j=1}^{\infty}$ 
has a subsequence, which converges in 
$M(\omega _2,\mascB )$.
By the assumptions there is a sequence of increasing balls $B_k$,
$k\in \mathbf Z_+$,
centered at the origin with radius tending to $+\infty$ as $k\to
\infty$ such that
\begin{equation}\label{n.1}
\frac{\omega_2(x,\xi)}{\omega_1(x,\xi)}\le\frac{1}{k},\quad
\text{when}\quad (x,\xi )\in \rr {2d}\setminus B_k.
\end{equation}

\par

Proposition \ref{Prop:BasicEmbModSp1} gives the boundedness of 
$\{f_j\}_{j=1}^{\infty}$ in 
$M ^{\infty} _{(\omega _1 /v _0)} (\rr d)$.
Hence by Lemma \ref{Lemma:UniformSeq}
there is subsequence 
$\{ h_j\}_{j=1}^\infty$ of 
$\{ f_{j}\}_{j=1}^\infty$
such that $\{ V_\phi h_j\}_{j=1}^\infty$ 
converges uniformly on any $B_k$,
and converges on the whole $\rr {2d}$.

\par

 Let 
$\chi _k$ be the characteristic function
of $B_k$, $k\ge 1$. 
By a straight forward calculation we get 
$
\chi _k 
\in 
W ^{\infty}(\omega_2 v _0, \ell ^{r_0}
(\zz d))
\subseteq
\mascB _{(\omega_2)}
$,  
where the last inclusion follows 
form Corollary \ref{Cor:BasicEmbWSp}.
Due to \eqref{n.1} and from the fact that $\{h_j\} _{j=1}^{\infty}$ is a bounded sequence in $M(\omega _1, \mascB)$ we have
\begin{multline}\label{estimate}
\begin{aligned}
\| &h_{m_1} - h_{m_2} \| ^{r_0} _{M(\omega_2,\mascB )} = \nm
{V_\phi h_{m_1}-V_\phi h_{m_2}}{\mascB _{(\omega _2)}} ^{r_0}
\\[1ex]
&\le
\|(V_\phi h_{m_1}-V_\phi h_{m_2})\chi _{k}\|_{\mascB _{(\omega_2)}} ^{r_0}+
{\|V_\phi h_{m_1}-V_\phi h_{m_2}\|_{\mascB _{(\omega_1)}}^{r_0}} /{k ^{r_0}}
\\[1ex]
&\le
\|(V_\phi h_{m_1}-V_\phi h_{m_2})\chi _{k}\|_{\mascB _{(\omega_2)}} ^{r_0}+
{2C}/{k ^{r_0}}, 
\qquad m_1, m_2 \in \mathbb{N}.
\end{aligned}
\end{multline}

\par

In order to make the right-hand side arbitrarily small, we choose $k$ large enough. Note that $V_\phi h_1, V_\phi h_2,\dots$ is a sequence
of bounded continuous functions converging uniformly on the compact set
$\overline B_k$. Since $\omega_2$ is a weight and $\mascB$ is a normal QBF space we obtain that
\begin{align*}
    &\|\left( V_{\phi}h_{m_1}-V_{\phi}h_{m_2} \right) \chi_{k}
    \|_{\mascB_{(\omega_2)}} ^{r_0}
    =\|\left( V_{\phi}h_{m_1}-V_{\phi}h_{m_2} \right)\omega_2 \chi_{k}
    \|_{\mascB} ^{r_0}
    \\ 
    &\qquad\qquad \lesssim
    \left (
    \sup_{(x,\xi) \in B_k} |\left(
    V_{\phi}h_{m_1}(x,\xi)-V_{\phi}h_{m_2}(x,\xi) \right)
    \omega_2(x, \xi)|
    \right ) ^{r_0}
    \|\chi_{k}\|_{\mascB} ^{r_0}
    \\
    &\qquad \qquad \lesssim 
    \left(
    \sup_{(x,\xi) \in B_k}
    |V_{\phi}h_{m_1}(x,\xi)-V_{\phi}h_{m_2}(x,\xi)|
    \right)^{r_0}
\end{align*}
where $\|\left( V_{\phi}h_{m_1}-V_{\phi}h_{m_2} \right) \chi_{k}
    \|_{\mascB_{(\omega_2)}}$
    tends to zero as $m_1$ and $m_2$ tend to infinity which proves (2).

\medspace

To show (4) we let $v_0$ be bounded 
and assume that the embedding $i$ in
\eqref{Embedding} is compact. 
Due to (2) it remains to prove that 
$\omega _2/\omega _1$ turns to zero at infinity.
We intend to prove this claim by 
contradiction. 
Hence we suppose the existence of a sequence 
$(x_k,\xi _k) \in \rr {2d}$ with 
$|(x_k,\xi _k)| \rightarrow \infty$ if
$k \rightarrow \infty$ and a $C>0$ fulfilling
\begin{equation}\label{Eq:ContradIneq}
    \frac{\omega_2(x_k,\xi _k)}{\omega _1(x_k,\xi _k)}
    \geq C \qquad \text{for all } k \in \mathbf {N}.
\end{equation}
Then it follows by straight forward computations, c.f. the proof of Theorem 3.7. in \cite{PfTo19}, that 
$\{f_k \}_{k=1}^{\infty}$ defined by
\begin{equation} \label{eq:Def_fk}
   f_k
   = \frac 1 {\omega _1(X_k)} 
   e ^{i \scal \cdo {\xi _k} } 
   \phi(\cdo - x _k), 
   \qquad X_k=(x_k, \xi_k), 
   k \in \mathbb{N}.
\end{equation}
is bounded in $M(\omega _1, \mascB)$.
Since the embedding $i$ is compact, 
$\{f_k \}_{k=1}^{\infty}$ is precompact in $M(\omega _2, \mascB)$.

\par

From $V _{\phi} \phi \in \Sigma _1 (\rr d)$, \eqref{Eq:GSFtransfChar} 
and $\omega _1\gtrsim e^{-r_0|\cdo |}$ for some $r_0>0$,
c.f. \eqref{Eq:weight1}, 
we obtain 
$$
\int \phi(x)\overline{f_k(x)}\, dx =
\frac{1}{\omega_1(x_k,\xi _k)}(V_{\phi}\phi)(x_k,\xi _k)\to 0,
$$
as $k \to \infty$. This implies that $f_k$ tends to zero in
$\Sigma_1 '(\rr d)$.
 Hence the only possible limit point in
$M(\omega_2, \mascB)$ of $\{ f_k \} _{k=1}^\infty$ is zero.

\par

Since $\{ f_k \} _{k=1}^\infty$ is precompact in $M(\omega_2, \mascB)$,
there is a subsequence 
$\{ f_{k_j} \} _{j=1}^\infty$
which
converges to zero
in $M(\omega_2, \mascB)$.
Due to Proposition
\ref{Prop:BasicEmbModSp1} 
we have 
$M(\omega_2, \mathscr{B}) \hookrightarrow M^\infty_{(\omega_2/ v_0)}$
and consequently we get by the boundedness of $v_0$
\begin{equation}
\sup_{X\in\rr {2d}}\omega_2(X)|(V_\phi (f_{k_j}))(X)|\le
C\|f_{k_j}\|_{M(\omega_2, \mascB)}\to 0
\end{equation}
as $j\to \infty$.

\par

 Taking $X=X_{k_j}$ in the previous inequality we obtain
\begin{multline}
\frac{\omega_2(X_{k_j})}{\omega_1(X_{k_j})}
= 
(2\pi )^{\frac d2}\, \frac{\omega_2(X_{k_j})}{\omega_1(X_{k_j})}|(V_\phi \phi)(0)|
\\
=
(2\pi )^{\frac d2}\, \frac{\omega_2(X_{k_j})}{\omega_1(X_{k_j})}|(V_\phi (e^{i\langle \cdo
,\xi _{k_j}\rangle}\phi (\cdo -x_{k_j}))(X_{k_j})|
\\
=
(2\pi )^{\frac d2}\, \omega_2(X_{k_j})|(V_\phi (f_{k_j}))(X_{k_j})|
\to 0,
\end{multline}
which contradicts \eqref{Eq:ContradIneq} and proves (4).
\end{proof}

\par

\end{document}